\documentclass[review]{elsarticle}
\usepackage{lineno,hyperref}
\usepackage{amsmath}
\usepackage{amssymb}
\usepackage{amsfonts}
\usepackage{color}
\usepackage{soul}

\soulregister\cite7
\soulregister\ref7
\soulregister\eqref7

\newtheorem{theorem}{Theorem}[section]
\renewcommand{\theequation}{\thesection.\arabic{equation}}

\newtheorem{assumption}[theorem]{Assumption}

\newtheorem{lemma}[theorem]{Lemma}

\newtheorem{remark}[theorem]{Remark}

\numberwithin{equation}{section}
\newenvironment{proof}[1][Proof]{\noindent\textit{#1.}}{\hfill \rule{0.5em}{0.5em}}
\usepackage{paralist,graphics,epsfig,graphicx,epstopdf,mathrsfs}
\usepackage{geometry}
\journal{International Journal of Bifurcation and Chaos (IJBC)}

\begin{document}

\begin{frontmatter}

\title{Turing-Turing Bifurcation and Normal Form in a Predator-Prey Model with Predator-Taxis and Prey Refuge}
\author[1]{Yehu Lv\corref{mycorrespondingauthor}}
\ead{sms\_lvyh@ujn.edu.cn}
\address[1]{School of Mathematical Sciences, University of Jinan, Jinan 250022, Shandong Province, People's Republic of China}
\cortext[mycorrespondingauthor]{Corresponding author.}

\begin{abstract}
This paper investigates a predator-prey reaction-diffusion model incorporating predator-taxis and a prey refuge mechanism, subject to homogeneous Neumann boundary conditions. Our primary focus is the analysis of codimension-two Turing-Turing bifurcation and the calculation of its associated normal form for this model. Firstly, employing the maximum principle and Amann's theorem, we rigorously prove the local existence and uniqueness of classical solutions. Secondly, utilizing linear stability theory and bifurcation theory, we conduct a thorough analysis of the existence and stability properties of the positive constant steady state. Furthermore, we derive precise conditions under which the model undergoes a Turing-Turing bifurcation. Thirdly, by applying center manifold reduction and normal form theory, we derive the method for calculating the third-truncated normal form characterizing the dynamics near the Turing-Turing bifurcation point. Finally, we present numerical simulations to validate the theoretical findings, confirming the correctness of the analytical results concerning the bifurcation conditions and the derived normal form.
\end{abstract}


\end{frontmatter}

\section{Introduction}
\label{sec:1}

Traditional approaches for modeling prey refuge effects in predator-prey systems typically consider either a constant number or constant proportion of prey protected from predation \cite{lv1}. Following \cite{lv2}, we denote the refuge-occupying prey population as $u_{r}$. This parameterization manifests in two biologically distinct forms:

(i) Density-dependent refuge: protection scales with prey abundance, yielding $u_{r}=\beta u(t)$ where $\beta \in (0,1)$ represents the refuge proportion.

(ii) Fixed-capacity refuge: the sheltered prey number depends on environmental carrying capacity, giving $u_{r}=C$ (a positive constant).

\noindent Moreover, in the literature, various functional responses have been proposed to model prey refuge. Common mathematical representations include the following:

(i$^{\prime}$) A bilinear refuge term $muv$, where $m$ is the refuge coefficient, where $u$, $v$ denote prey and predator densities, respectively \cite{lv3, lv4}.

(ii$^{\prime}$) Saturating nonlinear refuge functions, such as $\frac{muv}{a+v}$, where $a$ is the half-saturation constant \cite{lv5, lv6}.

(iii$^{\prime}$) A ratio-dependent nonlinear refuge term $\frac{mu^{2}v}{a+u^{2}}$ \cite{lv7}.

(iv$^{\prime}$) A more general nonlinear form $\frac{mu^{2}v}{a+u^{2}v}$, which incorporates a composite saturation term involving both species densities \cite{lv8}.

\noindent These expressions reflect a spectrum of nonlinear interactions, from bilinear to saturating and ratio-dependent forms, illustrating how refuge protection can depend nonlinearly on species densities, often leading to richer ecological dynamics.

When refuge usage follows a constant ratio (Case (i)) as in \cite{lv9}, the resulting predator-prey dynamics reduce to an ordinary differential equation (ODE) model where predation acts only on the exposed prey population $u-u_{r}$, that is
\begin{equation*}
\left\{\begin{aligned}
&\frac{\mathrm{d}u}{\mathrm{d}t}=\left(r\left(1-\frac{u}{K}\right)-\frac{q(1-\beta)v}{(1-\beta)u+a}\right)u, \\
&\frac{\mathrm{d}v}{\mathrm{d}t}=b\left(\frac{p(1-\beta)u}{(1-\beta)u+a}-c\right)v,
\end{aligned}\right.
\end{equation*}
where $u=u(t)$ represents prey population density (measured in biomass or number density), $v=v(t)$ denotes predator population density. The model dynamics are governed by ecologically meaningful parameters, with all constants satisfying $r, K, q, a, b, p, c \in \mathbb{R}^{+}$.

In spatially explicit predator-prey models, directional movement mechanisms critically influence pattern formation: predators exhibit prey-taxis (directed movement toward prey density gradients), while prey display predator-taxis (directional movement opposite to predator gradients). The formal study of these mechanisms originated with Kareiva and Odell's seminal 1987 work \cite{lv10}, which quantified predator aggregation through non-random spatial distributions and established the first ecological model incorporating prey-taxis. Subsequent research has extensively modeled prey-taxis effects on spatiotemporal population dynamics (e.g., \cite{lv13, lv11, lv12}). Complementary to prey-taxis, predator-taxis represents prey's chemotactic avoidance behavior in response to predation risk. As experimentally validated by Wang and Zou \cite{lv14}, this anti-predator strategy actively suppresses spatial heterogeneity formation. Recent modeling advances have further explored predator-taxis implications (\cite{lv17, lv15, lv16}). Most recently, Lv \cite{lv18} extended this framework by examining Turing-Hopf bifurcation in delayed diffusive predator-prey model incorporating both taxis mechanism and fear effect.

The preceding analysis establishes that biologically realistic predator-prey models should incorporate two critical mechanisms: (1) prey's active avoidance movement (predator-taxis) to evade predation, and (2) refuge-mediated protection of prey subpopulations to prevent extinction. These ecologically mandated features maintain model persistence and pattern-forming capabilities. To address these requirements, we propose the following reaction-diffusion-taxis model with prey refuge, that is
\begin{equation}
\left\{\begin{aligned}
&\frac{\partial u}{\partial t}=d_{u}\Delta u+\nabla \cdot(\alpha u \nabla v)+u(r_{0}-au)-\frac{b_{1}(1-\beta)u}{b_{2}v+(1-\beta)u}v, & x \in \Omega, t>0, \\
&\frac{\partial v}{\partial t}=d_{v}\Delta v-m_{1}v+\frac{cb_{1}(1-\beta)u}{b_{2}v+(1-\beta)u}v, & x \in \Omega, t>0, \\
&\frac{\partial u}{\partial \nu}=\frac{\partial v}{\partial \nu}=0, & x \in \partial \Omega, t>0, \\
&u(x,0)=u_{0}(x), v(x,0)=v_{0}(x), & x \in \Omega,
\end{aligned}\right.
\end{equation}
where $u(x,t)$ and $v(x,t)$ denote prey and predator densities at position $x \in \Omega$ and time $t>0$, where $\Omega \subset \mathbb{R}^{n}$ ($n \geq 1$) is a smooth bounded domain with outward unit normal $\nu$ on $\partial\Omega$. $d_{u}$ and $d_{v}$ are the random diffusion coefficients of prey and predator, respectively, flux term $\nabla \cdot (\alpha u \nabla v)$ with sensitivity $\alpha>0$, modeling prey's directional evasion from predator gradients, $\frac{b_{1}u}{u+b_{2}v}$ is the ratio-dependent functional response (Arditi-Ginzburg type \cite{lv19}), $u_{0}(x)$ and $v_{0}(x)$ are continuous initial functions. Ecologically, $r_{0}$ is the intrinsic growth rate of the prey, and $a$ measures the strength of intra-specific competition among prey individuals. The predation process is described by a ratio-dependent functional response: $b_{1}$ is the maximum predation rate, and $b_{2}$ accounts for predator interference during foraging. The parameter $m_{1}$ is the natural mortality rate of predator, and $c$ is the conversion efficiency, quantifying how efficiently consumed prey are converted into predator biomass. All parameters $r_{0}, a, b_{1}, b_{2}, m_{1}, c>0$. The ratio-dependent response reflects empirical evidence where per capita predation depends on prey-to-predator abundance ratios, accounting for competition effects during foraging \cite{lv20,lv21,lv22}.

This paper focuses on the Turing-Turing bifurcation, a codimension-two spatial instability occurring at the intersection of two Turing bifurcation curves. This phenomenon drives the formation of multi-scale spatial patterns through predator-taxis and prey refuge interactions, where prey movement generates localized ecological structures including refuge hotspots and predator aggregation zones. Such pattern multi-stability and superposition explain complex spatial configurations observed in natural models, from marine environments (e.g., zooplankton-phytoplankton dynamics mediated by nutrient gradients) to terrestrial ecosystems with refuge heterogeneity. The study of codimension-two Turing-Turing bifurcations has become a pivotal avenue for understanding complex, multi-scale pattern formation across biological systems. In ecology, the role of organism movement is paramount, as demonstrated by Xing et al. \cite{lv23}, who showed that predator-taxis can induce both Turing and Turing-Turing bifurcations, underscoring its critical role in generating spatiotemporal diversity. This builds upon a rich theoretical foundation established in developmental biology, where models like the Gierer-Meinhardt model, analyzed by Zhao et al. \cite{lv24}, have long utilized Turing-Turing bifurcations to explain multi-stable patterns in processes like vascular self-organization. Subsequent research has expanded this framework to various contexts: analyzed alongside Turing-Hopf interactions in generalized Brusselator model \cite{lv25}, the normal form has been derived for the diffusive Bazykin system with prey-taxis \cite{lv26}, and explored in activator–inhibitor system with gene expression delay \cite{lv27}. Methodologically, the derivation of normal form formulae has been extended to systems with nonlinear diffusion \cite{lv28}. However, despite these significant advances, a comprehensive analysis that derives the explicit normal form for a Turing-Turing bifurcation in a predator-prey model incorporating both predator-taxis and a prey refuge mechanism remains an open challenge. This work is dedicated to filling this specific gap.

Computational methodologies for codimension-two bifurcations are well-established for several classical scenarios: Turing-Hopf \cite{lv29, lv30, lv31, lv32, lv33, lv34, lv35, lv36, lv37, lv38, lv39}, Hopf-Hopf \cite{lv40, lv41, lv42, lv43}, and Bogdanov-Takens (BT) \cite{lv44, lv45, lv46, lv47, lv48} bifurcations have been extensively studied in the literature. However, the derivation of normal form for codimension-two Turing-Turing bifurcation remains relatively unexplored. This significant gap persists despite the theoretical importance of such bifurcation in explaining multi-scale pattern formation arising from the interaction of two independent Turing instabilities.

The original contributions of this work are fourfold: First, we introduce an integrated ecological model, a reaction-diffusion-taxis system that synergistically incorporates predator-taxis (reflecting active prey avoidance) and a prey refuge. This integration offers a more ecologically realistic framework for investigating spatial dynamics. Second, we present a pioneering bifurcation analysis by conducting the first comprehensive study of the codimension-two Turing-Turing bifurcation in an ecological context. We rigorously establish the conditions for its occurrence and derive the corresponding explicit normal form, accompanied by fully computable coefficients. Third, we develop a general analytical framework rooted in center manifold and normal form theory, enabling the systematic analysis of this high-codimension bifurcation. The approach is applicable to a broad range of reaction-diffusion-taxis systems. Fourth, we achieve a synthesis of theoretical and numerical results: the complex multi-stable spatial patterns predicted by the normal form analysis are confirmed through numerical simulations, thereby bridging rigorous mathematical theory with observable dynamical behavior.

The paper is structured as follows. Section \ref{sec:2} establishes the local existence and uniqueness of classical solutions for model (1.1). Turing and Turing-Turing bifurcation analyses for model (1.1) are performed in Section \ref{sec:3}. Section \ref{sec:4} presents the algorithm for calculating the normal form of the Turing-Turing bifurcation in model (1.1). Numerical simulations validating the theoretical findings are conducted in Section \ref{sec:5}. Finally, Section \ref{sec:6} concludes this paper with a summary and discussion. Throughout this paper, $\mathbb{N}$ denotes the set of positive integers, and $\mathbb{N}_{0}=\mathbb{N} \cup \{0\}$.

\section{Local existence and uniqueness of classical solutions to model (1.1)}
\label{sec:2}

\begin{assumption}\label{asm:2.1}
$(u_{0}(x),v_{0}(x)) \in \left[W^{1,p}(\Omega)\right]^{2}$ where $p>1$ with $u_{0}(x), v_{0}(x) \geq 0 (\not \equiv 0)$.
\end{assumption}
Then, we have the following theorem.
\begin{theorem}\label{thm:2.2}
If Assumption \ref{asm:2.1} holds, then model (1.1) admits a unique local-in-time, non-negative classical solution
\begin{equation*}
(u,v) \in \left[C\left(\left[0,\mathrm{T}_{\max}\right); W^{1,p}(\Omega)\right) \cap C^{2,1}\left(\overline{\Omega} \times \left(0,\mathrm{T}_{\max}\right)\right)\right]^{2},
\end{equation*}
where $\mathrm{T}_{\max}$ denotes the maximal existence time.
\end{theorem}

\begin{proof}
To establish the non-negativity of solutions, we apply the maximum principle for parabolic equations to model (1.1). Denote
\begin{equation*}\begin{aligned}
&f(u,v)=u(r_{0}-au)-\frac{b_{1}(1-\beta)uv}{b_{2}v+(1-\beta)u}, \\
&g(u,v)=v\left(-m_{1}+\frac{cb_{1}(1-\beta)u}{b_{2}v+(1-\beta)u}\right).
\end{aligned}\end{equation*}
Consider the evolution of $u(x,t)$. If $u$ were to become negative at some point, the strong maximum principle would be violated. Specifically, at any point where $u=0$, the diffusion term $d_{u}\Delta u$ and the taxis term $\nabla \cdot (\alpha u\nabla v)$ both vanish. The reaction term $f(u,v)$ also equals zero when $u=0$. Therefore, if $u$ reaches zero, its time derivative cannot become negative, preventing $u$ from crossing to negative values. A similar argument applies to $v(x,t)$. When $v=0$, the diffusion term $d_{v}\Delta v$ vanishes. The reaction term $g(u,v)$ equals zero. Thus, $v$ cannot become negative once it is initially non-negative. Since the initial functions satisfy $u_{0}(x), v_{0}(x) \geq 0$ by Assumption 1, and the Neumann boundary conditions preserve non-negativity, we conclude that both $u(x,t)$ and $v(x,t)$ remain non-negative for all $t \in [0,\mathrm{T}_{\max})$.

Let $w=(u,v)$, then model (1.1) can be simplified to
\begin{equation}\left\{\begin{aligned}
&\frac{\partial w}{\partial t}=\nabla \cdot(a(w)\nabla w)+\Phi(w), & x \in \Omega,~t>0, \\
&\frac{\partial w}{\partial \nu}=0, & x \in \partial \Omega,~t>0, \\
&w(\cdot,0)=(u_{0},v_{0}), & x \in \Omega,
\end{aligned}\right.\end{equation}
where
\begin{equation*}
a(w)=\left(\begin{array}{cc}
d_{u} & \alpha u \\
0 & d_{v}
\end{array}\right), \quad \Phi(w)=\binom{f(u,v)}{g(u,v)}.
\end{equation*}
The diffusion matrix $a(w)$ exhibits a lower triangular structure. The second equation for $v$ depends only on its own diffusion coefficient $d_{v}$ and the reaction term $g(u,v)$, but not on the taxis terms involving $u$. This triangular structure enables a sequential solution approach, whereby the equation for $v$ is solved first, independently of the spatial derivatives of $u$, and its solution is then substituted into the equation for $u$. This structural property is crucial for applying Amann's theorem \cite{lv49}. The lower triangular form ensures that the system (2.1) is strongly parabolic, as the eigenvalues of $a(w)$ are $d_{u}$ and $d_{v}$, both positive. The reaction terms $\Phi(w)$ are smooth functions satisfying the required regularity conditions in $W^{1,p}(\Omega)$. Amann's theorem guarantees the existence of a unique local classical solution for such quasi-linear parabolic systems with triangular diffusion structure. Therefore, by Amann's theorem, there exists a maximal time $\mathrm{T}_{\max}>0$ such that model (1.1) admits a unique classical solution
\begin{equation*}
(u,v) \in \left[C\left(\left[0,\mathrm{T}_{\max}\right);W^{1,p}(\Omega)\right) \cap C^{2,1}\left(\overline{\Omega} \times \left(0,\mathrm{T}_{\max}\right)\right)\right]^{2}.
\end{equation*}
This completes the proof.
\end{proof}

\section{Turing bifurcation and Turing-Turing bifurcation}
\label{sec:3}

\subsection{Turing bifurcation analysis and Turing-Turing bifurcation analysis}

For model (1.1), in what follows, we let $\Omega=(0,\ell\pi), \ell>0$. Firstly, we analyze the existence and stability of the positive constant steady state for model (1.1). In fact, there exists a unique positive constant steady state $E_{*}:=(u_{*},v_{*})$ for model (1.1) if
\begin{equation*}
c>\frac{m_{1}}{b_{1}} \text{ and } r_{0}>\frac{(cb_{1}-m_{1})(1-\beta)}{b_{2}c},
\end{equation*}
where
\begin{equation}
u_{*}=\frac{1}{a}\left(r_{0}-\frac{(cb_{1}-m_{1})(1-\beta)}{b_{2}c}\right),~v_{*}=\frac{(cb_{1}-m_{1})(1-\beta)u_{*}}{b_{2}m_{1}}.
\end{equation}
Then, the linearized equation of model (1.1) at $E_{*}$ is given by
\begin{equation}
\left\{\begin{aligned}
&u_{t}=d_{u}\frac{\partial^{2}}{\partial x^{2}}u+\alpha u_{*}\frac{\partial^{2}}{\partial x^{2}}v+\left(-au_{*}+\frac{b_{1}(1-\beta)^{2}u_{*}v_{*}}{(b_{2}v_{*}+(1-\beta)u_{*})^{2}}\right)u-\frac{b_{1}(1-\beta)^{2}u_{*}^{2}}{(b_{2}v_{*}+(1-\beta)u_{*})^{2}}v, \\
&v_{t}=d_{v}\frac{\partial^{2}}{\partial x^{2}}v+\frac{cb_{1}b_{2}(1-\beta)v_{*}^{2}}{(b_{2}v_{*}+(1-\beta)u_{*})^{2}}u-\frac{cb_{1}b_{2}(1-\beta)u_{*}v_{*}}{(b_{2}v_{*}+(1-\beta)u_{*})^{2}}v, \\
&u_{x}(0,t)=v_{x}(0,t)=0, u_{x}(\ell\pi,t)=v_{x}(\ell\pi,t)=0,
\end{aligned}\right.
\end{equation}
which can be written as
\begin{equation}\begin{aligned}
\left(\begin{aligned}
u_{t} \\
v_{t}
\end{aligned}\right)&=\left(\begin{array}{cc}
d_{u} & \alpha u_{*} \\
0 & d_{v}
\end{array}\right)\left(\begin{aligned}
u_{xx} \\
v_{xx}
\end{aligned}\right)+\left(\begin{array}{cc}
-au_{*}+\frac{b_{1}(1-\beta)^{2}u_{*}v_{*}}{(b_{2}v_{*}+(1-\beta)u_{*})^{2}} & -\frac{b_{1}(1-\beta)^{2}u_{*}^{2}}{(b_{2}v_{*}+(1-\beta)u_{*})^{2}}  \\
\frac{cb_{1}b_{2}(1-\beta)v_{*}^{2}}{(b_{2}v_{*}+(1-\beta)u_{*})^{2}} & -\frac{cb_{1}b_{2}(1-\beta)u_{*}v_{*}}{(b_{2}v_{*}+(1-\beta)u_{*})^{2}}
\end{array}\right)\left(\begin{aligned}
u \\
v
\end{aligned}\right) \\
&:=D_{0}\left(\begin{aligned}
u_{xx} \\
v_{xx}
\end{aligned}\right)+A_{1}\left(\begin{aligned}
u \\
v
\end{aligned}\right).
\end{aligned}\end{equation}
Let $\left\{\mu_{n}=n^{2}/\ell^{2}: n \in \mathbb{N}_{0}\right\}$ be the eigenvalues of operator $-\frac{\partial^{2}}{\partial x^{2}}$ on $(0,\ell\pi)$ subject to Neumann boundary conditions. And for the sake of convenience, denote
\begin{equation*}
\varpi:=\frac{m_{1}(cb_{1}-m_{1})}{c^{2}b_{1}b_{2}}>0,
\end{equation*}
then the characteristic equations of (3.2) read
\begin{equation}
\Gamma_{n}(\lambda,\alpha):=\lambda^{2}+(p_{n}+s_{n})\lambda+(\sigma_{n}+q_{n}(\alpha))=0,~n \in \mathbb{N}_{0},
\end{equation}
where
\begin{equation}\begin{aligned}
&p_{n}=\frac{n^{2}}{\ell^{2}}(d_{u}+d_{v})+(au_{*}-(1-\beta)\varpi),~s_{n}=cb_{2}\varpi, \\
&\sigma_{n}=\frac{n^{4}}{\ell^{4}}d_{u}d_{v}+\frac{n^{2}}{\ell^{2}}d_{v}(au_{*}-(1-\beta)\varpi), \\
&q_{n}(\alpha)=\frac{n^{2}}{\ell^{2}}cb_{2}v_{*}\varpi\alpha+cb_{2}\varpi(\frac{n^{2}}{\ell^{2}}d_{u}+au_{*}).
\end{aligned}\end{equation}
Denote $\operatorname{DET}_{n}:=\sigma_{n}+q_{n}(\alpha)$, $\operatorname{TR}_{n}:=-(p_{n}+s_{n})$ for $n \in \mathbb{N}_{0}$, then from (3.5), we have
\begin{equation}\begin{aligned}
&\operatorname{DET}_{n}=\frac{n^{4}}{\ell^{4}}d_{u}d_{v}+\frac{n^{2}}{\ell^{2}}\left(d_{v}(au_{*}-(1-\beta)\varpi)+d_{u}cb_{2}\varpi+\alpha cb_{2}v_{*}\varpi\right)+acb_{2}u_{*}\varpi, \\
&\operatorname{TR}_{n}=-\frac{n^{2}}{\ell^{2}}(d_{u}+d_{v})-(au_{*}-(1-\beta)\varpi)-cb_{2}\varpi.
\end{aligned}\end{equation}
If
\begin{equation}
c>\frac{m_{1}}{b_{1}}
\end{equation}
and
\begin{equation}\begin{aligned}
&r_{0}>\underline{r_{0}} \\
&\triangleq \max \left\{\frac{(cb_{1}-m_{1})(1-\beta)}{b_{2}c},\frac{((cb_{1}+m_{1})(1-\beta)-cb_{2}m_{1})(cb_{1}-m_{1})}{c^{2}b_{1}b_{2}}\right\},
\end{aligned}\end{equation}
then $\operatorname{TR}_{0}<0$, hence $E_{*}$ is locally asymptotically stable for local ODE model since $\operatorname{DET}_{0}>0$. So in the following discussions, we always assume that (3.7) and (3.8) are satisfied.

Next, we focus on discussing the occurrence of Turing bifurcation and Turing-Turing bifurcation for model (1.1). For convenience, let us denote
\begin{equation}
\overline{r_{0}}:=\frac{(cb_{1}-m_{1})(cb_{1}+m_{1})(1-\beta)}{c^{2}b_{1}b_{2}}>0,
\end{equation}
and it is obvious that $\overline{r_{0}}>\underline{r_{0}}$. When $r_{0} \geqslant \overline{r_{0}}$, then $au_{*}-(1-\beta)\varpi \geqslant 0$, hence $\operatorname{TR}_{n}<0$, $\operatorname{DET}_{n}>0$ for $n \in \mathbb{N}_{0}$, which indicates that $E_{*}$ is locally asymptotically stable. When $r_{0}<\overline{r_{0}}$, then $au_{*}-(1-\beta)\varpi<0$. According to (3.6), $\operatorname{TR}_{n}<0$ for $n \in \mathbb{N}_{0}$ also holds. Whereas, there exists some $n \in \mathbb{N}$ such that $\operatorname{DET}_{n}<0$, meaning that Turing bifurcation can occur for model (1.1). Then, for any fixed $d_{v}>0$, define
\begin{equation}
\alpha(n,d_{u}):=-\frac{d_{u}d_{v}\frac{n^{4}}{\ell^{4}}+\left(d_{v}(au_{*}-(1-\beta)\varpi)+d_{u}cb_{2}\varpi\right)\frac{n^{2}}{\ell^{2}}+acb_{2}u_{*}\varpi}{\frac{n^{2}}{\ell^{2}}cb_{2}v_{*}\varpi}, d_{u}>0, n \in \Lambda,
\end{equation}
where
\begin{equation}
\Lambda:=\left\{n \in \mathbb{N}: n>\widehat{n} \triangleq \ell\sqrt{\frac{acb_{2}u_{*}\varpi}{d_{v}((1-\beta)\varpi-au_{*})}}\right\}.
\end{equation}
Also, denote
\begin{equation*}
\overline{\alpha}(n):=\lim_{d_{u}\rightarrow 0^{+}}\alpha(n,d_{u})=-\frac{au_{*}\ell^{2}}{n^{2}v_{*}}+\frac{d_{v}((1-\beta)\varpi-au_{*})}{cb_{2}v_{*}\varpi}.
\end{equation*}
To describe the critical conditions for the instability of the positive constant steady state, we first investigate the monotonicity of $\overline{\alpha}(n)$ with respect to the wave number $n$.
\begin{lemma}\label{lem:3.1}
For any $a, b_{1}, b_{2}, \beta, d_{v}, m_{1}>0$, provided that $c>\frac{m_{1}}{b_{1}}$, $\underline{r_{0}}<r_{0}<\overline{r_{0}}$, then $\overline{\alpha}(n)>0$ is monotonically increasing in $n$ for $n \in \Lambda$.
\end{lemma}

A direct calculation from $\alpha(n,d_{u})=0$ yields
\begin{equation}
d_{u}^{n}:=-\frac{d_{v}(au_{*}-(1-\beta)\varpi)\frac{n^{2}}{\ell^{2}}+acb_{2}u_{*}\varpi}{\frac{n^{2}}{\ell^{2}}(d_{v}\frac{n^{2}}{\ell^{2}}+cb_{2}\varpi)}, n \in \Lambda,
\end{equation}
which is a critical value of $d_{u}$ when $E_{*}$ is destabilized. For any $d_{v}>0$, let
\begin{equation}
\overline{n}:=\left\{\begin{aligned}
\left\lfloor n_{*}\right\rfloor+1, & \text{ for } d_{u}^{\left\lfloor n_{*}\right\rfloor} \leqslant d_{u}^{\left\lfloor n_{*}\right\rfloor+1}, \\
\left\lfloor n_{*}\right\rfloor, & \text{ for } d_{u}^{\left\lfloor n_{*}\right\rfloor}>d_{u}^{\left\lfloor n_{*}\right\rfloor+1},
\end{aligned}\right.
\end{equation}
then $\overline{n}$ is the critical wave number of the nonconstant steady states arising from $E_{*}$ through Turing instability, where
\begin{equation}
n_{*}=\ell\sqrt{\frac{cb_{2}\varpi(au_{*}+\sqrt{au_{*}(1-\beta)\varpi})}{d_{v}((1-\beta)\varpi-au_{*})}},
\end{equation}
and $\lfloor\cdot\rfloor$ is the floor function.

The next lemma concerns the monotonicity of $d_{u}^{n}$ in $n$ and $\alpha(n,d_{u})$ in $d_{u}$, respectively.
\begin{lemma}\label{lem:3.2}
For any $a, b_{1}, b_{2}, \beta, d_{v}, m_{1}>0$, provided that $c>\frac{m_{1}}{b_{1}}$, $\underline{r_{0}}<r_{0}<\overline{r_{0}}$, then

(i) $d_{u}^{n}$ is monotonically increasing with respect to $n \in \mathbb{N}$ for $\widehat{n}<n<\overline{n}$, and monotonically decreasing with respect to $n \in \mathbb{N}$ for $n \geqslant \overline{n}$;

(ii) for any $n \in \Lambda$, $\alpha(n,d_{u})$ is linear and monotonically decreasing with respect to $d_{u}$ for $d_{u}>0$. Particularly, $\alpha(n,d_{u})>0$ for $0<d_{u}<d_{u}^{n}, n \in \Lambda$.
\end{lemma}

\begin{proof}
We only prove the case of $d_{u}^{\left\lfloor n_{\star}\right\rfloor} \leqslant d_{u}^{\left\lfloor n_{\star}\right\rfloor+1},\overline{n}=\left\lfloor n_{*}\right\rfloor+1$, and the other cases can be established using similar arguments. For part (i), by (3.12), define $\Gamma(x)$ by
\begin{equation*}
\Gamma(x):=-\frac{d_{v}(au_{*}-(1-\beta)\varpi)x+acb_{2}u_{*}\varpi}{x(d_{v}x+cb_{2}\varpi)}, x>0.
\end{equation*}
So by a direct calculation, there exists $x^{*}=\frac{n_{*}^{2}}{\ell^{2}}$ satisfying that $\Gamma(x)$ is increasing in $x$ on $(0,x^{*})$, and decreasing in $x$ on $(x^{*},\infty)$. And by (3.11) and (3.14), it follows that
\begin{equation*}
n_{*}^{2}-\widehat{n}^{2}=\ell^{2}\frac{cb_{2}\varpi\sqrt{au_{*}(1-\beta)\varpi}}{d_{v}((1-\beta)\varpi-au_{*})}>0,
\end{equation*}
that is, $n_{*}>\widehat{n}$, thus $\left\lfloor n_{*}\right\rfloor+1\geqslant n_{*}>\widehat{n}$, which implies that $d_{u}^{n}$ is decreasing for $\left\lfloor n_{*}\right\rfloor+1 \leqslant n \in \mathbb{N}$. On the other hand, for $n \in\left\{n \in \mathbb{N}: \widehat{n}<n<\left\lfloor n_{*}\right\rfloor+1\right\}$, $d_{u}^{n}$ is monotonically increasing with respect to $n$. For part (ii), for any $n \in \Lambda$, we rewrite the expression of $\alpha(n,d_{u})$ in (3.10) as follows
\begin{equation}
\alpha(n,d_{u}):=-\frac{d_{v}\frac{n^{2}}{\ell^{2}}+cb_{2}\varpi}{cb_{2}v_{*}\varpi}d_{u}-\frac{d_{v}(au_{*}-(1-\beta)\varpi)\frac{n^{2}}{\ell^{2}}+acb_{2}u_{*}\varpi}{\frac{n^{2}}{\ell^{2}}cb_{2}v_{*}\varpi}, d_{u}>0,
\end{equation}
then $\alpha(n,d_{u})$ is a linear function of $d_{u}$, thus for $d_{u}>0$, $\alpha(n,d_{u})$ is monotonically decreasing with respect to $d_{u}$, implying that $\alpha(n,d_{u})>0$ for $0<d_{u}<d_{u}^{n}, n \in \Lambda$.
\end{proof}

To this end, based on (3.15), and for any $d_{v}>0$, we define
\begin{equation}
\alpha_{*}(n,d_{u}):=\max\left\{\alpha(n,d_{u}),0\right\}, n \in \Lambda, 0<d_{u}<d_{u}^{\overline{n}}.
\end{equation}
It is well known that the first critical value of the Turing bifurcation determines the stability of the positive constant steady state \cite{lv33}. Therefore, to describe the critical region where the stability of $E_{*}$ changes in the $(d_{u},\alpha)$-plane, we next discuss the intersections of $\alpha=\alpha_{*}(n,d_{u})$ for $n \in \Lambda, 0<d_{u}<d_{u}^{\overline{n}}$. By solving $\alpha(n,d_{u})=\alpha(n+1,d_{u})$, we have the following Lemma.
\begin{lemma}\label{lem:3.3}
For any $a, b_{1}, b_{2}, \beta, d_{v}, m_{1}>0$, provided that $c>\frac{m_{1}}{b_{1}}$, $\underline{r_{0}}<r_{0}<\overline{r_{0}}$, then for any $n \in \Lambda$, the equation
\begin{equation*}
\alpha_{*}(n,d_{u})=\alpha_{*}(n+1,d_{u}),~0<d_{u}<d_{u}^{\overline{n}}
\end{equation*}
has a unique root $d_{u}^{n,n+1} \in (0,d_{u}^{n+1})$, which is denoted by
\begin{equation}
d_{u}^{n,n+1}:=\frac{acb_{2}u_{*}\varpi}{d_{v}\frac{n^{2}(n+1)^{2}}{\ell^{4}}}.
\end{equation}
\end{lemma}

Synthesizing the above discussions, based on (3.16), and for $n \in \Lambda$, $0<d_{u}<d_{u}^{\overline{n}}$, we define
\begin{equation}
\alpha_{*}(d_{u}):=\alpha_{*}(n,d_{u}), d_{u} \in \left[d_{u}^{n,n+1},d_{u}^{n-1,n}\right), n \geqslant \overline{n}
\end{equation}
with $\overline{n}$ is defined as in (3.13). Then $\alpha=\alpha_{*}(d_{u})$, $0<d_{u}<d_{u}^{\overline{n}}$ is called the first Turing bifurcation curve, which is the critical curve for Turing instability of $E_{*}$.

\begin{lemma}\label{lem:3.4}
For any $a, b_{1}, b_{2}, \beta, d_{v}, m_{1}>0$, provided that $c>\frac{m_{1}}{b_{1}}$, $\underline{r_{0}}<r_{0}<\overline{r_{0}}$, then for any $d_{u} \in (0,d_{u}^{\overline{n}})$, there must exist some integer $n_{1} \geqslant \overline{n}$ such that $d_{u} \in \left[d_{u}^{n_{1},n_{1}+1},d_{u}^{n_{1}-1,n_{1}}\right)$. Specifically,

(i) if $d_{u} \in (d_{u}^{n_{1},n_{1}+1},d_{u}^{n_{1}-1,n_{1}})$ and $\alpha=\alpha_{*}(d_{u})$, 0 is a simple root of characteristic equation (3.4) with $n=n_{1}$, and the other roots of (3.4) have strictly negative real parts. Moreover, let $\lambda_{1}(n_{1},\alpha)$ be the root of (3.4) satisfying $\lambda_{1}(n_{1},\alpha_{*}(d_{u}))=0$, then
\begin{equation}
\frac{\mathrm{d}\lambda_{1}(n_{1},\alpha_{*}(d_{u}))}{\mathrm{d}\alpha}=\frac{-\frac{n_{1}^{2}}{\ell^{2}}cb_{2}v_{*}\varpi}{p_{n_{1}}+s_{n_{1}}}<0.
\end{equation}

(ii) if $d_{u}=d_{u}^{n_{1},n_{1}+1}$ and $\alpha=\alpha_{*}(d_{u}^{n_{1},n_{1}+1})$, then 0 is a simple root of characteristic equation (3.4) for both $n_{1}$ and $n_{1}+1$, and the other roots of (3.4) have strictly negative real parts.
\end{lemma}

\begin{proof}
We provide the proof for part (i), and the second assertion can be proved in a similar manner. Note that in (3.6), $\operatorname{DET}_{n}=0$ if and only if $\alpha=\alpha_{*}(n,d_{u})$ for $n \in \Lambda$, $0<d_{u}<d_{u}^{\overline{n}}$. So, for any $n_{1} \in \Lambda$, $\lambda_{1}(n_{1},\alpha)=0$ is always a root of (3.4) with such $n_{1}$ when $\alpha=\alpha_{*}(n_{1},d_{u})$. By the definition of $\alpha_{*}(d_{u})$ and $\alpha_{*}(n,d_{u})$, if $d_{u} \in (d_{u}^{n_{1},n_{1}+1},d_{u}^{n_{1}-1,n_{1}})$ and $\alpha=\alpha_{*}(d_{u})$, then 0 is a root of (3.4) with $n=n_{1}$. Moreover, it follows
\begin{equation*}
\left.\frac{\mathrm{d}\Gamma_{n_{1}}(\lambda,\alpha)}{\mathrm{d}\lambda}\right|_{\lambda=0}=-\operatorname{TR}_{n_{1}}>0
\end{equation*}
that $\lambda=0$ is simple. The conditions (3.7) and (3.8) ensures that $\operatorname{TR}_{n}<0$ for all $n \in \mathbb{N}_{0}$ and $\operatorname{DET}_{n}>0$ for $n \geqslant \overline{n}, n \in \mathbb{N}$ and $n \neq n_{1}$. Thus, all the other roots of (3.4) have strictly negative real parts. By differentiating (3.4) with respect to $\alpha$, we can verify that the transversality condition (3.19) is satisfied.

\end{proof}

\begin{theorem}\label{thm:3.5}
For model (1.1) and any $a, b_{1}, b_{2}, \beta, d_{v}, m_{1}>0$, provided that $c>\frac{m_{1}}{b_{1}}$, then

(1) when $r_{0} \geqslant \overline{r_{0}}$, for $d_{u}>0, \alpha>0$, the positive constant steady state $E_{*}$ is always locally asymptotically stable;

(2) when $\underline{r_{0}}<r_{0}<\overline{r_{0}}$,

(a) if further $d_{u} \geqslant d_{u}^{\overline{n}}$, then $E_{*}$ is locally asymptotically stable for $\alpha>0$;

(b) when $0<d_{u}<d_{u}^{\overline{n}}$,

(i) $E_{*}$ is locally asymptotically stable for $\alpha>\alpha_{*}(d_{u})$, and unstable for $0<\alpha<\alpha_{*}(d_{u})$;

(ii) if further $d_{u} \in (d_{u}^{n_{1},n_{1}+1},d_{u}^{n_{1}-1,n_{1}})$ for some $n_{1} \geqslant \overline{n}$ and $n_{1} \in \mathbb{N}$, the model (1.1) undergoes mode-$n_{1}$ Turing bifurcation when $\alpha=\alpha_{*}(d_{u})$;

(iii) if further $d_{u}=d_{u}^{n_{1},n_{1}+1}$ for some $n_{1} \geqslant \overline{n}$ and $n_{1} \in \mathbb{N}$, the model (1.1) undergoes mode-$(n_{1},n_{1}+1)$ Turing-Turing bifurcation when $\alpha=\alpha_{*}(d_{u}^{n_{1},n_{1}+1})$.

\end{theorem}
Here, $E_{*}$, $\underline{r_{0}}$, $\overline{r_{0}}$, $d_{u}^{\overline{n}}$, $\overline{n}$, $d_{u}^{n_{1},n_{1}+1}$, and $\alpha_{*}(d_{u})$ are given in (3.1), (3.8), (3.9), (3.12), (3.13), (3.17), and (3.18), respectively. Let $\ell=3$, $a=0.4$, $r_{0}=0.5$, $c=1$, $b_{1}=1$, $b_{2}=0.98$, $m_{1}=0.6$, $d_{v}=0.6$, $\beta=0.05$, it follows from (3.13) that $\overline{n}=3$, and (3.7) and (3.8) are satisfied. Thus, $E_{*}=(0.2806,0.1813)$ is locally asymptotically stable for the local ODE model. In Fig. \ref{fig:1}(a), we demonstrate the schematic diagram of Turing bifurcation curves $\alpha=\alpha_{*}(n,d_{u})$ when $0<d_{u}<d_{u}^{\overline{n}}$ for different $n \in \Lambda$ in $(d_{u},\alpha)$-plane. The corresponding first Turing bifurcation curve $\alpha=\alpha_{*}(d_{u})$ with $0<d_{u}<d_{u}^{\overline{n}}$ is plotted in Fig. \ref{fig:1}(b), and the non-smooth points of $\alpha=\alpha_{*}(d_{u})$, $0<d_{u}<d_{u}^{\overline{n}}$, $T_{\overline{n}, \overline{n}+1}, T_{\overline{n}+1, \overline{n}+2}, \cdots$ are Turing-Turing bifurcation points.

\begin{figure}[!htbp]
\centering
\includegraphics[width=3in]{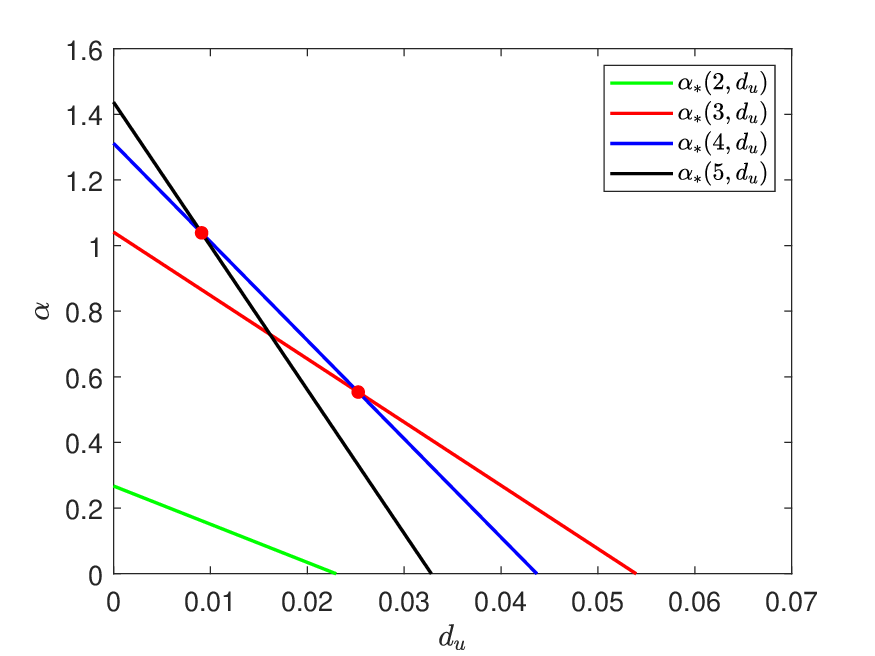}
\includegraphics[width=3in]{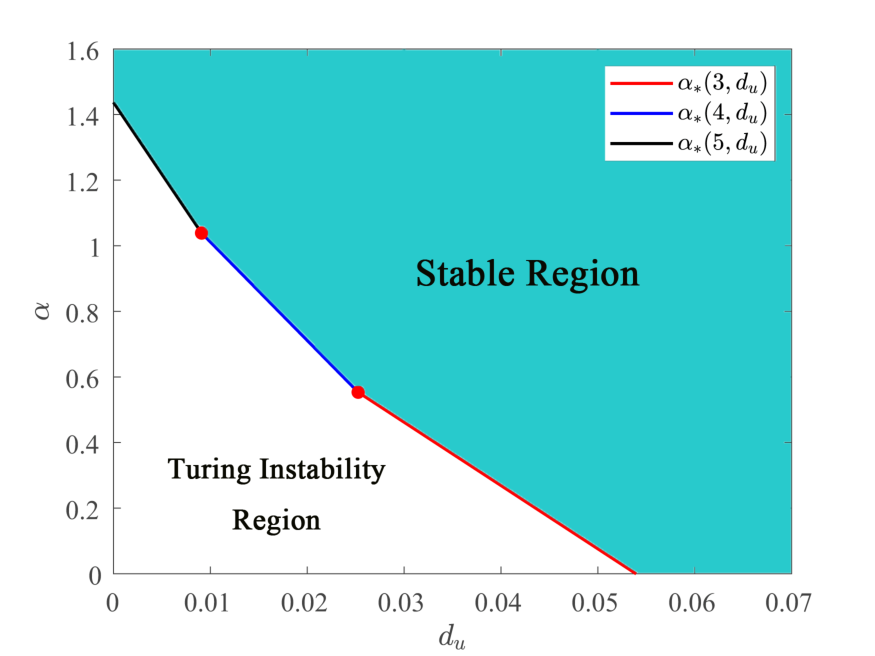} \\
\text{(a)} \hspace{7.8cm} \text{(b)} \\
\caption{(a) Turing bifurcation curves $\alpha=\alpha_{*}(n,d_{u})$ when $0<d_{u}<d_{u}^{\overline{n}}$ for different $n \in \Lambda$ in $(d_{u},\alpha)$-plane. (b) The first Turing bifurcation curve $\alpha=\alpha_{*}(d_{u})$ with $0<d_{u}<d_{u}^{\overline{n}}$.}
\label{fig:1}
\end{figure}

\section{Algorithm for calculating the normal form of Turing-Turing bifurcation for model (1.1)}
\label{sec:4}

\subsection{Basic assumption and equation transformation}

\begin{assumption}\label{asm:4.1}
When $(d_{u},\alpha)=(d_{u}^{*},\alpha^{*})$, assume that there is a neighborhood $V_{0} \subset \mathbb{R}^{2}$ containing $(0,0)$ such that $\mu:=(\mu_{1},\mu_{2}) \in V_{0}$, and the characteristic equations (3.4) of the model (1.1) possess two simple independent real eigenvalues $\lambda_{i}(\mu),~i=1,2$ associated with different $n_{1}, n_{2} \in \mathbb{N}$, which satisfy $\lambda_{i}(0)=0$, $\partial\lambda_{i}(0)/\partial\mu_{i}  \neq 0$, $i=1,2$. Meanwhile, the real part of other eigenvalues of the characteristic equations (3.4) is non-zero.
\end{assumption}

Define the real-valued Sobolev space
\begin{equation*}
X:=\left\{(u,v)^{T} \in \left(W^{2,2}(0,\ell\pi)\right)^{2}: \frac{\partial u}{\partial x}=\frac{\partial v}{\partial x}=0 \text { at } x=0, \ell\pi\right\}
\end{equation*}
with the inner product defined by
\begin{equation*}
\left[U_{1},U_{2}\right]=\int_{0}^{\ell\pi}U_{1}^{T}U_{2}\mathrm{d}x \text { for } U_{1}=(u_{1},v_{1})^{T} \in X \text{ and } U_{2}=(u_{2},v_{2})^{T} \in X,
\end{equation*}
where the symbol $T$ represents the transpose of vector. Furthermore, notice that the eigenvalue problem
\begin{equation*}\left\{\begin{aligned}
&\widetilde{\varphi}^{\prime \prime}(x)=\widetilde{\lambda}\widetilde{\varphi}(x),~x \in(0,\ell\pi), \\
&\widetilde{\varphi}^{\prime}(0)=\widetilde{\varphi}^{\prime}(\ell\pi)=0
\end{aligned}\right.\end{equation*}
has eigenvalues $\widetilde{\lambda}_{n}=-n^{2}/\ell^{2}$ with corresponding normalized eigenfunctions
\begin{equation}
\beta_{n}^{(j)}=\gamma_{n}(x)e_{j},~\gamma_{n}(x)=\frac{\cos(nx/\ell)}{\left\|\cos(nx/\ell)\right\|_{2,2}}=\left\{\begin{aligned}
&\frac{1}{\sqrt{\ell\pi}}, & n=0, \\
&\frac{\sqrt{2}}{\sqrt{\ell\pi}}\cos\left(\frac{nx}{\ell}\right), & n \geq 1,
\end{aligned}\right.\end{equation}
where $e_{j},~j=1,2$ are the unit coordinate vector of $\mathbb{R}^{2}$. Let $d_{u}=d_{u}^{*}+\mu_{1}$ and $\alpha=\alpha^{*}+\mu_{2}$ such that $(\mu_{1},\mu_{2})=(0,0)$ corresponds to the Turing-Turing bifurcation point of the model (1.1). Moreover, we shift the positive constant steady state $E_{*}(u_{*},v_{*})$ to the origin by setting
\begin{equation*}
U(x,t)=(U_{1}(x,t),U_{2}(x,t))^{T}=(u(x,t),v(x,t))^{T}-(u_{*},v_{*})^{T},
\end{equation*}
and rewrite $U(t)$ for $U(x,t)$. Then the model (1.1) becomes the compact form
\begin{equation}
\frac{\mathrm{d}U(t)}{\mathrm{d}t}=d(\mu)\Delta U(t)+LU(t)+F(U(t)),
\end{equation}
where denote $U(t):=\varphi=(\varphi^{(1)},\varphi^{(2)})^{T}$, and $d(\mu)\Delta$, $L$, $F$ are given, respectively, by
\begin{equation}\begin{aligned}
&d(\mu)\Delta\varphi=d_{0}\Delta\varphi+F^{d}(\varphi,\mu),~L\varphi=A_{1}\varphi, \\
&F(\varphi)=\left(\begin{array}{c}
(\varphi^{(1)}+u_{*})\left(r_{0}-a(\varphi^{(1)}+u_{*})\right)-\frac{b_{1}(1-\beta)(\varphi^{(1)}+u_{*})(\varphi^{(2)}+v_{*})}{b_{2}(\varphi^{(2)}+v_{*})+(1-\beta)(\varphi^{(1)}+u_{*})} \\
-m_{1}(\varphi^{(2)}+v_{*})+\frac{cb_{1}(1-\beta)(\varphi^{(1)}+u_{*})(\varphi^{(2)}+v_{*})}{b_{2}(\varphi^{(2)}+v_{*})+(1-\beta)(\varphi^{(1)}+u_{*})}
\end{array}\right)-L\varphi
\end{aligned}\end{equation}
with
\begin{equation}\begin{aligned}
d_{0}\Delta\varphi&=D_{0}\varphi_{xx}=\left(\begin{array}{cc}
d_{u}^{*} & \alpha^{*}u_{*} \\
0 & d_{v}
\end{array}\right)\varphi_{xx}, \\
F^{d}(\varphi,\mu)&=\left(\begin{array}{c}
\alpha^{*}\left(\varphi^{(1)}_{x}\varphi^{(2)}_{x}+\varphi^{(1)}\varphi^{(2)}_{xx}\right) \\
0
\end{array}\right)+\mu_{1}\left(\begin{array}{c}
\varphi_{xx}^{(1)} \\
0
\end{array}\right)+\mu_{2}\left(\begin{array}{c}
\varphi^{(1)}_{x}\varphi^{(2)}_{x}+\varphi^{(1)}\varphi^{(2)}_{xx}+u_{*}\varphi_{xx}^{(2)} \\
0
\end{array}\right).
\end{aligned}\end{equation}
Here, $D_{0}$ and $A_{1}$ are given by (3.3). In what follows, assume that $F(\varphi)$ is $C^{n},~n \geq 3$, smooth with respect to $\varphi$, then (4.2) is rewritten as
\begin{equation}
\frac{\mathrm{d}U(t)}{\mathrm{d}t}=d_{0}\Delta U(t)+LU(t)+\widetilde{F}(U(t),\mu)
\end{equation}
by separating the linear terms from the nonlinear terms, where
\begin{equation}
\widetilde{F}(\varphi,\mu)=F^{d}(\varphi,\mu)+F(\varphi).
\end{equation}
Then the linearized system of (4.5) is
\begin{equation}
\frac{\mathrm{d}U(t)}{\mathrm{d}t}=d_{0}\Delta U(t)+LU(t),
\end{equation}
and the characteristic equation for the linearized system (4.7) is
\begin{equation*}
\prod_{n \in \mathbb{N}_{0}}\widetilde{\Gamma}_{n}(\lambda)=0,
\end{equation*}
where $\widetilde{\Gamma}_{n}(\lambda)=\operatorname{det}(\widetilde{\mathcal{M}}_{n}(\lambda))$ with
\begin{equation}
\widetilde{\mathcal{M}}_{n}(\lambda)=\lambda I_{2}+\frac{n^{2}}{\ell^{2}}D_{0}-A_{1},
\end{equation}
and $I_{2}$ is the $2\times 2$ identity matrix. Choose
\begin{equation}
\Phi_{n_{1}}=\phi_{n_{1}},~\Phi_{n_{2}}=\phi_{n_{2}},~\Psi_{n_{1}}=\psi_{n_{1}}^{T},~\Psi_{n_{2}}=\psi_{n_{2}}^{T},
\end{equation}
where $\phi_{n_{1}}=\operatorname{col}(\phi_{n_{1}}^{(1)},\phi_{n_{1}}^{(2)}) \in \mathbb{R}^{2}$ and $\phi_{n_{2}}=\operatorname{col}(\phi_{n_{2}}^{(1)},\phi_{n_{2}}^{(2)}) \in \mathbb{R}^{2}$ are the eigenvectors associated with two simple independent real eigenvalues $0$, $\psi_{n_{1}}=\operatorname{col}(\psi_{n_{1}}^{(1)},\psi_{n_{1}}^{(2)}) \in \mathbb{R}^{2}$ and $\psi_{n_{2}}=\operatorname{col}(\psi_{n_{2}}^{(1)},\psi_{n_{2}}^{(2)}) \in \mathbb{R}^{2}$ are the corresponding adjoint eigenvectors such that
\begin{equation}
\left\langle\Psi_{n_{1}},\Phi_{n_{1}}\right\rangle_{n_{1}}=1,~\left\langle\Psi_{n_{2}},\Phi_{n_{2}}\right\rangle_{n_{2}}=1.
\end{equation}
By referring to \cite{lv29}, the phase space $X$ can be decomposed as
\begin{equation}
X=\mathcal{P} \oplus \mathcal{Q},~\mathcal{P}=\operatorname{Im}\pi,~\mathcal{Q}=\operatorname{Ker}\pi,
\end{equation}
where for $\phi \in X$, the projection $\pi: X \rightarrow \mathcal{P}$ is defined by
\begin{equation}
\pi(\phi)=\Phi_{n_{1}}\left\langle\Psi_{n_{1}},\left(\begin{aligned}
&\left[\phi,\beta_{n_{1}}^{(1)}\right] \\
&\left[\phi,\beta_{n_{1}}^{(2)}\right]
\end{aligned}\right)\right\rangle_{n_{1}}\gamma_{n_{1}}(x)+\Phi_{n_{2}}\left\langle\Psi_{n_{2}},\left(\begin{aligned}
&\left[\phi,\beta_{n_{2}}^{(1)}\right] \\
&\left[\phi,\beta_{n_{2}}^{(2)}\right]
\end{aligned}\right)\right\rangle_{n_{2}}\gamma_{n_{2}}(x).
\end{equation}
Then by combining with (4.9), (4.10) and (4.11), and by a direct calculation, we have
\begin{equation*}\begin{aligned}
\phi_{n_{1}}&=\left(1,\frac{g_{u}}{d_{v}\frac{n_{1}^{2}}{\ell^{2}}-g_{v}}\right)^{T},~\phi_{n_{2}}=\left(1,\frac{g_{u}}{d_{v}\frac{n_{2}^{2}}{\ell^{2}}-g_{v}}\right)^{T}, \\
\psi_{n_{1}}&=\left(1-\frac{g_{u}\left(d_{u}^{*}\frac{n_{1}^{2}}{\ell^{2}}-f_{u}\right)}{g_{u}\left(d_{v}\frac{n_{1}^{2}}{\ell^{2}}-g_{v}\right)+g_{u}\left(d_{u}^{*}\frac{n_{1}^{2}}{\ell^{2}}-f_{u}\right)},\frac{d_{u}^{*}\frac{n_{1}^{2}}{\ell^{2}}-f_{u}}{g_{u}+\left(d_{u}^{*}\frac{n_{1}^{2}}{\ell^{2}}-f_{u}\right)\frac{g_{u}}{d_{v}\frac{n_{1}^{2}}{\ell^{2}}-g_{v}}}\right)^{T}, \\
\psi_{n_{2}}&=\left(1-\frac{g_{u}\left(d_{u}^{*}\frac{n_{2}^{2}}{\ell^{2}}-f_{u}\right)}{g_{u}\left(d_{v}\frac{n_{2}^{2}}{\ell^{2}}-g_{v}\right)+g_{u}\left(d_{u}^{*}\frac{n_{2}^{2}}{\ell^{2}}-f_{u}\right)},\frac{d_{u}^{*}\frac{n_{2}^{2}}{\ell^{2}}-f_{u}}{g_{u}+\left(d_{u}^{*}\frac{n_{2}^{2}}{\ell^{2}}-f_{u}\right)\frac{g_{u}}{d_{v}\frac{n_{2}^{2}}{\ell^{2}}-g_{v}}}\right)^{T}.
\end{aligned}\end{equation*}
By combining with (4.8), (4.10) and (4.11), $U(t)$ can be decomposed as
\begin{equation}\begin{aligned}
U(t)&=(\Phi_{n_{1}}~~\Phi_{n_{2}})\left(\begin{array}{c}
z_{1}\gamma_{n_{1}}(x) \\
z_{2}\gamma_{n_{2}}(x)
\end{array}\right)+w \\
&=\left(\Phi_{n_{1}}~~\Phi_{n_{2}}\right)\left(\begin{array}{c}
z_{1}\gamma_{n_{1}}(x) \\
z_{2}\gamma_{n_{2}}(x)
\end{array}\right)+\left(\begin{array}{c}
w_{1} \\
w_{2}
\end{array}\right),
\end{aligned}\end{equation}
where $w=\operatorname{col}(w_{1},w_{2})\in \mathcal{Q}$ and
\begin{equation*}
z_{1}=\left\langle\Psi_{n_{1}},\left(\begin{aligned}
&\left[U(t),\beta_{n_{1}}^{(1)}\right] \\
&\left[U(t),\beta_{n_{1}}^{(2)}\right]
\end{aligned}\right)\right\rangle_{n_{1}},~z_{2}=\left\langle\Psi_{n_{2}},\left(\begin{aligned}
&\left[U(t),\beta_{n_{2}}^{(1)}\right] \\
&\left[U(t),\beta_{n_{2}}^{(2)}\right]
\end{aligned}\right)\right\rangle_{n_{2}}.
\end{equation*}
Let
\begin{equation}
\Phi=(\Phi_{n_{1}},\Phi_{n_{2}}),~z_{x}=\left(z_{1}\gamma_{n_{1}}(x),z_{2}\gamma_{n_{2}}(x)\right)^{T},
\end{equation}
then (4.13) can be rewritten as $U(t)=\Phi z_{x}+w$. For the simplicity of notation, let
\begin{equation*}\begin{aligned}
&\left(\begin{aligned}
&\left[\widetilde{F}\left(\Phi z_{x}+w,\mu\right),\beta_{\nu}^{(1)}\right] \\
&\left[\widetilde{F}\left(\Phi z_{x}+w,\mu\right),\beta_{\nu}^{(2)}\right]
\end{aligned}\right)_{\nu=n_{1}}^{\nu=n_{2}} \\
&:=\operatorname{col}\left(\left(\begin{aligned}
&\left[\widetilde{F}\left(\Phi z_{x}+w,\mu\right),\beta_{n_{1}}^{(1)}\right] \\
&\left[\widetilde{F}\left(\Phi z_{x}+w,\mu\right),\beta_{n_{1}}^{(2)}\right]
\end{aligned}\right),\left(\begin{aligned}
&\left[\widetilde{F}\left(\Phi z_{x}+w,\mu\right),\beta_{n_{2}}^{(1)}\right] \\
&\left[\widetilde{F}\left(\Phi z_{x}+w,\mu\right),\beta_{n_{2}}^{(2)}\right]
\end{aligned}\right)\right),
\end{aligned}\end{equation*}
then the system (4.5) is decomposed as a system of abstract ODE on $\mathbb{R}^{2} \times \operatorname{Ker}\pi$, with finite and infinite dimensional variables are separated in the linear term, i.e.,
\begin{equation}
\left\{\begin{aligned}
\dot{z}&=Bz+\Psi\left(\begin{aligned}
&\left[\widetilde{F}(\Phi z_{x}+w,\mu),\beta_{\nu}^{(1)}\right] \\
&\left[\widetilde{F}(\Phi z_{x}+w,\mu),\beta_{\nu}^{(2)}\right]
\end{aligned}\right)_{\nu=n_{1}}^{\nu=n_{2}}, \\
\dot{w}&=\mathcal{L}_{n}w+(I_{2}-\pi)\widetilde{F}(\Phi z_{x}+w,\mu),
\end{aligned}\right.
\end{equation}
where $z=(z_{1},z_{2})^{T}$, $B=\operatorname{diag}\left\{0,0\right\}$ and $\Psi=\operatorname{diag}\left\{\Psi_{n_{1}},\Psi_{n_{2}}\right\}$ are the diagonal matrices, and $\mathcal{L}_{n}$ is defined by
\begin{equation}
\mathcal{L}_{n}w=D_{0}w_{xx}+A_{1}w.
\end{equation}
Consider the formal Taylor expansions
\begin{equation}\begin{aligned}
&\widetilde{F}(\varphi,\mu)=\sum_{j \geq 2}\frac{1}{j!}\widetilde{F}_{j}(\varphi,\mu),~F^{d}(\varphi,\mu)=\sum_{j \geq 2}\frac{1}{j!}F_{j}^{d}(\varphi,\mu), \\
&F(\varphi)=\sum_{j \geq 2}\frac{1}{j!}F_{j}(\varphi),
\end{aligned}\end{equation}
and according to (4.3), (4.4), (4.6), we have
\begin{equation}\begin{aligned}
&\widetilde{F}_{2}(\varphi,\mu)=F_{2}^{d}(\varphi,\mu)+F_{2}(\varphi), \\
&\widetilde{F}_{3}(\varphi,\mu)=F_{3}^{d}(\varphi,\mu)+F_{3}(\varphi), \\
\end{aligned}\end{equation}
and
\begin{equation}
\widetilde{F}_{j}(\varphi,\mu)=F_{j}^{d}(\varphi,\mu)+F_{j}(\varphi)=F_{j}(\varphi),~j=4,\cdots,
\end{equation}
where $F_{j}^{d}(\varphi,\mu)=(0,0)^{T},~j=4,\cdots$. By combining with (4.17), the system (4.15) can be rewritten as
\begin{equation}
\left\{\begin{aligned}
&\dot{z}=Bz+\sum_{j \geq 2}\frac{1}{j!}f_{j}^{1}(z,w,\mu), \\
&\dot{w}=\mathcal{L}_{n}w+\sum_{j \geq 2}\frac{1}{j!}f_{j}^{2}(z,w,\mu),
\end{aligned}\right.
\end{equation}
where
\begin{equation}\begin{aligned}
f_{j}^{1}(z,w,\mu)&=\Psi\left(\begin{aligned}
&\left[\widetilde{F}_{j}(\Phi z_{x}+w,\mu),\beta_{\nu}^{(1)}\right] \\
&\left[\widetilde{F}_{j}(\Phi z_{x}+w,\mu),\beta_{\nu}^{(2)}\right]
\end{aligned}\right)_{\nu=n_{1}}^{\nu=n_{2}}, \\
f_{j}^{2}(z,w,\mu)&=(I_{2}-\pi)\widetilde{F}_{j}(\Phi z_{x}+w,\mu).
\end{aligned}\end{equation}

\begin{remark}
By combining with (4.18), (4.19) and (4.21), we can obtain that
\begin{equation}\begin{aligned}
f_{2}^{1}(z,w,\mu)&=\Psi\left(\begin{aligned}
&\left[\widetilde{F}_{2}(\Phi z_{x}+w,\mu),\beta_{\nu}^{(1)}\right] \\
&\left[\widetilde{F}_{2}(\Phi z_{x}+w,\mu),\beta_{\nu}^{(2)}\right]
\end{aligned}\right)_{\nu=n_{1}}^{\nu=n_{2}} \\
&=\Psi\left(\begin{aligned}
&\left[F_{2}^{d}(\Phi z_{x}+w,\mu)+F_{2}(\Phi z_{x}+w),\beta_{\nu}^{(1)}\right] \\
&\left[F_{2}^{d}(\Phi z_{x}+w,\mu)+F_{2}(\Phi z_{x}+w),\beta_{\nu}^{(2)}\right]
\end{aligned}\right)_{\nu=n_{1}}^{\nu=n_{2}}, \\
f_{3}^{1}(z,w,\mu)&=\Psi\left(\begin{aligned}
&\left[\widetilde{F}_{3}(\Phi z_{x}+w,\mu),\beta_{\nu}^{(1)}\right] \\
&\left[\widetilde{F}_{3}(\Phi z_{x}+w,\mu),\beta_{\nu}^{(2)}\right]
\end{aligned}\right)_{\nu=n_{1}}^{\nu=n_{2}} \\
&=\Psi\left(\begin{aligned}
&\left[F_{3}^{d}(\Phi z_{x}+w,\mu)+F_{3}(\Phi z_{x}+w),\beta_{\nu}^{(1)}\right] \\
&\left[F_{3}^{d}(\Phi z_{x}+w,\mu)+F_{3}(\Phi z_{x}+w),\beta_{\nu}^{(2)}\right]
\end{aligned}\right)_{\nu=n_{1}}^{\nu=n_{2}}, \\
f_{j}^{1}(z,w,\mu)&=\Psi\left(\begin{aligned}
&\left[\widetilde{F}_{j}(\Phi z_{x}+w,\mu),\beta_{\nu}^{(1)}\right] \\
&\left[\widetilde{F}_{j}(\Phi z_{x}+w,\mu),\beta_{\nu}^{(2)}\right]
\end{aligned}\right)_{\nu=n_{1}}^{\nu=n_{2}}=\Psi\left(\begin{aligned}
&\left[F_{j}(\Phi z_{x}+w),\beta_{\nu}^{(1)}\right] \\
&\left[F_{j}(\Phi z_{x}+w),\beta_{\nu}^{(2)}\right]
\end{aligned}\right)_{\nu=n_{1}}^{\nu=n_{2}},~j \geq 4, \\
f_{2}^{2}(z,w,\mu)&=(I_{2}-\pi)\widetilde{F}_{2}(\Phi z_{x}+w,\mu)=(I_{2}-\pi)\left(F_{2}^{d}(\Phi z_{x}+w,\mu)+F_{2}(\Phi z_{x}+w)\right), \\
f_{3}^{2}(z,w,\mu)&=(I_{2}-\pi)\widetilde{F}_{3}(\Phi z_{x}+w,\mu)=(I_{2}-\pi)\left(F_{3}^{d}(\Phi z_{x}+w,\mu)+F_{3}(\Phi z_{x}+w)\right), \\
f_{j}^{2}(z,w,\mu)&=(I_{2}-\pi)\widetilde{F}_{2}(\Phi z_{x}+w,\mu)=(I_{2}-\pi)F_{j}(\Phi z_{x}+w),~j \geq 4.
\end{aligned}\end{equation}
\end{remark}

After a recursive transformation of variables of the form
\begin{equation}
(z,w)=(\widetilde{z},\widetilde{w})+\frac{1}{j!}\left(U_{j}^{1}(\widetilde{z},\mu),U_{j}^{2}(\widetilde{z},\mu)\right),~j \geq 2,
\end{equation}
where $z, \widetilde{z} \in \mathbb{R}^{2}$, $w, \widetilde{w} \in \mathcal{Q}$ and $U_{j}^{1}: \mathbb{R}^{4} \rightarrow \mathbb{R}^{2}$, $U_{j}^{2}: \mathbb{R}^{4} \rightarrow \mathcal{Q}$ are homogeneous polynomials of degree $j$ in $\widetilde{z}$ and $\mu$. From (4.23), and by dropping the tilde for simplicity of notation, the system (4.20) is transformed into the normal form
\begin{equation}
\left\{\begin{aligned}
&\dot{z}=Bz+\sum_{j \geq 2}\frac{1}{j!}g_{j}^{1}(z,w,\mu), \\
&\dot{w}=\mathcal{L}_{n}w+\sum_{j \geq 2}\frac{1}{j!}g_{j}^{2}(z,w,\mu),
\end{aligned}\right.
\end{equation}
where $g_{j}^{1}$ and $g_{j}^{2}$ are given by
\begin{equation}
g_{j}^{1}=\widetilde{f}_{j}^{1}-M_{j}^{1}U_{j}^{1},~g_{j}^{2}=\widetilde{f}_{j}^{2}-M_{j}^{2}U_{j}^{2}
\end{equation}
with $\widetilde{f}_{j}^{1}$ and $\widetilde{f}_{j}^{2}$ are the terms of order $j$ obtained after the changes of variables in previous step, and
\begin{equation}
M_{j}^{1}U_{j}^{1}=D_{z}U_{j}^{1}(z,\mu)Bz-BU_{j}^{1}(z,\mu),~M_{j}^{2}U_{j}^{2}=D_{z}U_{j}^{2}(z,\mu)Bz-\mathcal{L}_{n}U_{j}^{2}(z,\mu).
\end{equation}
By referring to \cite{lv29}, a locally center manifold for (4.24) satisfies $w=0$, and the flow on it is given by the two-dimensional ODE
\begin{equation}
\dot{z}=Bz+\sum_{j \geq 2}\frac{1}{j!}g_{j}^{1}(z,0,\mu),
\end{equation}
which is the normal form as in the usual sense for ODE. Furthermore, from (4.25), we have
\begin{equation}
g_{2}^{1}(z,0,\mu)=\operatorname{Proj}_{\operatorname{Ker}(M_{2}^{1})}f_{2}^{1}(z,0,\mu)
\end{equation}
and
\begin{equation}\begin{aligned}
g_{3}^{1}(z,0,\mu)&=\operatorname{Proj}_{\operatorname{Ker}(M_{3}^{1})}\widetilde{f}_{3}^{1}(z,0,\mu) \\
&=\operatorname{Proj}_{S_{1}}\widetilde{f}_{3}^{1}(z,0,0)+O(|z|^{2}|\mu|)+O(|z||\mu|^{2}),
\end{aligned}\end{equation}
where $\operatorname{Proj}_{p}(q)$ represents the projection of $q$ on $p$, $\widetilde{f}_{3}^{1}(z,0,\mu)$ is vector and its element is the cubic polynomial of $(z,\mu)$ after the variable transformation of (4.23). Furthermore, we have
\begin{equation}\begin{aligned}
\operatorname{Ker}(M_{2}^{1})&=\operatorname{span}\left\{z_{1}^{2}e_{1},z_{1}z_{2}e_{1},z_{1}\mu_{1}e_{1},z_{1}\mu_{2}e_{1},z_{2}^{2}e_{1},z_{2}\mu_{1}e_{1},z_{2}\mu_{2}e_{1},\mu_{1}^{2}e_{1},\mu_{1}\mu_{2}e_{1},\mu_{2}^{2}e_{1},z_{1}^{2}e_{2},\right. \\
&\left.z_{1}z_{2}e_{2},z_{1}\mu_{1}e_{2},z_{1}\mu_{2}e_{2},z_{2}^{2}e_{2},z_{2}\mu_{1}e_{2}, z_{2}\mu_{2}e_{2},\mu_{1}^{2}e_{2},\mu_{1}\mu_{2}e_{2},\mu_{2}^{2}e_{2}\right\}, \\
\operatorname{Ker}(M_{3}^{1})&=\operatorname{span}\left\{z_{1}^{3}e_{1},z_{1}^{2}z_{2}e_{1},z_{1}^{2}\mu_{1}e_{1},z_{1}^{2}\mu_{2}e_{1},z_{1}z_{2}^{2} e_{1},z_{1}z_{2}\mu_{1}e_{1},z_{1}z_{2}\mu_{2}e_{1},z_{1}\mu_{1}^{2}e_{1},z_{1}\mu_{1}\mu_{2}e_{1},\right. \\
&\left.z_{1}\mu_{2}^{2}e_{1},z_{2}^{3}e_{1},z_{2}^{2}\mu_{1}e_{1},z_{2}^{2}\mu_{2}e_{1},z_{2}\mu_{1}^{2}e_{1},z_{2}\mu_{1}\mu_{2}e_{1}, z_{2}\mu_{2}^{2}e_{1},\mu_{1}^{3}e_{1},\mu_{1}^{2}\mu_{2}e_{1},\mu_{1}\mu_{2}^{2}e_{1},\mu_{2}^{3}e_{1},\right. \\
&\left.z_{1}^{3}e_{2},z_{1}^{2}z_{2}e_{2},z_{1}^{2}\mu_{1}e_{2},z_{1}^{2}\mu_{2}e_{2},z_{1}z_{2}^{2} e_{2},z_{1}z_{2}\mu_{1}e_{2},z_{1}z_{2}\mu_{2}e_{2},z_{1}\mu_{1}^{2}e_{2},z_{1}\mu_{1}\mu_{2}e_{2},z_{1}\mu_{2}^{2}e_{2},\right. \\
&\left.z_{2}^{3}e_{2},z_{2}^{2}\mu_{1}e_{2},z_{2}^{2}\mu_{2}e_{2},z_{2}\mu_{1}^{2}e_{2},z_{2}\mu_{1}\mu_{2}e_{2}, z_{2}\mu_{2}^{2}e_{2},\mu_{1}^{3}e_{2},\mu_{1}^{2}\mu_{2}e_{2},\mu_{1}\mu_{2}^{2}e_{2},\mu_{2}^{3}e_{2}\right\}
\end{aligned}\end{equation}
and
\begin{equation}
S_{1}=\operatorname{span}\left\{z_{1}^{3}e_{1},z_{1}^{2}z_{2}e_{1},z_{1}z_{2}^{2}e_{1},z_{2}^{3}e_{1},z_{1}^{3}e_{2},z_{1}^{2}z_{2}e_{2},z_{1}z_{2}^{2}e_{2},z_{2}^{3}e_{2}\right\}.
\end{equation}

\begin{remark}
Since $B=\operatorname{diag}\left\{0,0\right\}$, then from (4.26), we have
\begin{equation*}\begin{aligned}
&M_{j}^{1}U_{j}^{1}=D_{z}U_{j}^{1}(z,\mu)Bz-BU_{j}^{1}(z,\mu)=\left(\begin{array}{c}
0 \\
0
\end{array}\right), \\
&M_{j}^{2}U_{j}^{2}=D_{z}U_{j}^{2}(z,\mu)Bz-\mathcal{L}_{n}U_{j}^{2}(z,\mu)=-\mathcal{L}_{n}U_{j}^{2}(z,\mu).
\end{aligned}\end{equation*}
Thus we can obtain that $\operatorname{Im}(M_{j}^{1})=\left\{0\right\}$ for $j \geq 2$ and $j \in \mathbb{N}$. Furthermore, consider that
\begin{equation*}
V_{j}^{2+2}(\mathbb{R}^{2})=\operatorname{span}\left\{\left(\begin{array}{c}
z^{q}\mu^{l} \\
0
\end{array}\right),\left(\begin{array}{c}
0 \\
z^{q}\mu^{l}
\end{array}\right): |q|+|l|=j \in \mathbb{N}, q, l \in \mathbb{N}_{0}^{2}\right\},
\end{equation*}
and from the decomposition
\begin{equation*}
V_{j}^{2+2}(\mathbb{R}^{2})=\operatorname{Im}(M_{j}^{1}) \oplus \operatorname{Ker}(M_{j}^{1}),~j \geq 2,~j \in \mathbb{N},
\end{equation*}
we have
\begin{equation*}
\operatorname{Ker}(M_{j}^{1})=V_{j}^{2+2}(\mathbb{R}^{2}),~j \geq 2,~j \in \mathbb{N}.
\end{equation*}
\end{remark}

By combining with (4.28) and (4.29), we can calculate $g_{j}^{1}(z,0,\mu),~j=2,3$ step by step, and the detailed derivation of the normal form for the Turing-Turing bifurcation in model (1.1), which includes all symbolic computations, has been deposited on arXiv at [arXiv:2506.09360].

\subsection{The normal form of Turing-Turing bifurcation for model (1.1)}

By combining with (4.27), (4.28), (4.29), and [arXiv:2506.09360], the third-order truncated normal form of Turing-Turing bifurcation for the model (1.1) can be written as
\begin{equation}\begin{aligned}
\dot{z}&=Bz+\left(\begin{array}{c}
B_{2000}^{(1)}z_{1}^{2}+B_{1100}^{(1)}z_{1}z_{2}+(B_{1010}^{(1)}\mu_{1}+B_{1001}^{(1)}\mu_{2})z_{1}+B_{0200}^{(1)}z_{2}^{2}+(B_{0110}^{(1)}\mu_{1}+B_{0101}^{(1)}\mu_{2})z_{2} \\
B_{2000}^{(2)}z_{1}^{2}+B_{1100}^{(2)}z_{1}z_{2}+(B_{1010}^{(2)}\mu_{1}+B_{1001}^{(2)}\mu_{2})z_{1}+B_{0200}^{(2)}z_{2}^{2}+(B_{0110}^{(2)}\mu_{1}+B_{0101}^{(2)}\mu_{2})z_{2}
\end{array}\right) \\
&+\left(\begin{array}{c}
B_{3000}^{(1)}z_{1}^{3}+B_{2100}^{(1)}z_{1}^{2}z_{2}+B_{1200}^{(1)}z_{1}z_{2}^{2}+B_{0300}^{(1)}z_{2}^{3} \\
B_{3000}^{(2)}z_{1}^{3}+B_{2100}^{(2)}z_{1}^{2}z_{2}+B_{1200}^{(2)}z_{1}z_{2}^{2}+B_{0300}^{(2)}z_{2}^{3}
\end{array}\right) \\
&+O(|z|^{2}|\mu|)+O(|z||\mu|^{2}),
\end{aligned}\end{equation}
where
\begin{equation}\left\{\begin{aligned}
B_{3000}^{(1)}&=C_{3000}^{(1)}+\frac{3}{2}(D_{3000}^{(1)}+E_{3000}^{(1)}+F_{3000}^{(1)}+G_{3000}^{(1)}), \\
B_{2100}^{(1)}&=C_{2100}^{(1)}+\frac{3}{2}(D_{2100}^{(1)}+E_{2100}^{(1)}+F_{2100}^{(1)}+G_{2100}^{(1)}), \\
B_{1200}^{(1)}&=C_{1200}^{(1)}+\frac{3}{2}(D_{1200}^{(1)}+E_{1200}^{(1)}+F_{1200}^{(1)}+G_{1200}^{(1)}), \\
B_{0300}^{(1)}&=C_{0300}^{(1)}+\frac{3}{2}(D_{0300}^{(1)}+E_{0300}^{(1)}+F_{0300}^{(1)}+G_{0300}^{(1)}), \\
B_{3000}^{(2)}&=C_{3000}^{(2)}+\frac{3}{2}(D_{3000}^{(2)}+E_{3000}^{(2)}+F_{3000}^{(2)}+G_{3000}^{(2)}), \\
B_{2100}^{(2)}&=C_{2100}^{(2)}+\frac{3}{2}(D_{2100}^{(2)}+E_{2100}^{(2)}+F_{2100}^{(2)}+G_{2100}^{(2)}), \\
B_{1200}^{(2)}&=C_{1200}^{(2)}+\frac{3}{2}(D_{1200}^{(2)}+E_{1200}^{(2)}+F_{1200}^{(2)}+G_{1200}^{(2)}), \\
B_{0300}^{(2)}&=C_{0300}^{(2)}+\frac{3}{2}(D_{0300}^{(2)}+E_{0300}^{(2)}+F_{0300}^{(2)}+G_{0300}^{(2)}).
\end{aligned}\right.\end{equation}

\section{Numerical simulations}
\label{sec:5}

For fixed $\ell=3$, $a=0.4$, $r_{0}=0.5$, $c=1$, $b_{1}=1$, $b_{2}=0.98$, $m_{1}=0.6$, $d_{v}=0.6$, $\beta=0.05$, it follows from (3.13) that $\overline{n}=3$, and (3.7) and (3.8) are satisfied. Thus, when $0<d_{u}<d_{u}^{\overline{n}}\doteq 0.0539$, by choosing $n_{1}=3$, according to Theorem \ref{thm:3.5}, the model (1.1) undergoes mode-$(3,4)$ Turing-Turing bifurcation when $d_{u}=d_{u}^{3,4}\doteq 0.0253$ and $\alpha=\alpha_{*}(0.0253)\doteq 0.5527$. Furthermore, according to Assumption \ref{asm:4.1}, we can let $d_{u}^{*}=d_{u}^{3,4}\doteq 0.0253$, $\alpha^{*}=\alpha_{*}(0.0253)\doteq 0.5527$, and $n_{1}=3$, $n_{2}=4$, then by combining with (4.32) and (4.33), and by using the matlab software for auxiliary calculation, we obtain that
\begin{equation}\begin{aligned}
&B_{2000}^{(1)}=0,~B_{1100}^{(1)}=0,~B_{1010}^{(1)}\doteq -0.5638,~B_{1001}^{(1)}\doteq -0.1041, \\
&B_{0200}^{(1)}=0,~B_{0110}^{(1)}=0,~B_{0101}^{(1)}=0, \\
&B_{2000}^{(2)}=0,~B_{1100}^{(2)}=0,~B_{1010}^{(2)}=0,~B_{1001}^{(2)}=0,~B_{0200}^{(2)}=0, \\
&B_{0110}^{(2)}\doteq -0.9433,~B_{0101}^{(2)}\doteq -0.1119, \\
&B_{3000}^{(1)}\doteq -0.0027,~B_{2100}^{(1)}=0,~B_{1200}^{(1)}\doteq -0.1098,~B_{0300}^{(1)}=0, \\
&B_{3000}^{(2)}=0,~B_{2100}^{(2)}\doteq 0.0490,~B_{1200}^{(2)}=0,~B_{0300}^{(2)}\doteq -0.1276.
\end{aligned}\end{equation}
From (5.1), we can see that $B_{2000}^{(1)}=0$, $B_{1100}^{(1)}=0$, $B_{0200}^{(1)}=0$, $B_{0110}^{(1)}=0$, $B_{0101}^{(1)}=0$, $B_{2000}^{(2)}=0$, $B_{1100}^{(2)}=0$, $B_{1010}^{(2)}=0$, $B_{1001}^{(2)}=0$, $B_{0200}^{(2)}=0$, $B_{2100}^{(1)}=0$, $B_{0300}^{(1)}=0$, $B_{3000}^{(2)}=0$, $B_{1200}^{(2)}=0$, thus the normal form (4.32) turns into
\begin{equation}\begin{aligned}
\dot{z}&=Bz+\left(\begin{array}{c}
(B_{1010}^{(1)}\mu_{1}+B_{1001}^{(1)}\mu_{2})z_{1} \\
(B_{0110}^{(2)}\mu_{1}+B_{0101}^{(2)}\mu_{2})z_{2}
\end{array}\right)+\left(\begin{array}{c}
B_{3000}^{(1)}z_{1}^{3}+B_{1200}^{(1)}z_{1}z_{2}^{2} \\
B_{2100}^{(2)}z_{1}^{2}z_{2}+B_{0300}^{(2)}z_{2}^{3}
\end{array}\right) \\
&+O(|z|^{2}|\mu|)+O(|z||\mu|^{2}).
\end{aligned}\end{equation}
By noticing that $z=(z_{1},z_{2})^{T}$ and $B=\operatorname{diag}\left\{0,0\right\}$, then (5.2) can be reformulated as
\begin{equation}\left\{\begin{aligned}
\frac{\mathrm{d}z_{1}}{\mathrm{d}t}&=\alpha_{1}(\mu)z_{1}+\kappa_{11}z_{1}^{3}+\kappa_{12}z_{1}z_{2}^{2}, \\
\frac{\mathrm{d}z_{2}}{\mathrm{d}t}&=\alpha_{2}(\mu)z_{2}+\kappa_{21}z_{1}^{2}z_{2}+\kappa_{22}z_{2}^{3},
\end{aligned}\right.\end{equation}
where
\begin{equation}\begin{aligned}
&\alpha_{1}(\mu):=B_{1010}^{(1)}\mu_{1}+B_{1001}^{(1)}\mu_{2},~\alpha_{2}(\mu):=B_{0110}^{(2)}\mu_{1}+B_{0101}^{(2)}\mu_{2}, \\
&\kappa_{11}:=B_{3000}^{(1)},~\kappa_{12}:=B_{1200}^{(1)},~\kappa_{21}:=B_{2100}^{(2)},~\kappa_{22}:=B_{0300}^{(2)}.
\end{aligned}\end{equation}

\begin{remark}
In this paper, we mainly study the simple case of $\kappa_{11}\kappa_{22}>0$ for the third-order truncated normal form (5.3). For the difficult case of $\kappa_{11}\kappa_{22}<0$, the higher order normal form is needed to be calculated.
\end{remark}

By re-scaling the variables with
\begin{equation}
\widetilde{z_{1}}=\sqrt{\left|\kappa_{11}\right|}z_{1},~\widetilde{z_{2}}=\sqrt{\left|\kappa_{22}\right|}z_{2},~\widetilde{t}=\frac{t}{\widetilde{\varepsilon}},~\widetilde{\varepsilon}=\operatorname{sign}(\kappa_{11}),
\end{equation}
the normal form (5.3) can be rewritten as the equivalent planar system
\begin{equation}\left\{\begin{aligned}
&\frac{\mathrm{d}\widetilde{z_{1}}}{\mathrm{d}\widetilde{t}}=\widetilde{z_{1}}\left(\widetilde{\varepsilon}\alpha_{1}(\mu)+\widetilde{z_{1}}^{2}+\widetilde{b}\widetilde{z_{2}}^{2}\right), \\
&\frac{\mathrm{d}\widetilde{z_{2}}}{\mathrm{d}\widetilde{t}}=\widetilde{z_{2}}\left(\widetilde{\varepsilon}\alpha_{2}(\mu)+\widetilde{c}\widetilde{z_{1}}^{2}+\widetilde{d}\widetilde{z_{2}}^{2}\right),
\end{aligned}\right.\end{equation}
where $\operatorname{sign}(.)$ represents the sign function, and
\begin{equation}\begin{aligned}
&\widetilde{b}=\widetilde{\varepsilon}\frac{\kappa_{12}}{\left|\kappa_{22}\right|},~\widetilde{c}=\widetilde{\varepsilon}\frac{\kappa_{21}}{\left|\kappa_{11}\right|},~\widetilde{d}=\widetilde{\varepsilon}\frac{\kappa_{22}}{\left|\kappa_{22}\right|}=\pm 1.
\end{aligned}\end{equation}
For the planar system (5.6), there are equilibrium points
\begin{equation}\begin{aligned}
&E_{0}=(0,0), \text { for all } \mu_{1}, \mu_{2}, \\
&E_{1}^{\pm}=\left(\pm\sqrt{-\widetilde{\varepsilon}\alpha_{1}(\mu)},0\right), \text { for } \widetilde{\varepsilon}\alpha_{1}(\mu)<0, \\
&E_{2}^{\pm}=\left(0,\pm\sqrt{-\frac{\widetilde{\varepsilon}\alpha_{2}(\mu)}{\widetilde{d}}}\right), \text { for } \frac{\widetilde{\varepsilon}\alpha_{2}(\mu)}{\widetilde{d}}<0, \\
&E_{3}^{\pm\pm}=\left(\pm\sqrt{\frac{\widetilde{b}\widetilde{\varepsilon}\alpha_{2}(\mu)-\widetilde{d}\widetilde{\varepsilon}\alpha_{1}(\mu)}{\widetilde{d}-\widetilde{b}\widetilde{c}}},\pm\sqrt{\frac{\widetilde{c}\widetilde{\varepsilon}\alpha_{1}(\mu)-\widetilde{\varepsilon}\alpha_{2}(\mu)}{\widetilde{d}-\widetilde{b}\widetilde{c}}}\right), \\
&\text { for } \frac{\widetilde{b}\widetilde{\varepsilon}\alpha_{2}(\mu)-\widetilde{d}\widetilde{\varepsilon}\alpha_{1}(\mu)}{\widetilde{d}-\widetilde{b}\widetilde{c}}>0,~\frac{\widetilde{c}\widetilde{\varepsilon}\alpha_{1}(\mu)-\widetilde{\varepsilon}\alpha_{2}(\mu)}{\widetilde{d}-\widetilde{b}\widetilde{c}}>0.
\end{aligned}\end{equation}
In what follows, we denote
\begin{equation*}\begin{aligned}
&r_{1}:=\sqrt{-\frac{\widetilde{\varepsilon}\alpha_{1}(\mu)}{|\kappa_{11}|}},~r_{2}:=\sqrt{-\frac{\widetilde{\varepsilon}\alpha_{2}(\mu)}{\widetilde{d}|\kappa_{22}|}}, \\
&r_{3}:=\sqrt{\frac{\widetilde{b}\widetilde{\varepsilon}\alpha_{2}(\mu)-\widetilde{d}\widetilde{\varepsilon}\alpha_{1}(\mu)}{(\widetilde{d}-\widetilde{b}\widetilde{c})|\kappa_{11}|}},~r_{4}:=\sqrt{\frac{\widetilde{c}\widetilde{\varepsilon}\alpha_{1}(\mu)-\widetilde{\varepsilon}\alpha_{2}(\mu)}{(\widetilde{d}-\widetilde{b}\widetilde{c})|\kappa_{22}|}}.
\end{aligned}\end{equation*}

Depending on the different signs of $\widetilde{b}$, $\widetilde{c}$, $\widetilde{d}$ and $\widetilde{d}-\widetilde{b}\widetilde{c}$, there are 12 distinct types of unfolding for the planar system (5.6), see Chapter 7.5 in \cite{lv50} for detail. Here, we claim that the solutions of the original model (1.1) restrict on the center manifold is homeomorphic to $W(t)=z_{1}(t)\phi_{n_{1}}\gamma_{n_{1}}(x)+z_{2}(t)\phi_{n_{2}}\gamma_{n_{2}}(x)$. Based on the center manifold theorem given by in Chapter 6 in \cite{lv51} and in Chapter 5 in \cite{lv52}, the solutions of the original model (1.1), which the initial functions are in a neighborhood of the positive constant steady state $E_{*}(u_{*},v_{*})$ in $X$, are exponentially convergent to the homeomorphism of the attractors of the solutions restrict on the center manifold (5.6), that is
\begin{equation*}
U(t)\approx \mathcal{H}(W(t))+O(\mathrm{e}^{-\vartheta t}), \text { as } t \rightarrow \infty,
\end{equation*}
where $\mathcal{H}(\cdot)$ is an isomorphic mapping and $\vartheta>0$. By further analysis, the corresponding relationships between the equilibrium points of planar system (5.6) and the solutions of original model (1.1) are given in Table 1.

\begin{table}[!htbp]
	\centering
	{Table 1: Corresponding relationships between the equilibrium points of planar system (5.6) and the solutions of original model (1.1)}
	\label{table:1}
\begin{tabular}{cc}
\hline \text{Equilibrium points} & \text{Solutions of original model (1.1)} \\
 of planar system (5.6) & \\
\hline $E_{0}$ & \text{Positive constant steady state} \\
\hline
$E_{1}^{\pm}$ & \text{Non-constant steady states with} \\
 &  $r_{1}\phi_{n_{1}}\cos(2x/3)$-type spatial distribution \\
\hline
$E_{2}^{\pm}$ & \text{Non-constant steady states with} \\
 &  $r_{2}\phi_{n_{2}}\cos(x)$-type spatial distribution \\
\hline
$E_{3}^{\pm\pm}$ & \text{Non-constant steady states with} \\
 & $r_{3}\phi_{n_{1}}\cos(2x/3)+r_{4}\phi_{n_{2}}\cos(x)$-type spatial distribution \\
\hline
\end{tabular}
\end{table}

More precisely, by combining with (5.1), (5.4), (5.5), (5.6) and (5.7), we have
\begin{equation}\left\{\begin{aligned}
&\frac{\mathrm{d}\widetilde{z_{1}}}{\mathrm{d}\widetilde{t}}=\widetilde{z_{1}}\left(\widetilde{\varepsilon}\alpha_{1}(\mu)+\widetilde{z_{1}}^{2}+\widetilde{b}\widetilde{z_{2}}^{2}\right), \\
&\frac{\mathrm{d}\widetilde{z_{2}}}{\mathrm{d}\widetilde{t}}=\widetilde{z_{2}}\left(\widetilde{\varepsilon}\alpha_{2}(\mu)+\widetilde{c}\widetilde{z_{1}}^{2}+\widetilde{d}\widetilde{z_{2}}^{2}\right),
\end{aligned}\right.\end{equation}
where
\begin{equation*}\begin{aligned}
&\alpha_{1}(\mu)\doteq -0.5638\mu_{1}-0.1041\mu_{2}, \\
&\alpha_{2}(\mu)\doteq -0.9433\mu_{1}-0.1119\mu_{2}, \\
&\widetilde{\varepsilon}=-1,~\widetilde{b}\doteq 0.8599,~\widetilde{c}\doteq -18.2561,~\widetilde{d}=1, \\
&\widetilde{b}\widetilde{\varepsilon}\alpha_{2}(\mu)-\widetilde{d}\widetilde{\varepsilon}\alpha_{1}(\mu)\doteq 0.2474\mu_{1}-0.0078\mu_{2}, \\
&\widetilde{c}\widetilde{\varepsilon}\alpha_{1}(\mu)-\widetilde{\varepsilon}\alpha_{2}(\mu)\doteq -11.2368\mu_{1}-2.0121\mu_{2}, \\
&\widetilde{d}-\widetilde{b}\widetilde{c}\doteq 16.6986.
\end{aligned}\end{equation*}
Moreover, by combining with (5.8), the system (5.9) has zero equilibrium point $E_{0}=(0,0)$ for any $\mu_{1}, \mu_{2} \in \mathbb{R}$, boundary equilibrium points
\begin{equation*}\begin{aligned}
&E_{1}^{\pm}=\left(\pm\sqrt{-0.5638\mu_{1}-0.1041\mu_{2}},0\right), \text{ for } \mu_{2}<-5.4159\mu_{1}, \\
&E_{2}^{\pm}=\left(0,\pm\sqrt{-0.9433\mu_{1}-0.1119\mu_{2}}\right), \text{ for } \mu_{2}<-8.4298\mu_{1},
\end{aligned}\end{equation*}
and interior equilibrium points
\begin{equation*}
E_{3}^{\pm\pm}=\left(\pm\sqrt{\frac{0.2474\mu_{1}-0.0078\mu_{2}}{16.6986}},\pm\sqrt{\frac{-11.2368\mu_{1}-2.0121\mu_{2}}{16.6986}}\right)
\end{equation*}
for $\mu_{2}<31.7179\mu_{1}$ and $\mu_{2}<-5.5846\mu_{1}$. According to the above four inequalities which are used to ensure the existence of the equilibrium points $E_{1}^{\pm}$, $E_{2}^{\pm}$ and $E_{3}^{\pm\pm}$, we can define the critical bifurcation lines in the $\mu_{1}-\mu_{2}$ plane as follows, i.e.,
\begin{equation*}\begin{aligned}
&CL_{1}: \mu_{2}=-5.4159\mu_{1}, \\
&CL_{2}: \mu_{2}=-8.4298\mu_{1}, \\
&CL_{3}: \mu_{2}=31.7179\mu_{1},~\mu_{1}<0, \\
&CL_{4}: \mu_{2}=-5.5846\mu_{1},~\mu_{1}>0.
\end{aligned}\end{equation*}

\begin{figure}[!htbp]
\centering
\includegraphics[width=3.5in]{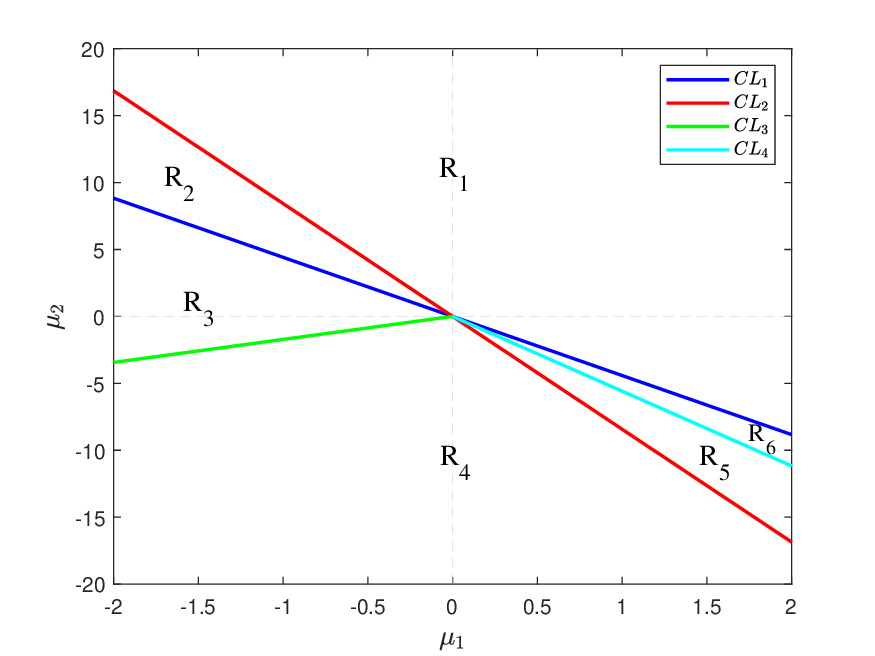}
\caption{Bifurcation set of (5.9) near the Turing-Turing bifurcation point $(d_{u}^{*},\alpha^{*})\doteq (0.0253,0.5527)$ in $\mu_{1}-\mu_{2}$ plane.}
\label{fig:2}
\end{figure}

From Fig. \ref{fig:2}, we can see that the $\mu_{1}-\mu_{2}$ plane is divided into six different regions by the above four critical bifurcation lines, and the six different regions are marked as $R_{j}(j=1,2,3,4,5,6)$. Furthermore, it is easy to obtain that the normal form (5.9) has equilibrium point $E_{0}$ in region $R_{1}$, the normal form (5.9) has equilibrium points $E_{0}$ and $E_{2}^{\pm}$ in region $R_{2}$, the normal form (5.9) has equilibrium points $E_{0}$, $E_{1}^{\pm}$ and $E_{2}^{\pm}$ in region $R_{3}$, the normal form (5.9) has equilibrium points $E_{0}$, $E_{1}^{\pm}$, $E_{2}^{\pm}$ and $E_{3}^{\pm\pm}$ in region $R_{4}$, the normal form (5.9) has equilibrium points $E_{0}$, $E_{1}^{\pm}$ and $E_{3}^{\pm}$ in region $R_{5}$, and the normal form (5.9) has equilibrium points $E_{0}$ and $E_{1}^{\pm}$ in region $R_{6}$.

\begin{figure}[!htbp]
\centering
\includegraphics[width=3in]{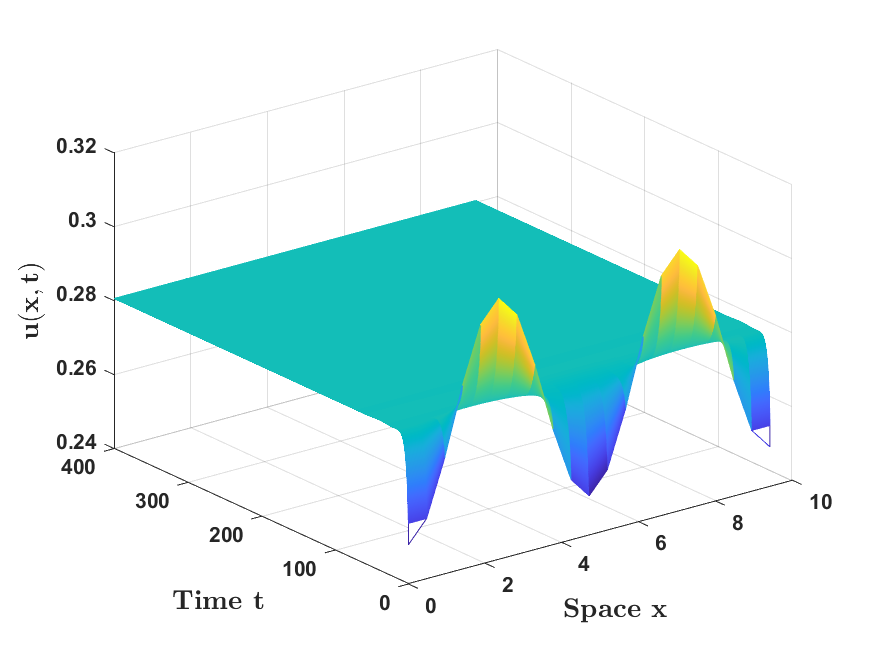}
\includegraphics[width=3in]{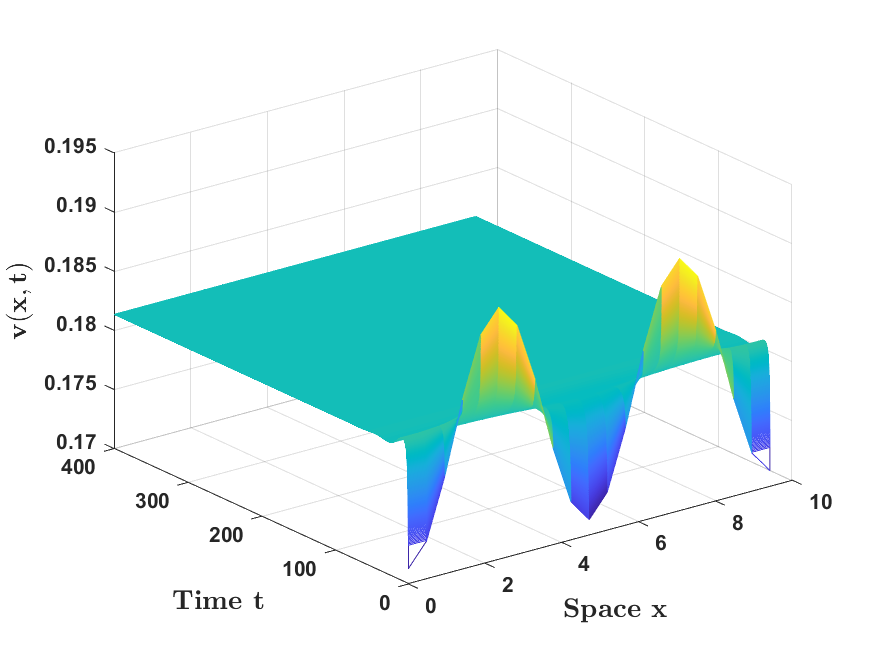} \\
\text{(a)} \hspace{7.8cm} \text{(b)} \\
\caption{For $(\mu_{1},\mu_{2})=(0.1,0.1) \in R_{1}$, the positive constant steady state $E_{*}(u_{*},v_{*})$ of the model (1.1) is locally asymptotically stable. (a) and (b) are the whole evolutionary process for $u(x,t)$ and $v(x,t)$ of the model (1.1), respectively. The initial functions are $u_{0}(x)=0.2806-0.03\cos(4x/3)$ and $v_{0}(x)=0.1813-0.01\cos(4x/3)$.}
\label{fig:3}
\end{figure}

\begin{figure}[!htbp]
\centering
\includegraphics[width=2in]{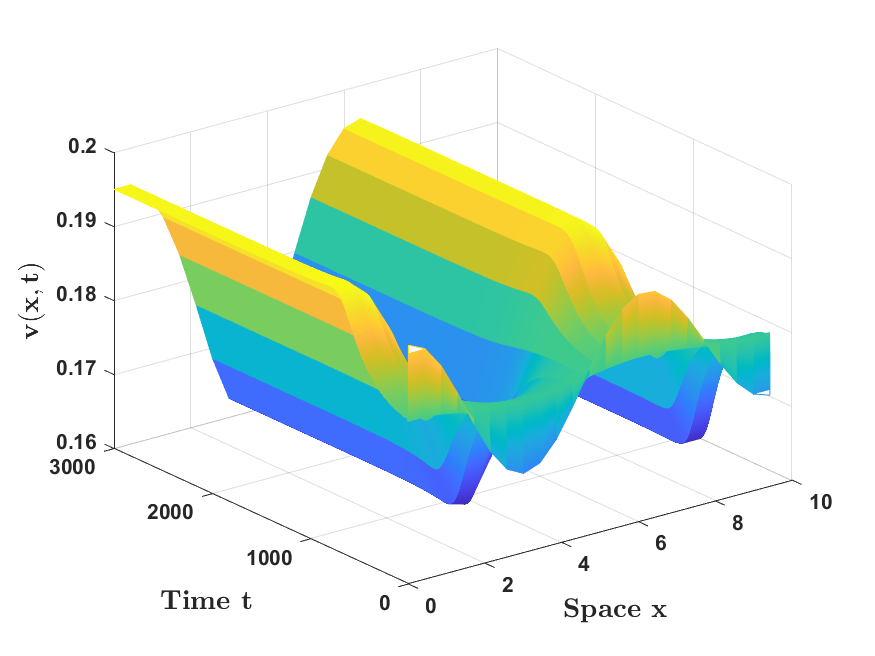}
\includegraphics[width=2in]{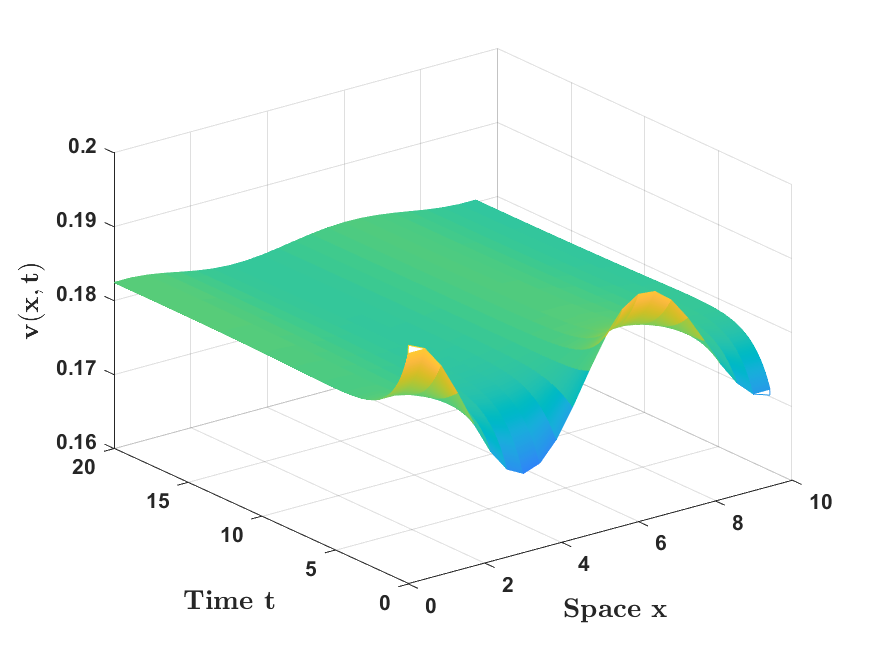}
\includegraphics[width=2in]{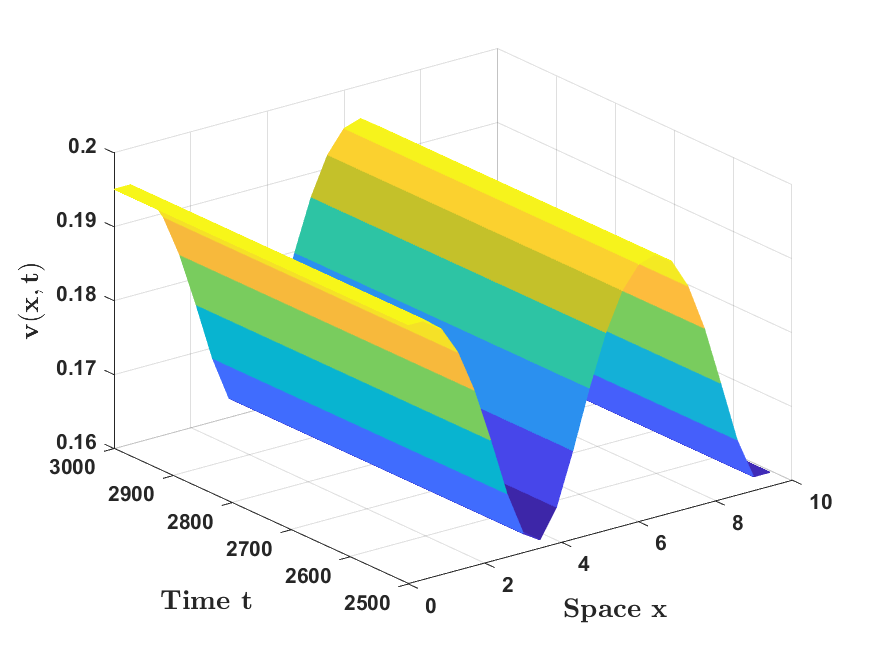} \\
\text{(a)} \hspace{4.6cm} \text{(b)} \hspace{4.6cm} \text{(c)} \\
\includegraphics[width=2in]{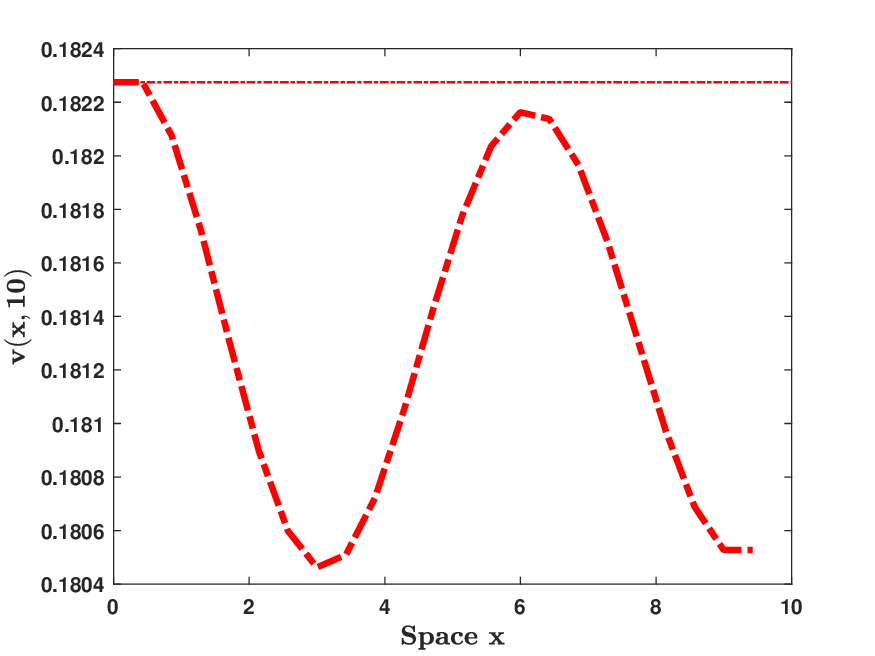}
\includegraphics[width=2in]{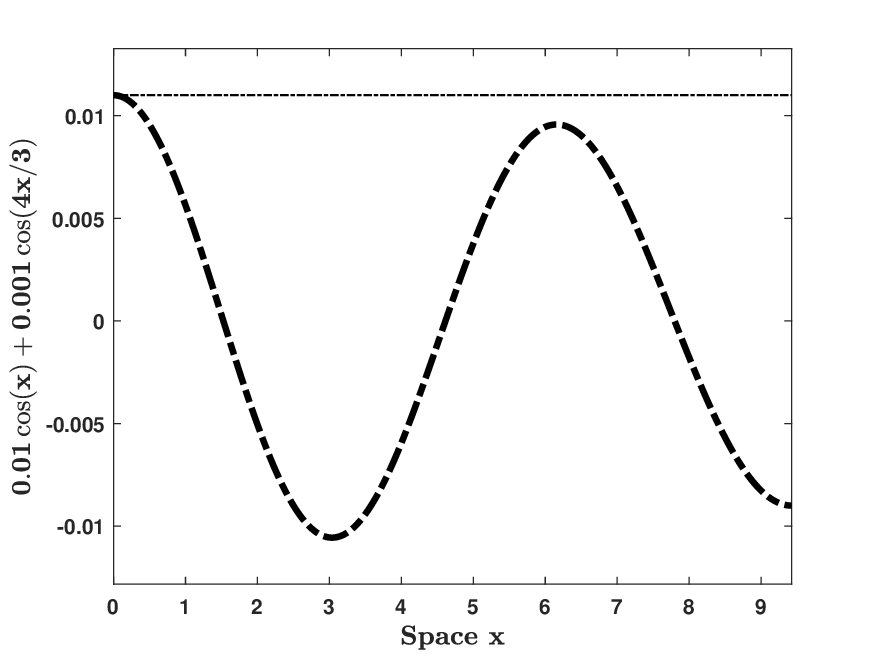}
\includegraphics[width=2in]{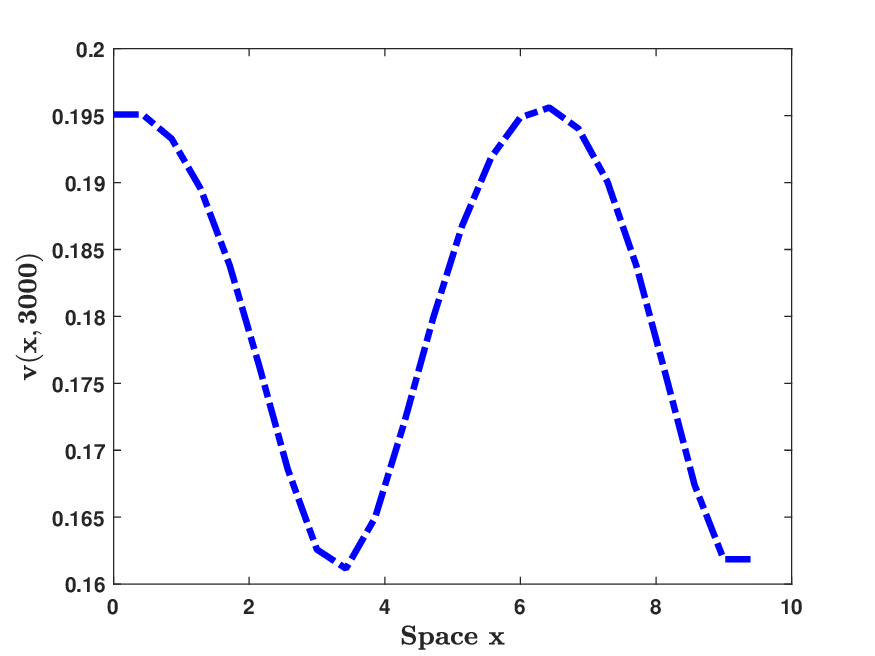} \\
\text{(d)} \hspace{4.6cm} \text{(e)} \hspace{4.6cm} \text{(f)} \\
\caption{For $(\mu_{1},\mu_{2})=(0.002,-0.1) \in R_{4}$, two stable spatially inhomogeneous steady states with $r_{1}\phi_{n_{1}}\cos(x)$-type spatial distribution exist. (a), (b) and (c) are the long-term, transient and final behaviors of $v(x,t)$, respectively, where the shape of the final behavior of $v(x,t)$ is $\cos(x)$. (d) is the concentration profile for $v(x,t)$ at $t=10$, (e) is the image of the function $0.01\cos(x)+0.001\cos(4x/3)$, (f) is the concentration profile for $v(x,t)$ at $t=3000$. The initial functions are $u_{0}(x)=0.2806+0.01\cos(x)+0.001\cos(4x/3)$ and $v_{0}(x)=0.1813+0.01\cos(x)+0.001\cos(4x/3)$.}
\label{fig:4}
\end{figure}

\begin{figure}[!htbp]
\centering
\includegraphics[width=2in]{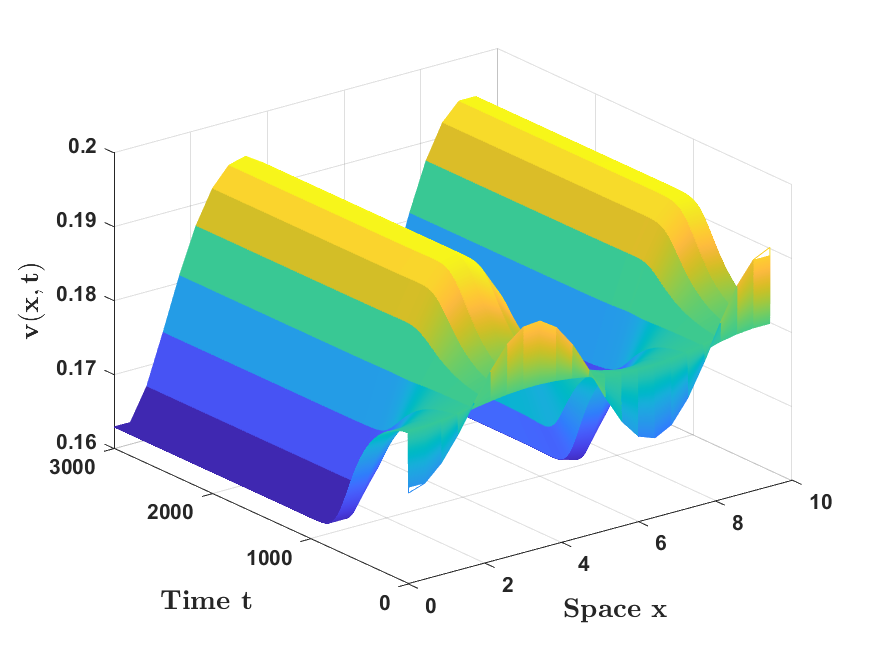}
\includegraphics[width=2in]{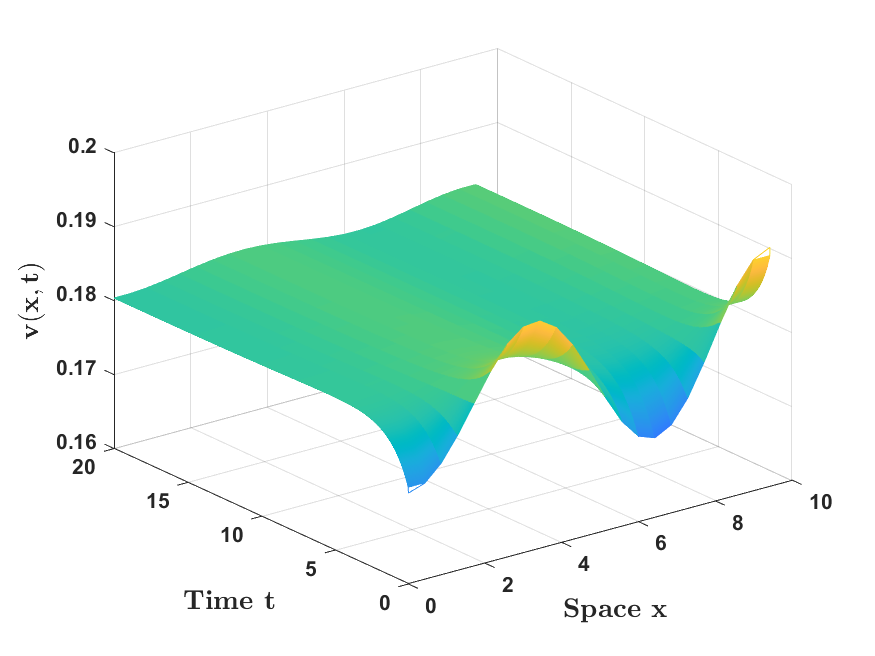}
\includegraphics[width=2in]{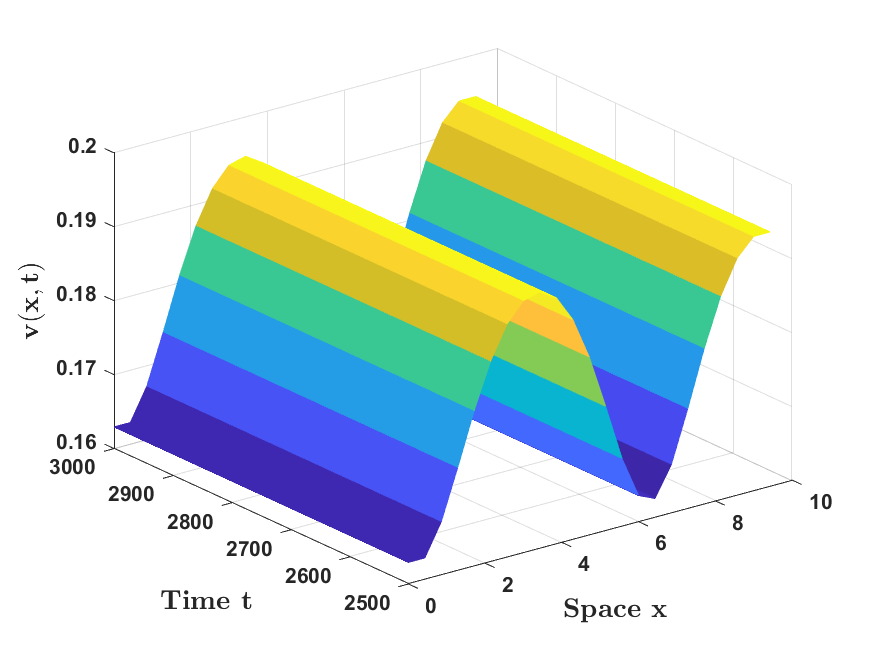} \\
\text{(a)} \hspace{4.6cm} \text{(b)} \hspace{4.6cm} \text{(c)} \\
\includegraphics[width=2in]{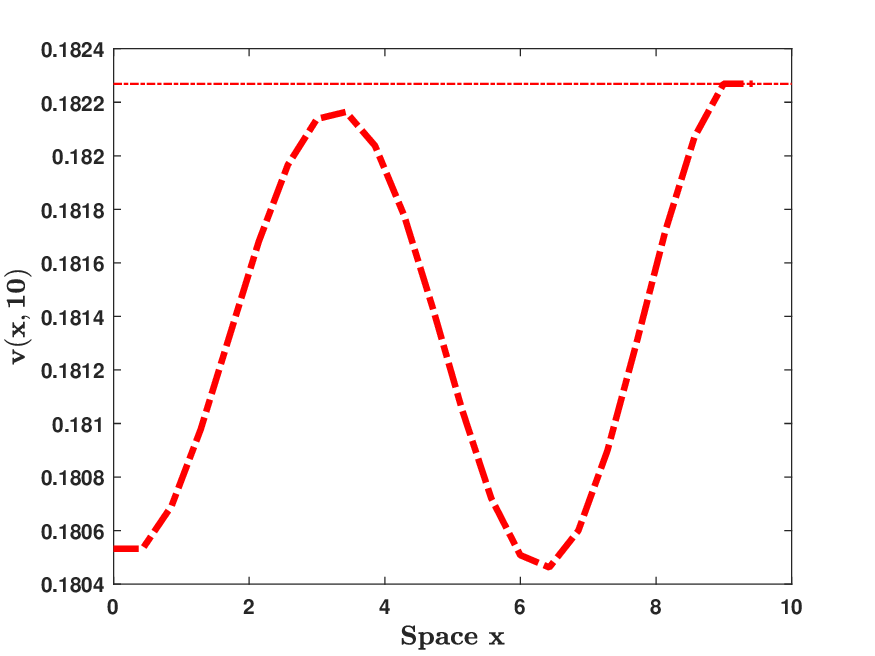}
\includegraphics[width=2in]{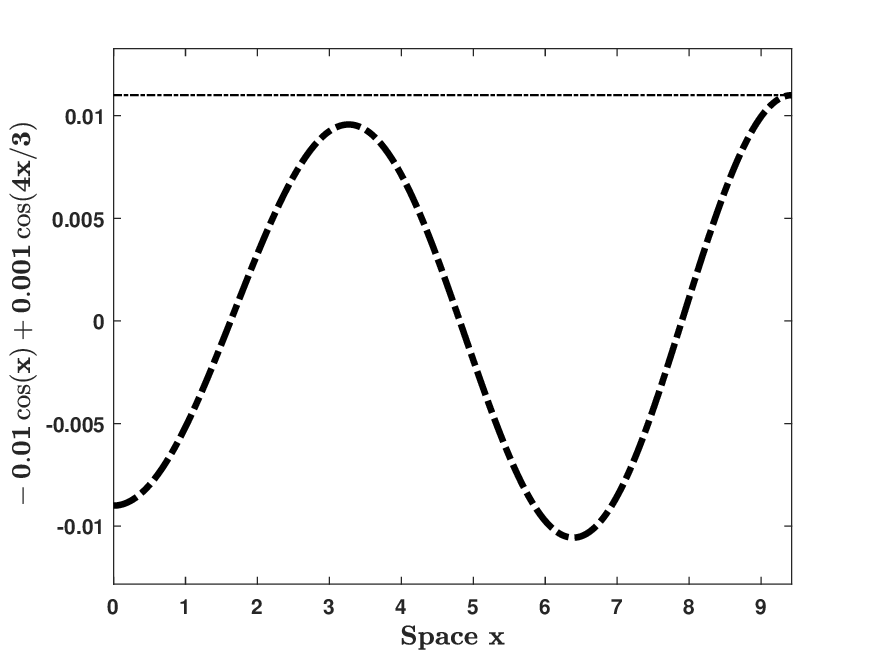}
\includegraphics[width=2in]{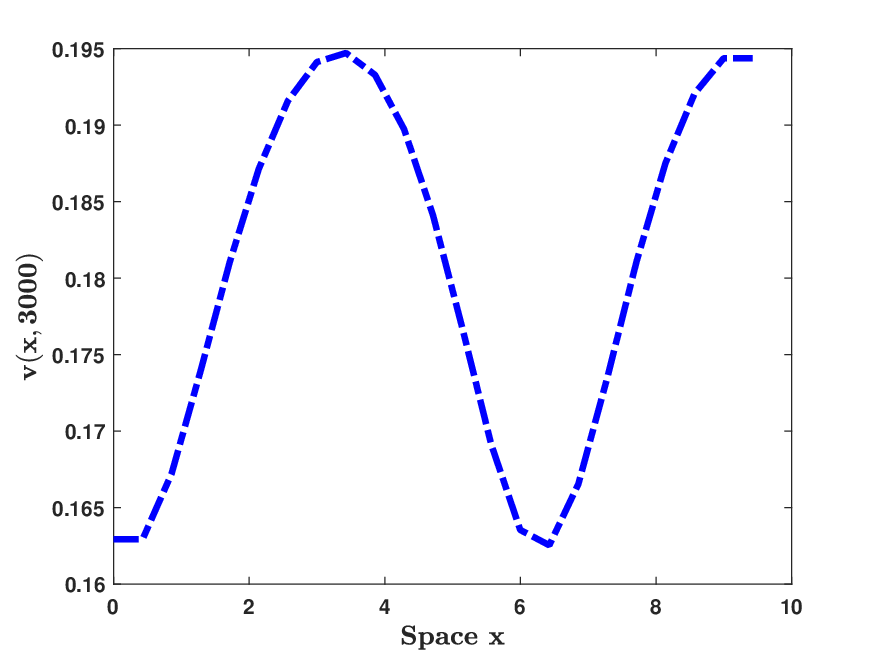} \\
\text{(d)} \hspace{4.6cm} \text{(e)} \hspace{4.6cm} \text{(f)} \\
\caption{For $(\mu_{1},\mu_{2})=(0.002,-0.1) \in R_{4}$, two stable spatially inhomogeneous steady states with $r_{1}\phi_{n_{1}}\cos(x)$-type spatial distribution exist. (a), (b) and (c) are the long-term, transient and final behaviors of $v(x,t)$, respectively, where the shape of the final behavior of $v(x,t)$ is $-\cos(x)$. (d) is the concentration profile for $v(x,t)$ at $t=10$, (e) is the image of the function $-0.01\cos(x)+0.001\cos(4x/3)$, (f) is the concentration profile for $v(x,t)$ at $t=3000$. The initial functions are $u_{0}(x)=0.2806-0.01\cos(x)+0.001\cos(4x/3)$ and $v_{0}(x)=0.1813-0.01\cos(x)+0.001\cos(4x/3)$.}
\label{fig:5}
\end{figure}

\begin{figure}[!htbp]
\centering
\includegraphics[width=2in]{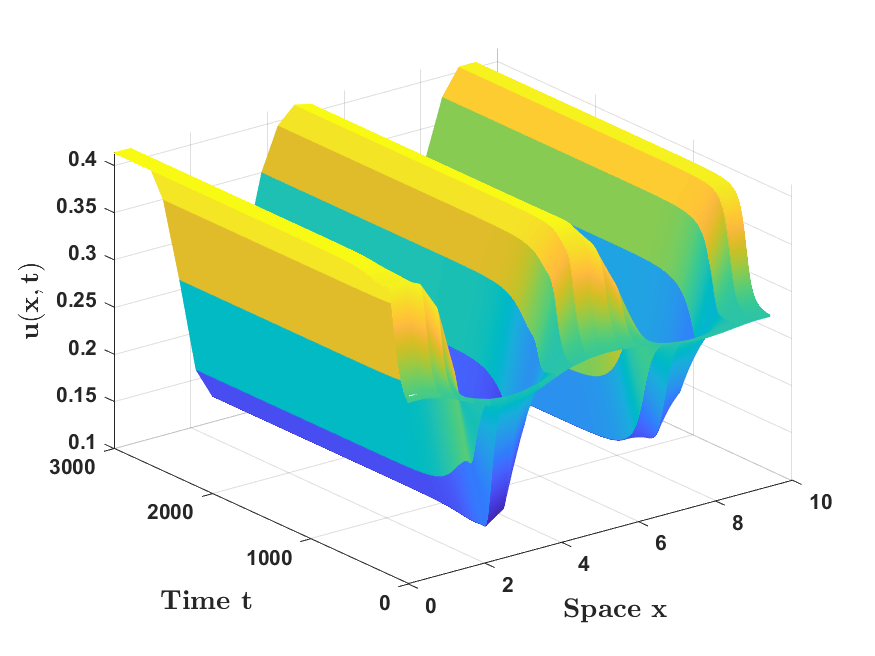}
\includegraphics[width=2in]{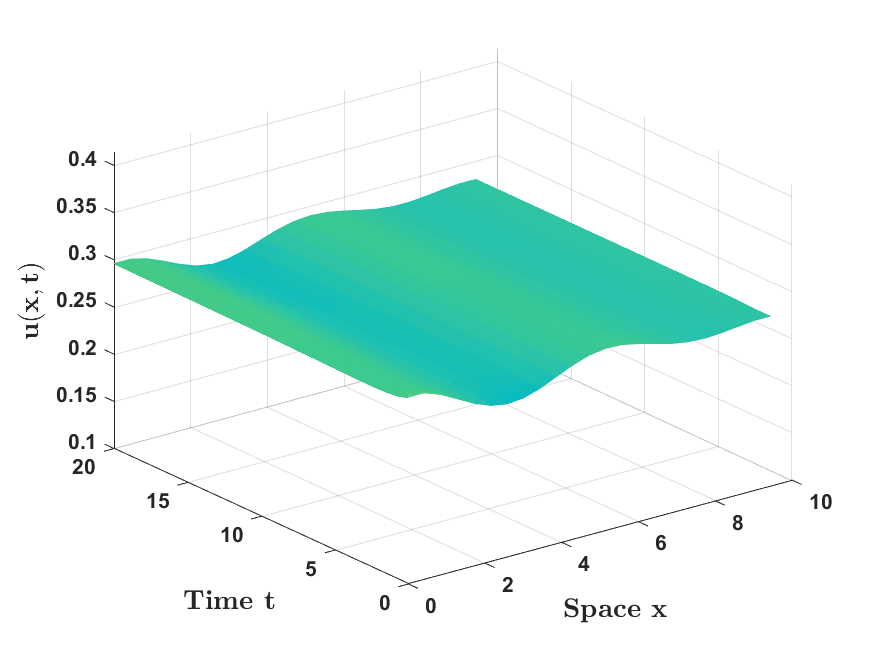}
\includegraphics[width=2in]{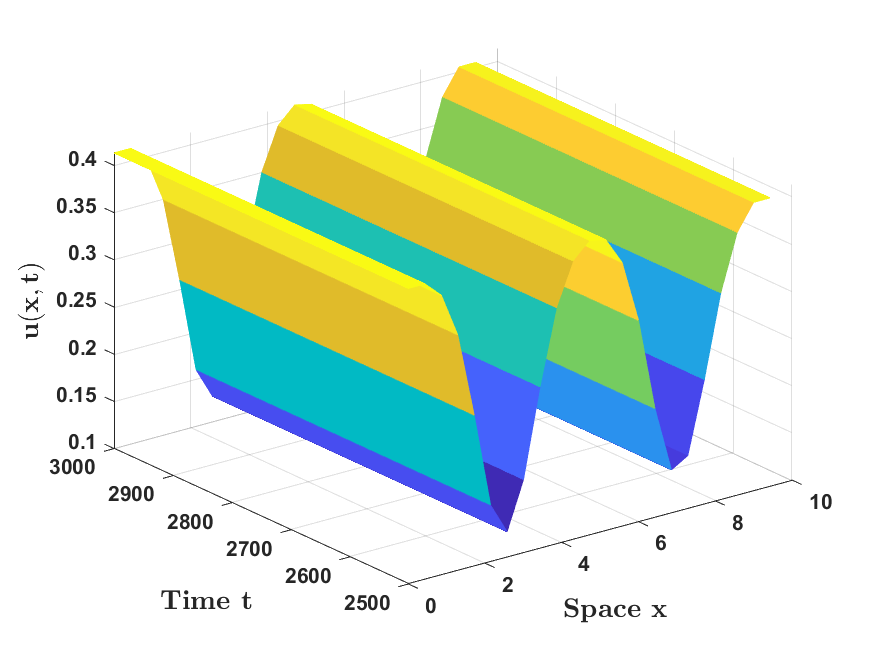} \\
\text{(a)} \hspace{4.6cm} \text{(b)} \hspace{4.6cm} \text{(c)} \\
\includegraphics[width=2in]{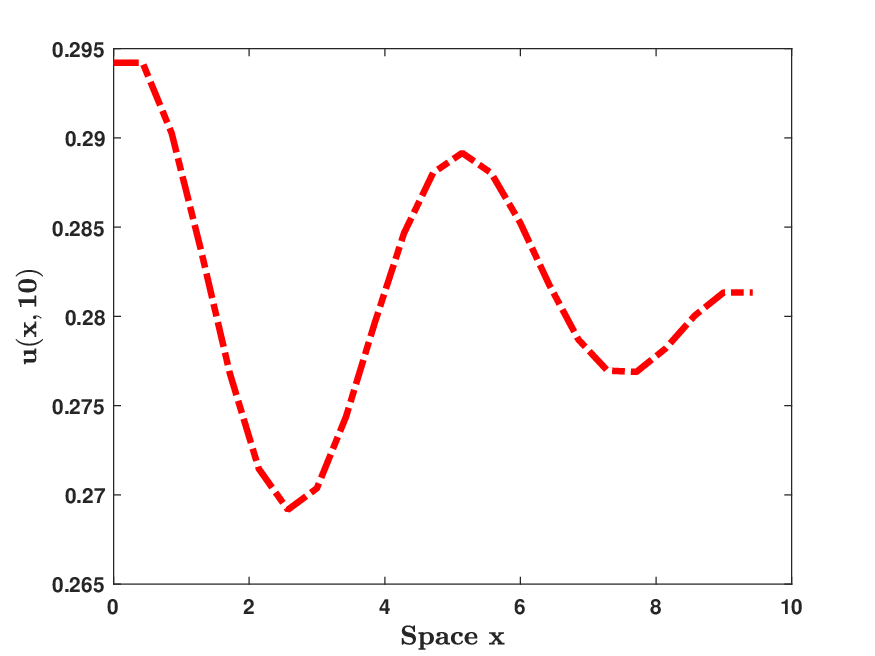}
\includegraphics[width=2in]{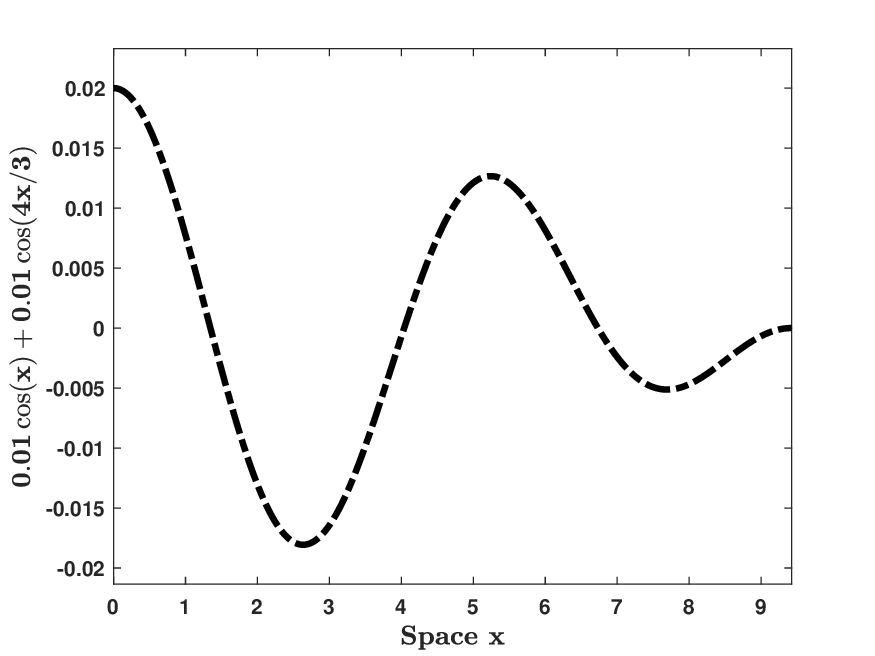}
\includegraphics[width=2in]{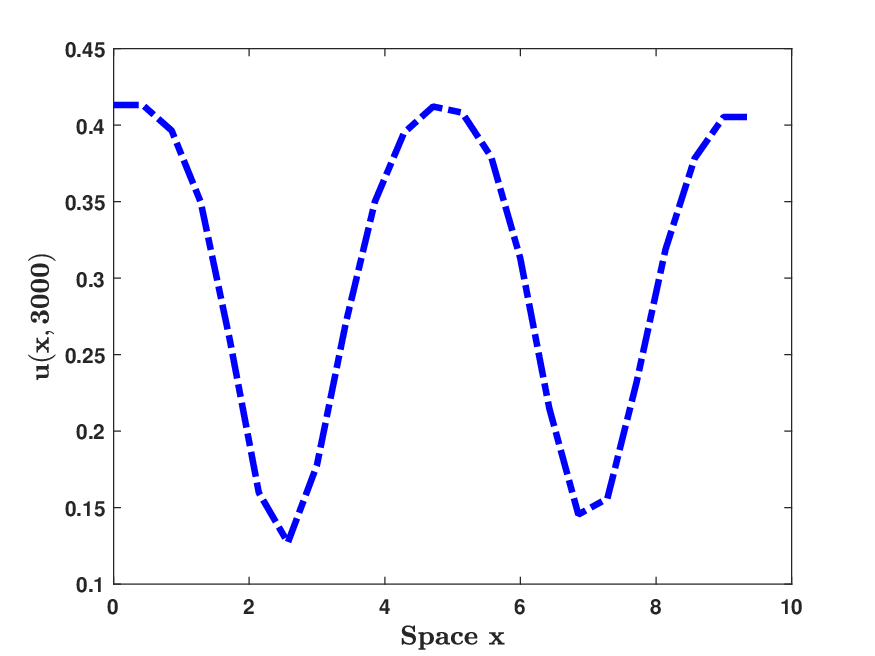} \\
\text{(d)} \hspace{4.6cm} \text{(e)} \hspace{4.6cm} \text{(f)} \\
\caption{For $(\mu_{1},\mu_{2})=(-0.002,-0.2) \in R_{4}$, two stable spatially inhomogeneous steady states with $r_{2}\phi_{n_{2}}\cos(4x/3)$-type spatial distribution exist. (a), (b) and (c) are the long-term, transient and final behaviors of $u(x,t)$, respectively, where the shape of the final behavior of $u(x,t)$ is $\cos(4x/3)$. (d) is the concentration profile for $u(x,t)$ at $t=10$, (e) is the image of the function $0.01\cos(x)+0.01\cos(4x/3)$, (f) is the concentration profile for $u(x,t)$ at $t=3000$. The initial functions are $u_{0}(x)=0.2806+0.01\cos(x)+0.01\cos(4x/3)$ and $v_{0}(x)=0.1813+0.01\cos(x)+0.01\cos(4x/3)$.}
\label{fig:6}
\end{figure}

\begin{figure}[!htbp]
\centering
\includegraphics[width=2in]{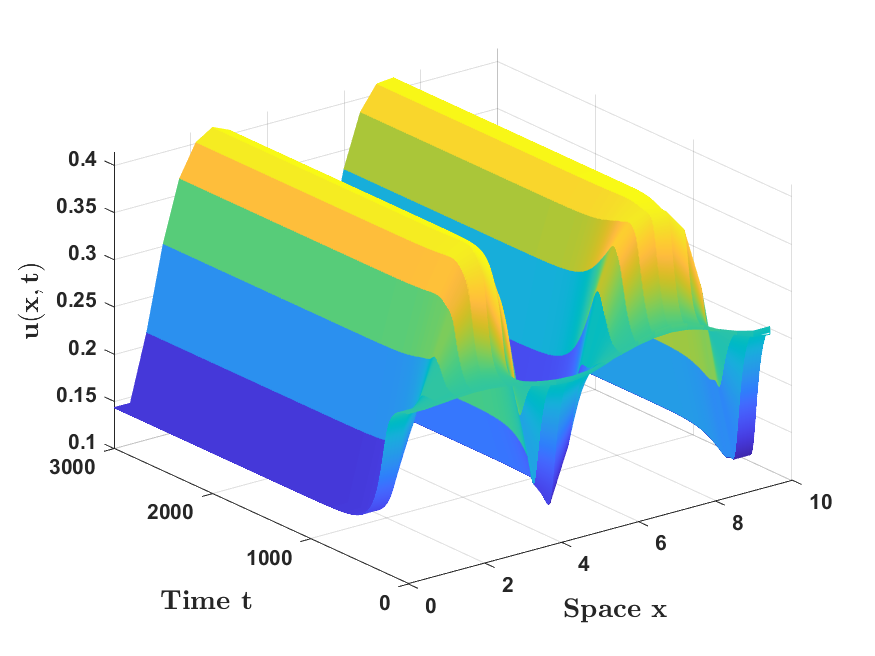}
\includegraphics[width=2in]{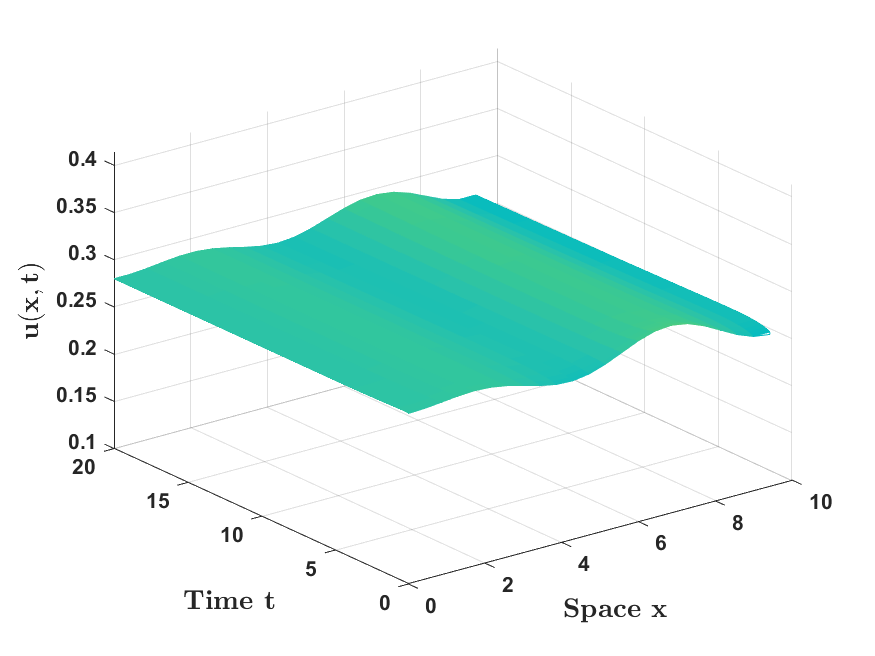}
\includegraphics[width=2in]{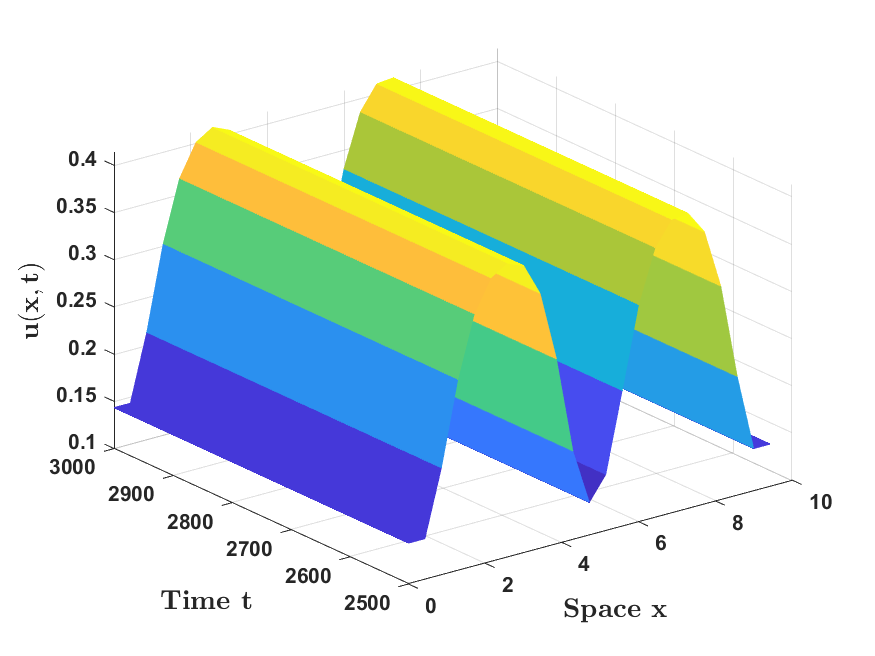} \\
\text{(a)} \hspace{4.6cm} \text{(b)} \hspace{4.6cm} \text{(c)} \\
\includegraphics[width=2in]{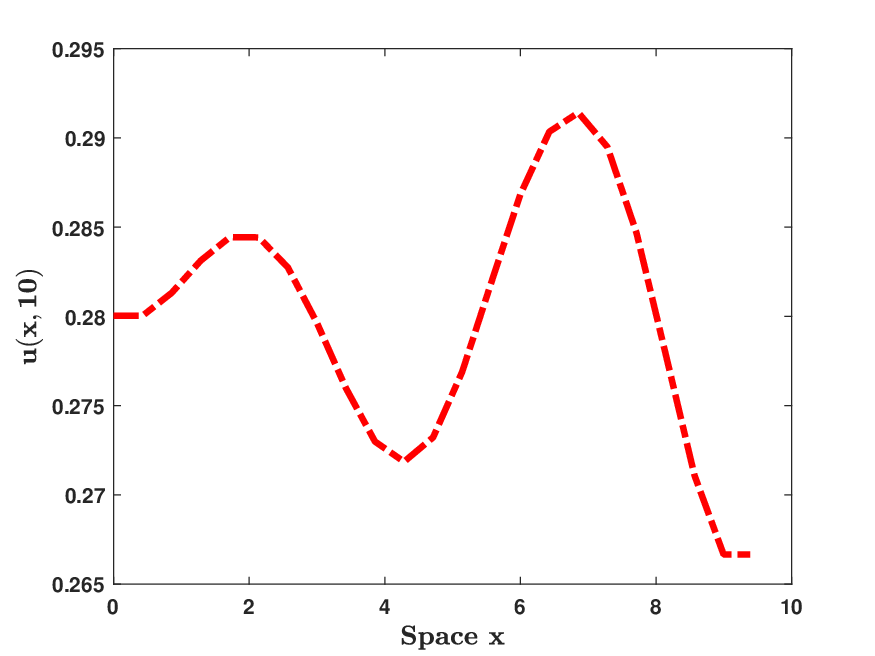}
\includegraphics[width=2in]{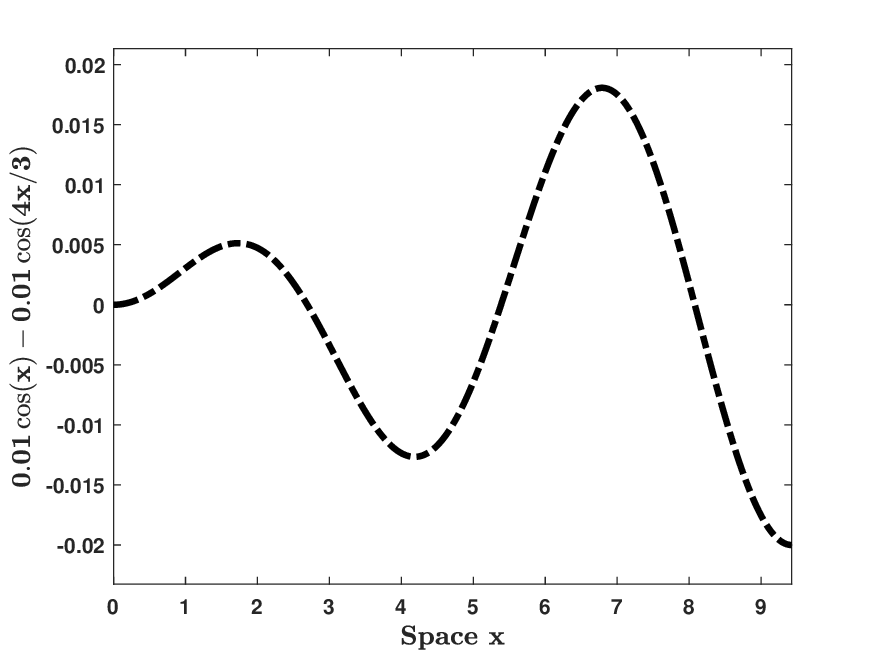}
\includegraphics[width=2in]{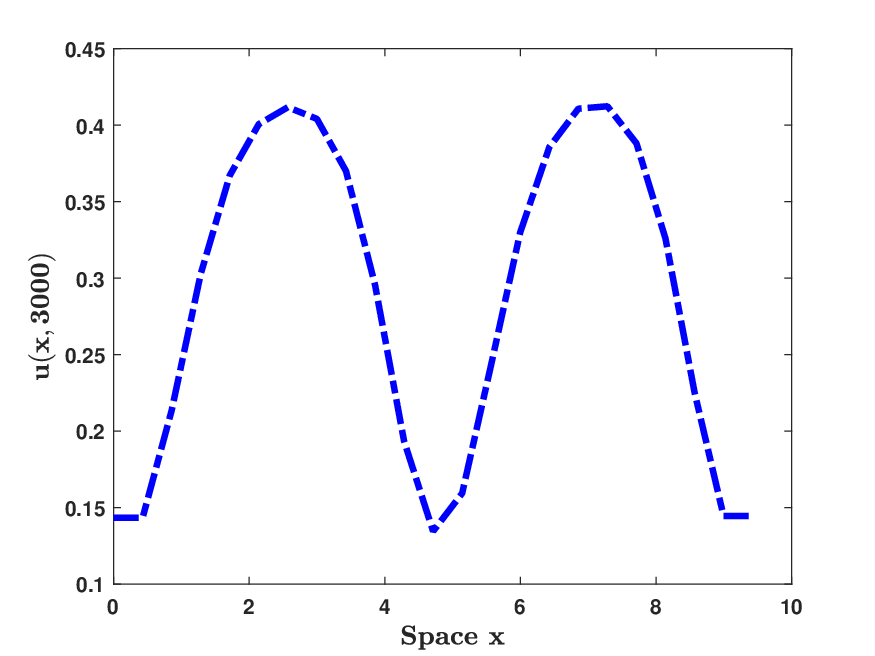} \\
\text{(d)} \hspace{4.6cm} \text{(e)} \hspace{4.6cm} \text{(f)} \\
\caption{For $(\mu_{1},\mu_{2})=(-0.002,-0.2) \in R_{4}$, two stable spatially inhomogeneous steady states with $r_{2}\phi_{n_{2}}\cos(4x/3)$-type spatial distribution exist. (a), (b) and (c) are the long-term, transient and final behaviors of $u(x,t)$, respectively, where the shape of the final behavior of $u(x,t)$ is $-\cos(4x/3)$. (d) is the concentration profile for $u(x,t)$ at $t=10$, (e) is the image of the function $0.01\cos(x)-0.01\cos(4x/3)$, (f) is the concentration profile for $u(x,t)$ at $t=3000$. The initial functions are $u_{0}(x)=0.2806+0.01\cos(x)-0.01\cos(4x/3)$ and $v_{0}(x)=0.1813+0.01\cos(x)-0.01\cos(4x/3)$.}
\label{fig:7}
\end{figure}

\begin{figure}[!htbp]
\centering
\includegraphics[width=2in]{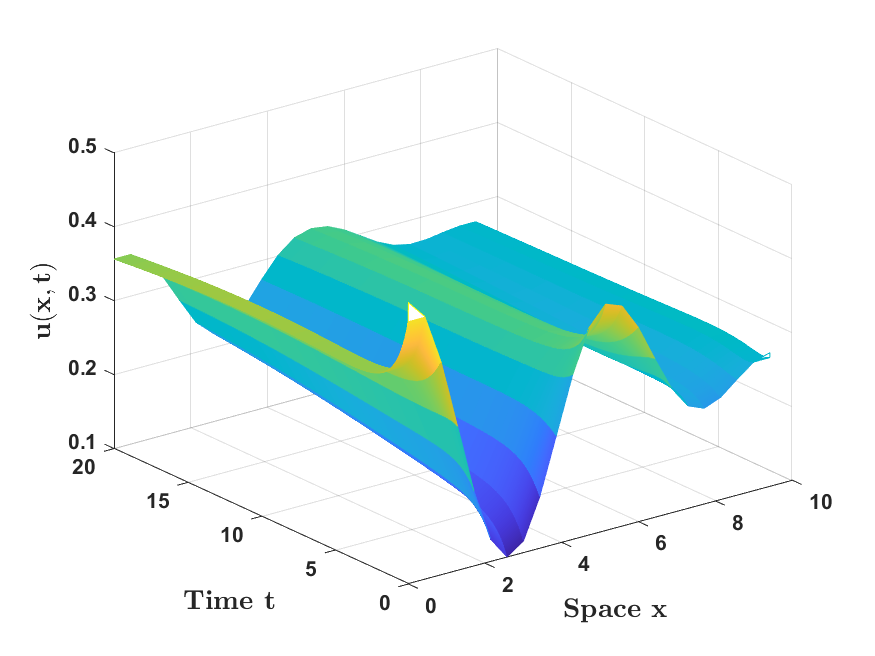}
\includegraphics[width=2in]{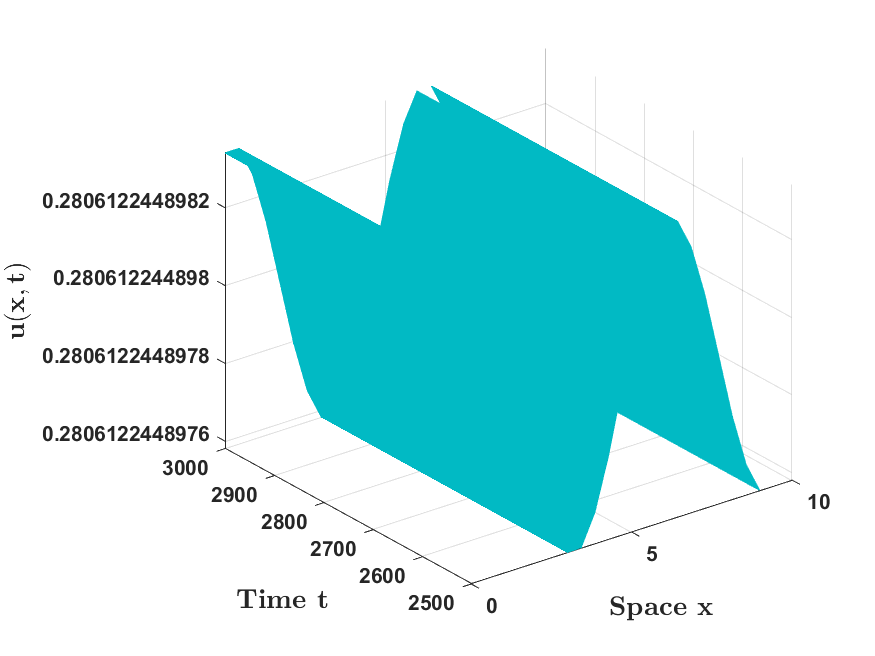} \\
\text{(a)} \hspace{4.6cm} \text{(b)} \\
\includegraphics[width=2in]{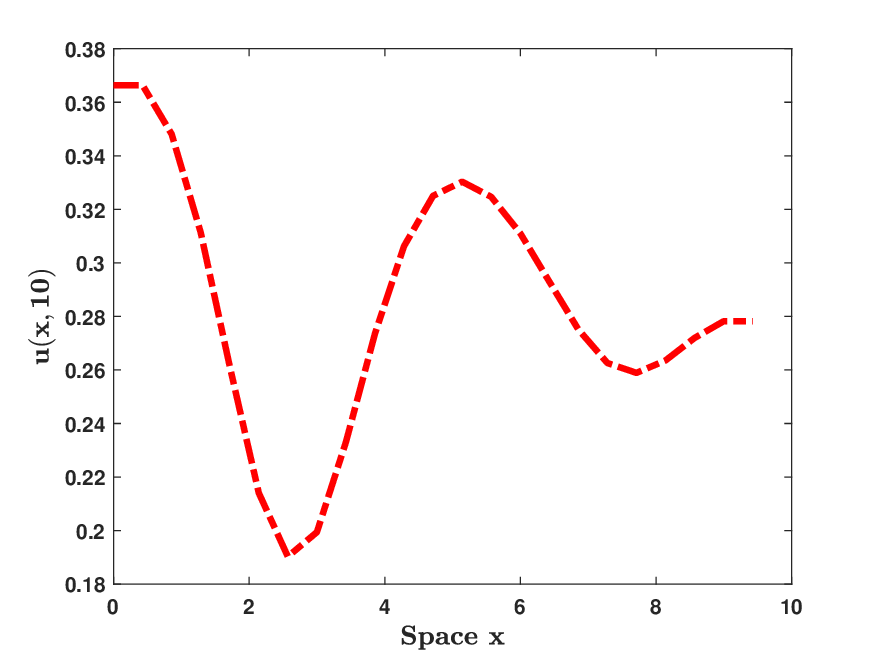}
\includegraphics[width=2in]{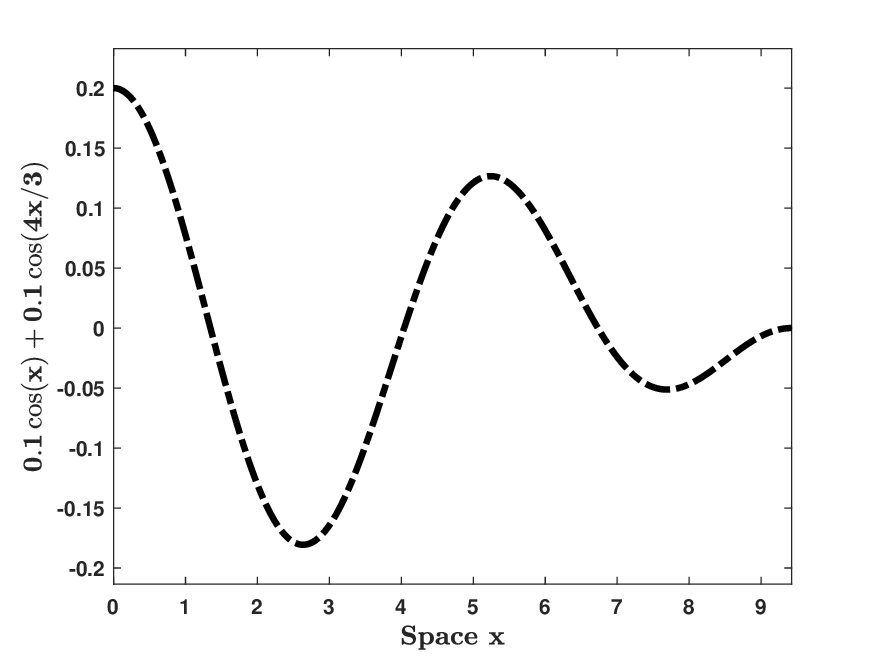}
\includegraphics[width=2in]{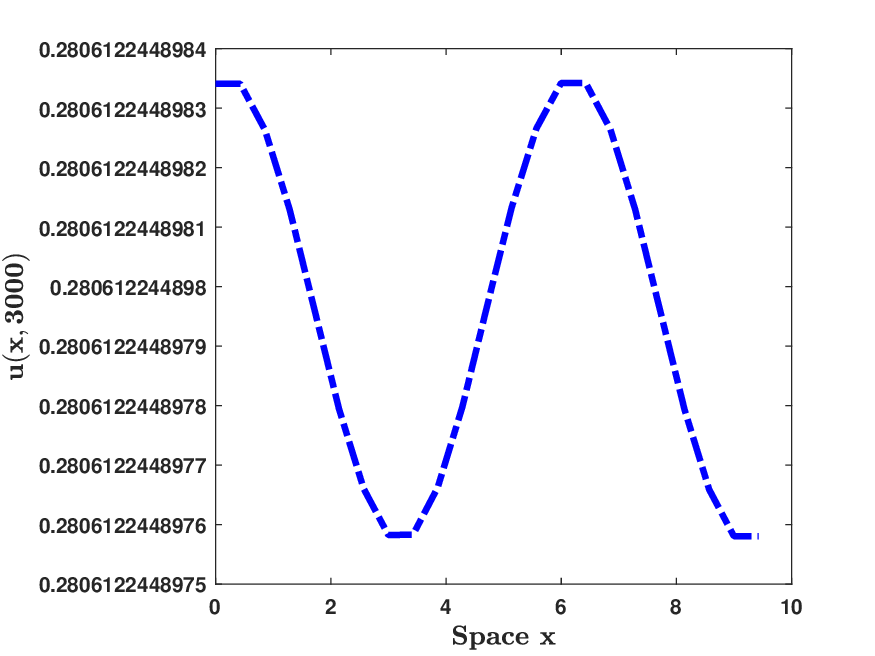} \\
\text{(c)} \hspace{4.6cm} \text{(d)} \hspace{4.6cm} \text{(e)} \\
\caption{For $(\mu_{1},\mu_{2})=(0.02,-0.14) \in R_{5}$, two stable spatially inhomogeneous steady states with $r_{1}\phi_{n_{1}}\cos(x)$-type spatial distribution exist. (a) and (b) are the transient and final behaviors of $u(x,t)$, respectively, where the shape of the final behavior of $u(x,t)$ is $\cos(x)$. (c) is the concentration profile for $u(x,t)$ at $t=10$, (d) is the image of the function $0.1\cos(x)+0.1\cos(4x/3)$, (e) is the concentration profile for $u(x,t)$ at $t=3000$. The initial functions are $u_{0}(x)=0.2806+0.1\cos(x)+0.1\cos(4x/3)$ and $v_{0}(x)=0.1813+0.1\cos(x)+0.1\cos(4x/3)$.}
\label{fig:8}
\end{figure}

\begin{figure}[!htbp]
\centering
\includegraphics[width=2in]{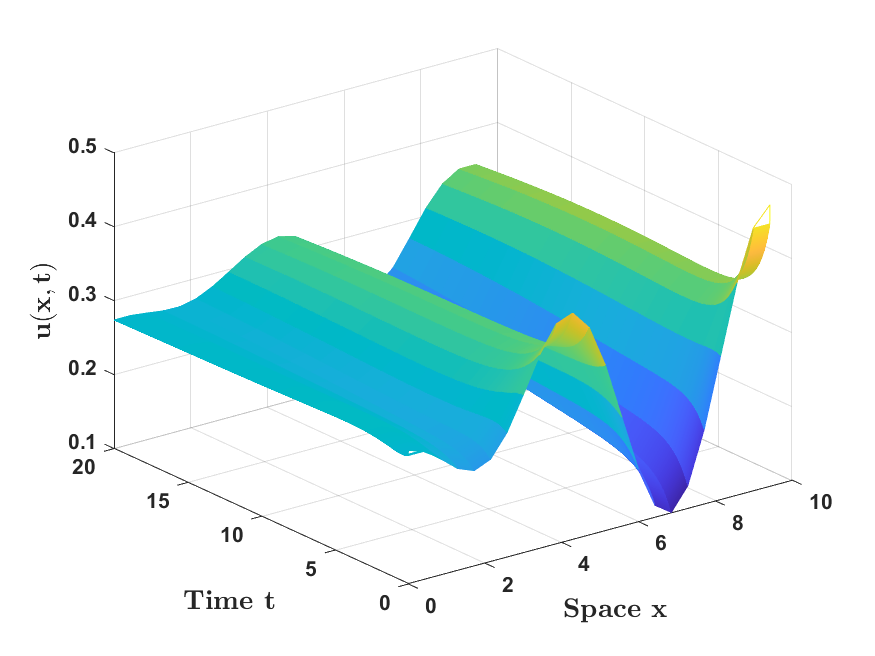}
\includegraphics[width=2in]{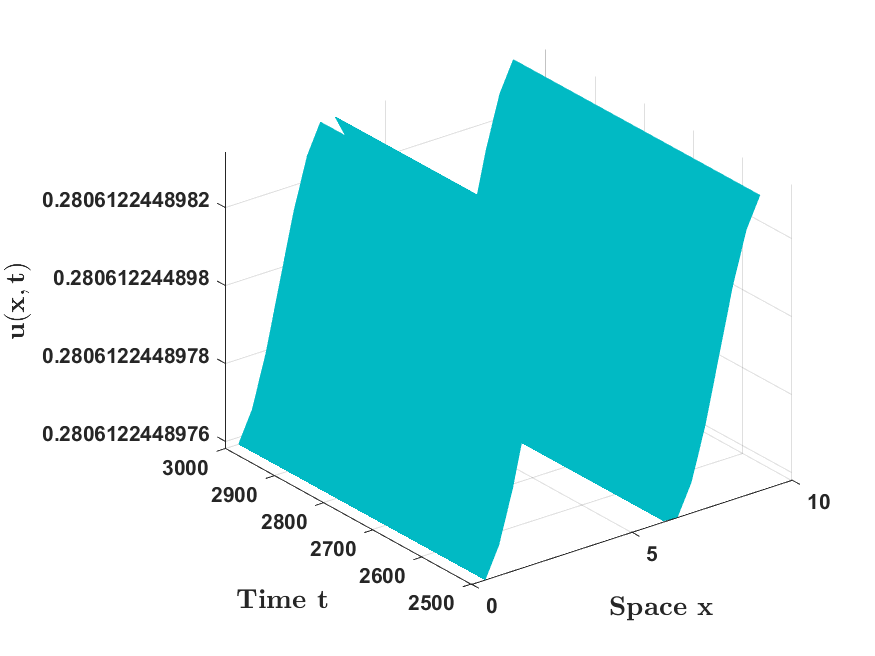} \\
\text{(a)} \hspace{4.6cm} \text{(b)} \\
\includegraphics[width=2in]{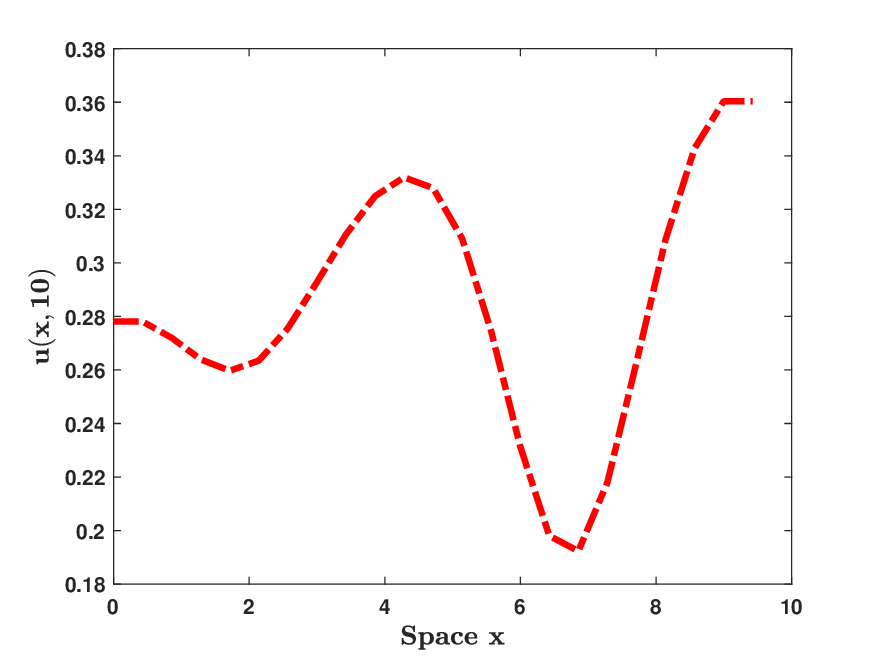}
\includegraphics[width=2in]{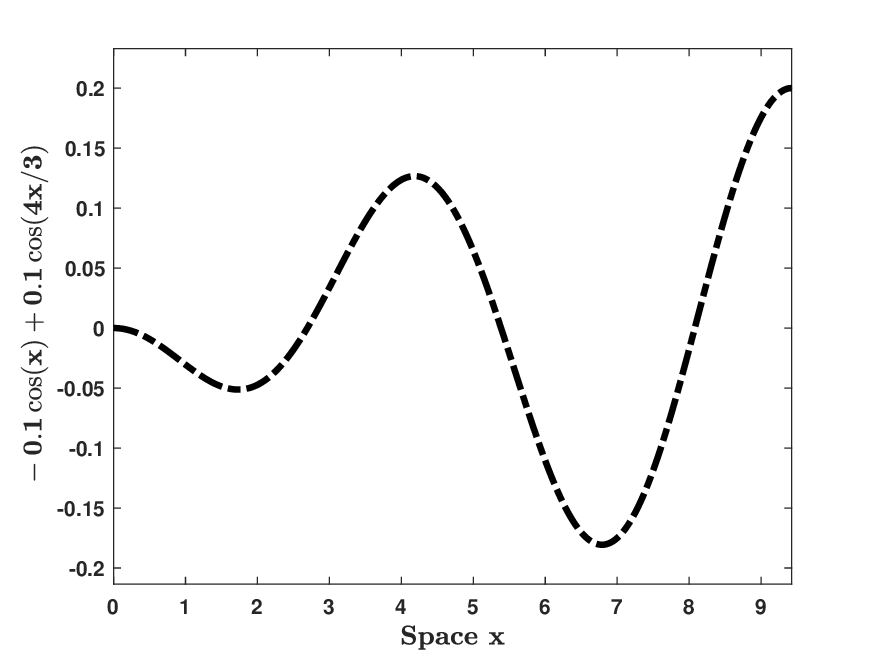}
\includegraphics[width=2in]{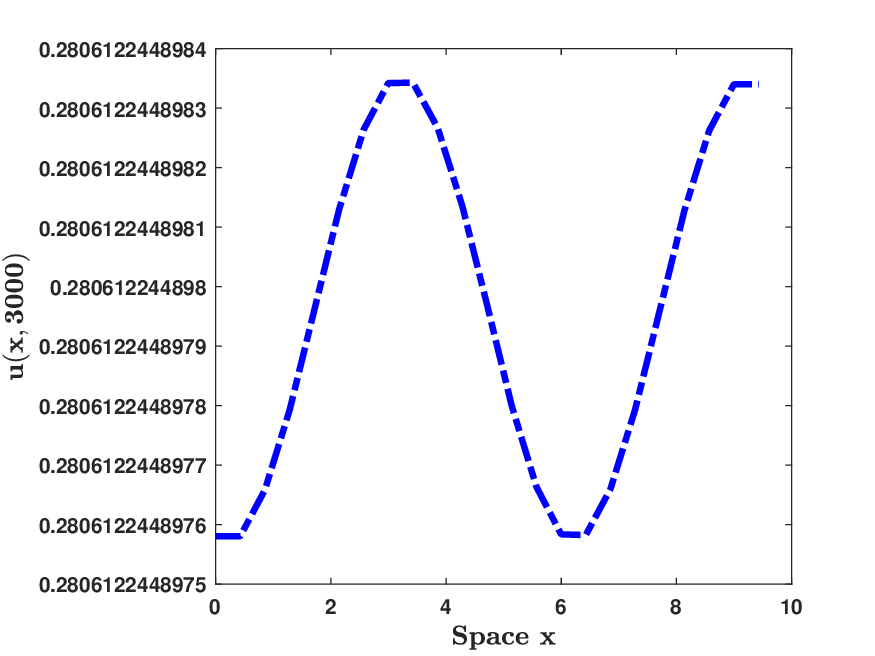} \\
\text{(c)} \hspace{4.6cm} \text{(d)} \hspace{4.6cm} \text{(e)} \\
\caption{For $(\mu_{1},\mu_{2})=(0.02,-0.14) \in R_{5}$, two stable spatially inhomogeneous steady states with $r_{1}\phi_{n_{1}}\cos(x)$-type spatial distribution exist. (a) and (b) are the transient and final behaviors of $u(x,t)$, respectively, where the shape of the final behavior of $u(x,t)$ is $-\cos(x)$. (c) is the concentration profile for $u(x,t)$ at $t=10$, (d) is the image of the function $-0.1\cos(x)+0.1\cos(4x/3)$, (e) is the concentration profile for $u(x,t)$ at $t=3000$. The initial functions are $u_{0}(x)=0.2806-0.1\cos(x)+0.1\cos(4x/3)$ and $v_{0}(x)=0.1813-0.1\cos(x)+0.1\cos(4x/3)$.}
\label{fig:9}
\end{figure}

More precisely, when $(\mu_{1},\mu_{2})$ is chosen in these six different regions, that is, when the bifurcation values $(d_{u},\alpha)=(d_{u}^{*}+\mu_{1},\alpha^{*}+\mu_{2})$ are chosen close to the $(3,4)$-mode Turing-Turing bifurcation point $(d_{u}^{*},\alpha^{*})\doteq (0.0253,0.5527)$, the model (1.1) shows different dynamics. That is,

(1) when $(\mu_{1},\mu_{2}) \in R_{1}$, the positive constant steady state $E_{*}(u_{*},v_{*})$ of the model (1.1) is locally asymptotically stable, and when $(\mu_{1},\mu_{2}) \notin R_{1}$, the positive constant steady state $E_{*}(u_{*},v_{*})$ of the model (1.1) is unstable. By choosing $(\mu_{1},\mu_{2})=(0.1,0.1) \in R_{1}$, i.e., $(d_{u},\alpha)\doteq (0.1253,0.6527)$, and by further analysis, we can see that the positive constant steady state $E_{*}(u_{*},v_{*})$ of the model (1.1) is locally asymptotically stable, see Fig. \ref{fig:3} for detail, and the initial functions are $u_{0}(x)=0.2806-0.03\cos(4x/3)$ and $v_{0}(x)=0.1813-0.01\cos(4x/3)$.

(2) when $(\mu_{1},\mu_{2}) \in R_{2}$, there are two stable spatially inhomogeneous steady states with $r_{2}\phi_{n_{2}}\cos(4x/3)$-type spatial distribution bifurcating from the positive constant steady state $E_{*}(u_{*},v_{*})$ as $(\mu_{1},\mu_{2})$ crosses $CL_{2}$.

(3) when $(\mu_{1},\mu_{2}) \in R_{3}$, the two stable spatially inhomogeneous steady states with $r_{2}\phi_{n_{2}}\cos(4x/3)$-type spatial distribution still remain, and the model (1.1) presents two unstable spatially inhomogeneous steady states with $r_{1}\phi_{n_{1}}\cos(x)$-type spatial distribution bifurcating from the positive constant steady state $E_{*}(u_{*},v_{*})$ as $(\mu_{1},\mu_{2})$ crosses $CL_{1}$.

(4) when $(\mu_{1},\mu_{2}) \in R_{4}$, two spatially inhomogeneous steady states with $r_{1}\phi_{n_{1}}\cos(x)$-type spatial distribution are stable as well as four superposed steady states possessing the saddle type and $r_{3}\phi_{n_{1}}\cos(x)+r_{4}\phi_{n_{2}}\cos(4x/3)$-type spatial distribution appear as $(\mu_{1},\mu_{2})$ crosses $CL_{3}$. Thus the model (1.1) exhibits the quad-stability of the coexistence of four spatially inhomogeneous steady states, accompanying that two stable spatially inhomogeneous steady states with $r_{2}\phi_{n_{2}}\cos(4x/3)$-type spatial distribution still persist. By choosing $(\mu_{1},\mu_{2})=(0.002,-0.1) \in R_{4}$, i.e., $(d_{u},\alpha)\doteq (0.0273,0.4527)$, from Fig. \ref{fig:4}, which corresponds to the initial functions $u_{0}(x)=0.2806+0.01\cos(x)+0.001\cos(4x/3)$ and $v_{0}(x)=0.1813+0.01\cos(x)+0.001\cos(4x/3)$, and from Fig. \ref{fig:5}, which corresponds to the initial functions $u_{0}(x)=0.2806-0.01\cos(x)+0.001\cos(4x/3)$ and $v_{0}(x)=0.1813-0.01\cos(x)+0.001\cos(4x/3)$, we can see that two stable spatially inhomogeneous steady states possessing $r_{1}\phi_{n_{1}}\cos(x)$-type spatial distribution exist. Furthermore, by choosing $(\mu_{1},\mu_{2})=(-0.002,-0.2) \in R_{4}$, i.e., $(d_{u},\alpha)\doteq (0.0233,0.3527)$, from Fig. \ref{fig:6}, which corresponds to the initial functions $u_{0}(x)=0.2806+0.01\cos(x)+0.01\cos(4x/3)$ and $v_{0}(x)=0.1813+0.01\cos(x)+0.01\cos(4x/3)$, and from Fig. \ref{fig:7}, which corresponds to the initial functions $u_{0}(x)=0.2806+0.01\cos(x)-0.01\cos(4x/3)$ and $v_{0}(x)=0.1813+0.01\cos(x)-0.01\cos(4x/3)$, we can see that two stable spatially inhomogeneous steady states possessing $r_{2}\phi_{n_{2}}\cos(4x/3)$-type spatial distribution exist.

(5) when $(\mu_{1},\mu_{2}) \in R_{5}$, the four superposed steady states possessing the saddle type and $r_{3}\phi_{n_{1}}\cos(x)+r_{4}\phi_{n_{2}}\cos(4x/3)$-type spatial distribution appear, and the two stable spatially inhomogeneous steady states with $r_{1}\phi_{n_{1}}\cos(x)$-type spatial distribution still remain. By choosing $(\mu_{1},\mu_{2})=(0.02,-0.14) \in R_{5}$, i.e., $(d_{u},\alpha)\doteq (0.0453,0.4127)$, from Fig. \ref{fig:8}, which corresponds to the initial functions $u_{0}(x)=0.2806+0.1\cos(x)+0.1\cos(4x/3)$ and $v_{0}(x)=0.1813+0.1\cos(x)+0.1\cos(4x/3)$, and from Fig. \ref{fig:9}, which corresponds to the initial functions $u_{0}(x)=0.2806-0.1\cos(x)+0.1\cos(4x/3)$ and $v_{0}(x)=0.1813-0.1\cos(x)+0.1\cos(4x/3)$, we can see that two stable spatially inhomogeneous steady states possessing $r_{1}\phi_{n_{1}}\cos(x)$-type spatial distribution exist.

(6) when $(\mu_{1},\mu_{2}) \in R_{6}$, the four superposed steady states possessing the saddle type and $r_{3}\phi_{n_{1}}\cos(x)+r_{4}\phi_{n_{2}}\cos(4x/3)$-type spatial distribution disappear because of the emergence of $CL_{4}$, and the two spatially inhomogeneous steady states with $r_{1}\phi_{n_{1}}\cos(x)$-type spatial distribution are stable.

\begin{remark}
Notice that the patterns corresponding to the Turing-Turing bifurcation are mainly embodied in the region $R_{4}$ and region $R_{5}$, thus in this paper, we mainly investigate the spatiotemporal dynamics of the model (1.1) for $(\mu_{1},\mu_{2}) \in R_{4}$ and $(\mu_{1},\mu_{2}) \in R_{5}$ with different initial functions.
\end{remark}

\begin{remark}
In order to better understand the function composition of the transient profiles, we draw the combined function images of $\cos(x)$ and $\cos(4x/3)$ in Fig. \ref{fig:4}(e), Fig. \ref{fig:5}(e), Fig. \ref{fig:6}(e), Fig. \ref{fig:7}(e), Fig. \ref{fig:8}(d), and Fig. \ref{fig:9}(d). From these figures, we can see that the curves of the transient profiles and the combined function are similar, which once again verify that the transient pattern is the superposition of two spatial patterns in the case of $3:4$ spatial response, which corresponds to the $(3,4)$-mode Turing-Turing bifurcation.
\end{remark}

\section{Conclusion and discussion}
\label{sec:6}

This paper has investigated a predator-prey reaction-diffusion model incorporating predator-taxis (representing prey avoidance behavior) and prey refuge, features of significant biological relevance. Selecting the prey's random diffusion coefficient and the predator-taxis coefficient as bifurcation parameters, we established the precise conditions under which the model undergoes a codimension-two Turing-Turing bifurcation. Utilizing center manifold theorem and normal form theory, we developed a method for calculating the normal form of this Turing-Turing bifurcation for the model and derived explicit formulas for its coefficients. This normal form enabled a detailed classification and description of the spatiotemporal dynamical behaviors exhibited by the model in the vicinity of the bifurcation point. The derived formulas are readily implementable computationally (e.g., via MATLAB). Numerical simulations, performed with appropriate parameter choices, successfully verified our theoretical findings and illustrated the emergence of spatially inhomogeneous steady-state solutions exhibiting diverse patterns.

Notably, the predator-taxis mechanism in model (1.1) is repulsive, reflecting the biologically meaningful scenario where prey individuals sense and actively avoid regions of high predator density to reduce predation risk. The inclusion of this nonlinear taxis term introduces strong coupling into the diffusion structure of model (1.1), consequently increasing the complexity of analyzing its linearized system and computing the cubic truncation of the normal form for the codimension-two Turing-Turing bifurcation.

In Section 2, we rigorously proved the local existence and uniqueness of classical solutions for model (1.1). However, the questions of global existence and uniform boundedness of these classical solutions remain open and warrant further investigation. Furthermore, our analysis did not account for time delays. Temporal delays are inherently present in intra-species interactions (within predator or prey populations) and inter-species interactions (between predators and prey). Therefore, extending this research to incorporate time delays alongside predator-taxis and prey refuge represents a highly relevant and challenging direction for future work. Investigating Turing-Turing bifurcations or other higher-codimension bifurcations in such delayed predator-prey taxis-diffusion models would be particularly valuable.

\section*{Acknowledgments}

\par\noindent The author is grateful to the anonymous referees for their useful suggestions which improve the contents of this paper. This paper is supported by the Natural Science Foundation of Shandong Province, China (No. ZR2024QA119), and the Excellent Doctor Project Fund from the Newly Introduced Talents Scientific Research Foundation at University of Jinan, China (No. XBS2448).

\section*{Author Contributions}

\par\noindent Yehu Lv: Methodology, formal analysis and investigation, writing and original draft preparation, writing, review and editing.

\section*{Declarations}

\subsection*{Ethical Approval}

\par\noindent This research did not involve human participants or animals.

\subsection*{Availability of Data and Materials}

\par\noindent The numerical simulations in this paper are carried out using MATLAB software.

\subsection*{Competing Interests}

\par\noindent The author declares that there is no conflict of interest, whether financial or non-financial.

\section*{Appendix A}
\setcounter{equation}{0}
\renewcommand\theequation{A.\arabic{equation}}

Here, we give the the detailed derivation processes of the normal form of Turing-Turing bifurcation for model (1.1).
\par~~~
\par\noindent {\bf{Calculation of $g_{2}^{1}(z,0,\mu)$}}
\par~~~

By noticing that $\varphi=\Phi z_{x}$ when $w=0$, then denote
\begin{equation}\begin{aligned}
&F_{20}^{d}(\varphi)=F_{20}^{d}(\Phi z_{x}):=\left(\begin{array}{c}
\alpha^{*}\left(\varphi^{(1)}_{x}\varphi^{(2)}_{x}+\varphi^{(1)}\varphi^{(2)}_{xx}\right) \\
0
\end{array}\right), \\
&F_{21}^{d}(\varphi,\mu)=F_{21}^{d}(\Phi z_{x},\mu):=\mu_{1}\left(\begin{array}{c}
\varphi_{xx}^{(1)} \\
0
\end{array}\right)+\mu_{2}\left(\begin{array}{c}
u_{*}\varphi_{xx}^{(2)} \\
0
\end{array}\right),
\end{aligned}\end{equation}
and from (4.22), we have
\begin{equation}\begin{aligned}
&f_{2}^{1}(z,0,\mu)=\Psi\left(\begin{aligned}
&\left[\widetilde{F}_{2}(\Phi z_{x},\mu),\beta_{\nu}^{(1)}\right] \\
&\left[\widetilde{F}_{2}(\Phi z_{x},\mu),\beta_{\nu}^{(2)}\right]
\end{aligned}\right)_{\nu=n_{1}}^{\nu=n_{2}},
\end{aligned}\end{equation}
where
\begin{equation}
\widetilde{F}_{2}(\Phi z_{x},\mu)=F_{2}^{d}(\Phi z_{x},\mu)+F_{2}(\Phi z_{x})=F_{20}^{d}(\Phi z_{x})+F_{21}^{d}(\Phi z_{x},\mu)+F_{2}(\Phi z_{x}).
\end{equation}
From (4.3), we can write $F_{2}(\Phi z_{x}+w)$ as
\begin{equation}\begin{aligned}
F_{2}(\Phi z_{x}+w)&=\sum_{q_{1}+q_{2}=2}A_{q_{1}q_{2}}z_{1}^{q_{1}}z_{2}^{q_{2}}\gamma_{n_{1}}^{q_{1}}(x)\gamma_{n_{2}}^{q_{2}}(x)+S_{2}(\Phi z_{x},w)+O\left(|w|^{2}\right) \\
&=A_{20}z_{1}^{2}\gamma_{n_{1}}^{2}(x)+A_{02}z_{2}^{2}\gamma_{n_{2}}^{2}(x)+A_{11}z_{1}z_{2}\gamma_{n_{1}}(x)\gamma_{n_{2}}(x) \\
&+S_{2}(\Phi z_{x},w)+O\left(|w|^{2}\right),
\end{aligned}\end{equation}
where $S_{2}(\Phi z_{x},w)$ is the product term of $\Phi z_{x}$ and $w$, and
\begin{equation*}
A_{q_{1}q_{2}}=\left(A_{q_{1}q_{2}}^{(1)},A_{q_{1}q_{2}}^{(2)}\right)^{T},~q_{1}, q_{2} \in \mathbb{N}_{0}.
\end{equation*}
Therefore, from (A.4), we have
\begin{equation}
F_{2}(\Phi z_{x})=A_{20}z_{1}^{2}\gamma_{n_{1}}^{2}(x)+A_{02}z_{2}^{2}\gamma_{n_{2}}^{2}(x)+A_{11}z_{1}z_{2}\gamma_{n_{1}}(x)\gamma_{n_{2}}(x).
\end{equation}
By combining with (4.1), (4.13) and (4.14), when $w=0$, we have
\begin{equation*}\begin{aligned}
\varphi&=\Phi z_{x}=\phi_{n_{1}}z_{1}\gamma_{n_{1}}(x)+\phi_{n_{2}}z_{2}\gamma_{n_{2}}(x) \\
&=\left(\begin{array}{c}
\phi_{n_{1}}^{(1)}z_{1}\frac{\sqrt{2}}{\sqrt{\ell\pi}}\cos\left(\frac{n_{1}x}{\ell}\right)+\phi_{n_{2}}^{(1)}z_{2}\frac{\sqrt{2}}{\sqrt{\ell\pi}}\cos\left(\frac{n_{2}x}{\ell}\right) \\
\phi_{n_{1}}^{(2)}z_{1}\frac{\sqrt{2}}{\sqrt{\ell\pi}}\cos\left(\frac{n_{1}x}{\ell}\right)+\phi_{n_{2}}^{(2)}z_{2}\frac{\sqrt{2}}{\sqrt{\ell\pi}}\cos\left(\frac{n_{2}x}{\ell}\right)
\end{array}\right),
\end{aligned}\end{equation*}
then we have
\begin{equation}\begin{aligned}
&\varphi^{(1)}_{x}=-\frac{n_{1}}{\ell}\phi_{n_{1}}^{(1)}z_{1}\frac{\sqrt{2}}{\sqrt{\ell\pi}}\sin\left(\frac{n_{1}x}{\ell}\right)-\frac{n_{2}}{\ell}\phi_{n_{2}}^{(1)}z_{2}\frac{\sqrt{2}}{\sqrt{\ell\pi}}\sin\left(\frac{n_{2}x}{\ell}\right), \\
&\varphi^{(2)}_{x}=-\frac{n_{1}}{\ell}\phi_{n_{1}}^{(2)}z_{1}\frac{\sqrt{2}}{\sqrt{\ell\pi}}\sin\left(\frac{n_{1}x}{\ell}\right)-\frac{n_{2}}{\ell}\phi_{n_{2}}^{(2)}z_{2}\frac{\sqrt{2}}{\sqrt{\ell\pi}}\sin\left(\frac{n_{2}x}{\ell}\right), \\
&\varphi^{(1)}_{xx}=-\frac{n_{1}^{2}}{\ell^{2}}\phi_{n_{1}}^{(1)}z_{1}\frac{\sqrt{2}}{\sqrt{\ell\pi}}\cos\left(\frac{n_{1}x}{\ell}\right)-\frac{n_{2}^{2}}{\ell^{2}}\phi_{n_{2}}^{(1)}z_{2}\frac{\sqrt{2}}{\sqrt{\ell\pi}}\cos\left(\frac{n_{2}x}{\ell}\right), \\
&\varphi^{(1)}=\phi_{n_{1}}^{(1)}z_{1}\frac{\sqrt{2}}{\sqrt{\ell\pi}}\cos\left(\frac{n_{1}x}{\ell}\right)+\phi_{n_{2}}^{(1)}z_{2}\frac{\sqrt{2}}{\sqrt{\ell\pi}}\cos\left(\frac{n_{2}x}{\ell}\right), \\
&\varphi^{(2)}_{xx}=-\frac{n_{1}^{2}}{\ell^{2}}\phi_{n_{1}}^{(2)}z_{1}\frac{\sqrt{2}}{\sqrt{\ell\pi}}\cos\left(\frac{n_{1}x}{\ell}\right)-\frac{n_{2}^{2}}{\ell^{2}}\phi_{n_{2}}^{(2)}z_{2}\frac{\sqrt{2}}{\sqrt{\ell\pi}}\cos\left(\frac{n_{2}x}{\ell}\right).
\end{aligned}\end{equation}
Therefore, from (A.6), we have
\begin{equation}\begin{aligned}
\varphi^{(1)}_{x}\varphi^{(2)}_{x}&=\frac{n_{1}^{2}}{\ell^{2}}\phi_{n_{1}}^{(1)}\phi_{n_{1}}^{(2)}z_{1}^{2}\frac{2}{\ell\pi}\sin^{2}\left(\frac{n_{1}x}{\ell}\right) \\
&+\frac{n_{2}^{2}}{\ell^{2}}\phi_{n_{2}}^{(1)}\phi_{n_{2}}^{(2)}z_{2}^{2}\frac{2}{\ell\pi}\sin^{2}\left(\frac{n_{2}x}{\ell}\right)\\
&+\frac{n_{1}n_{2}}{\ell^{2}}\left(\phi_{n_{1}}^{(1)}\phi_{n_{2}}^{(2)}+\phi_{n_{2}}^{(1)}\phi_{n_{1}}^{(2)}\right)z_{1}z_{2}\frac{2}{\ell\pi}\sin\left(\frac{n_{1}x}{\ell}\right)\sin\left(\frac{n_{2}x}{\ell}\right), \\
\varphi^{(1)}\varphi^{(2)}_{xx}&=-\frac{n_{1}^{2}}{\ell^{2}}\phi_{n_{1}}^{(1)}\phi_{n_{1}}^{(2)}z_{1}^{2}\frac{2}{\ell\pi}\cos^{2}\left(\frac{n_{1}x}{\ell}\right) \\
&-\frac{n_{2}^{2}}{\ell^{2}}\phi_{n_{1}}^{(1)}\phi_{n_{2}}^{(2)}z_{1}z_{2}\frac{2}{\ell\pi}\cos\left(\frac{n_{1}x}{\ell}\right)\cos\left(\frac{n_{2}x}{\ell}\right) \\
&-\frac{n_{1}^{2}}{\ell^{2}}\phi_{n_{2}}^{(1)}\phi_{n_{1}}^{(2)}z_{1}z_{2}\frac{2}{\ell\pi}\cos\left(\frac{n_{1}x}{\ell}\right)\cos\left(\frac{n_{2}x}{\ell}\right) \\
&-\frac{n_{2}^{2}}{\ell^{2}}\phi_{n_{2}}^{(1)}\phi_{n_{2}}^{(2)}z_{2}^{2}\frac{2}{\ell\pi}\cos^{2}\left(\frac{n_{2}x}{\ell}\right).
\end{aligned}\end{equation}
By combining with (A.1), (A.3), (A.5), (A.6) and (A.7), $\widetilde{F}_{2}(\Phi z_{x},\mu)$ can be rewritten as
\begin{equation}\begin{aligned}
\widetilde{F}_{2}(\Phi z_{x},\mu)&=F_{2}^{d}(\Phi z_{x},\mu)+F_{2}(\Phi z_{x}) \\
&=F_{20}^{d}(\Phi z_{x})+F_{21}^{d}(\Phi z_{x},\mu)+F_{2}(\Phi z_{x}) \\
&=\alpha^{*}\left(\begin{array}{c}
\varphi^{(1)}_{x}\varphi^{(2)}_{x} \\
0
\end{array}\right)+\alpha^{*}\left(\begin{array}{c}
\varphi^{(1)}\varphi^{(2)}_{xx} \\
0
\end{array}\right)+\mu_{1}\left(\begin{array}{c}
\varphi_{xx}^{(1)} \\
0
\end{array}\right)+\mu_{2}u_{*}\left(\begin{array}{c}
\varphi_{xx}^{(2)} \\
0
\end{array}\right) \\
&+A_{20}z_{1}^{2}\gamma_{n_{1}}^{2}(x)+A_{02}z_{2}^{2}\gamma_{n_{2}}^{2}(x)+A_{11}z_{1}z_{2}\gamma_{n_{1}}(x)\gamma_{n_{2}}(x) \\
&=\alpha^{*}\frac{n_{1}^{2}}{\ell^{2}}A_{20}^{(d,1)}z_{1}^{2}\widetilde{\gamma}_{n_{1}}^{2}(x)+\alpha^{*}\frac{n_{1}n_{2}}{\ell^{2}}(A_{11}^{(d,1)}+A_{11}^{(d,2)})z_{1}z_{2}\widetilde{\gamma}_{n_{1}}(x)\widetilde{\gamma}_{n_{2}}(x) \\
&+\alpha^{*}\frac{n_{2}^{2}}{\ell^{2}}A_{02}^{(d,1)}z_{2}^{2}\widetilde{\gamma}_{n_{2}}^{2}(x)-\alpha^{*}\frac{n_{1}^{2}}{\ell^{2}}A_{20}^{(d,1)}z_{1}^{2}\gamma_{n_{1}}^{2}(x)-\alpha^{*}\frac{n_{2}^{2}}{\ell^{2}}A_{11}^{(d,1)}z_{1}z_{2}\gamma_{n_{1}}(x)\gamma_{n_{2}}(x) \\
&-\alpha^{*}\frac{n_{1}^{2}}{\ell^{2}}A_{11}^{(d,2)}z_{1}z_{2}\gamma_{n_{1}}(x)\gamma_{n_{2}}(x)-\alpha^{*}\frac{n_{2}^{2}}{\ell^{2}}A_{02}^{(d,1)}z_{2}^{2}\gamma_{n_{2}}^{2}(x)-\mu_{1}\frac{n_{1}^{2}}{\ell^{2}}A_{1010}^{(d,1)}z_{1}\gamma_{n_{1}}(x) \\
&-\mu_{1}\frac{n_{2}^{2}}{\ell^{2}}A_{0110}^{(d,1)}z_{2}\gamma_{n_{2}}(x)-\mu_{2}u_{*}\frac{n_{1}^{2}}{\ell^{2}}A_{1001}^{(d,1)}z_{1}\gamma_{n_{1}}(x)-\mu_{2}u_{*}\frac{n_{2}^{2}}{\ell^{2}}A_{0101}^{(d,1)}z_{2}\gamma_{n_{2}}(x) \\
&+A_{20}z_{1}^{2}\gamma_{n_{1}}^{2}(x)+A_{02}z_{2}^{2}\gamma_{n_{2}}^{2}(x)+A_{11}z_{1}z_{2}\gamma_{n_{1}}(x)\gamma_{n_{2}}(x),
\end{aligned}\end{equation}
where
\begin{equation}
\widetilde{\gamma}_{n}(x)=\frac{\sqrt{2}}{\sqrt{\ell\pi}}\sin\left(\frac{nx}{\ell}\right),~n=n_{1}, n_{2}.
\end{equation}
By combining with (4.30), (A.2), (A.8) and (A.9), and by noticing that
\begin{equation*}\begin{aligned}
\alpha_{1}&:=\int_{0}^{\ell\pi}\widetilde{\gamma}_{n_{1}}^{2}(x)\gamma_{n_{1}}(x)\mathrm{d}x=0, \\
\alpha_{2}&:=\int_{0}^{\ell\pi}\widetilde{\gamma}_{n_{1}}(x)\widetilde{\gamma}_{n_{2}}(x)\gamma_{n_{1}}(x)\mathrm{d}x=\begin{cases}
\frac{1}{\sqrt{2\ell\pi}}, & n_{2}=2n_{1}, \\
0, & n_{2}\neq 2n_{1},
\end{cases} \\
\alpha_{3}&:=\int_{0}^{\ell\pi}\widetilde{\gamma}_{n_{2}}^{2}(x)\gamma_{n_{1}}(x)\mathrm{d}x=\begin{cases}
-\frac{1}{\sqrt{2\ell\pi}}, & n_{1}=2n_{2}, \\
0, & n_{1}\neq 2n_{2},
\end{cases} \\
\alpha_{4}&:=\int_{0}^{\ell\pi}\gamma_{n_{1}}^{2}(x)\gamma_{n_{1}}(x)\mathrm{d}x=0, \\
\alpha_{5}&:=\int_{0}^{\ell\pi}\gamma_{n_{1}}(x)\gamma_{n_{2}}(x)\gamma_{n_{1}}(x)\mathrm{d}x=\begin{cases}
\frac{1}{\sqrt{2\ell\pi}}, & n_{2}=2n_{1}, \\
0, & n_{2}\neq 2n_{1},
\end{cases} \\
\alpha_{6}&:=\int_{0}^{\ell\pi}\gamma_{n_{2}}^{2}(x)\gamma_{n_{1}}(x)\mathrm{d}x=\begin{cases}
\frac{1}{\sqrt{2\ell\pi}}, & n_{1}=2n_{2}, \\
0, & n_{1}\neq 2n_{2},
\end{cases} \\
\alpha_{7}&:=\int_{0}^{\ell\pi}\gamma_{n_{1}}(x)\gamma_{n_{1}}(x)\mathrm{d}x=1, \\
\alpha_{8}&:=\int_{0}^{\ell\pi}\gamma_{n_{2}}(x)\gamma_{n_{1}}(x)\mathrm{d}x=0
\end{aligned}\end{equation*}
and
\begin{equation*}\begin{aligned}
&\beta_{1}:=\int_{0}^{\ell\pi}\widetilde{\gamma}_{n_{1}}^{2}(x)\gamma_{n_{2}}(x)\mathrm{d}x=\begin{cases}
-\frac{1}{\sqrt{2\ell\pi}}, & n_{2}=2n_{1}, \\
0, & n_{2}\neq 2n_{1},
\end{cases} \\
&\beta_{2}:=\int_{0}^{\ell\pi}\widetilde{\gamma}_{n_{1}}(x)\widetilde{\gamma}_{n_{2}}(x)\gamma_{n_{2}}(x)\mathrm{d}x=\begin{cases}
\frac{1}{\sqrt{2\ell\pi}}, & n_{1}=2n_{2}, \\
0, & n_{1}\neq 2n_{2},
\end{cases} \\
&\beta_{3}:=\int_{0}^{\ell\pi}\widetilde{\gamma}_{n_{2}}^{2}(x)\gamma_{n_{2}}(x)\mathrm{d}x=0, \\
&\beta_{4}:=\int_{0}^{\ell\pi}\gamma_{n_{1}}^{2}(x)\gamma_{n_{2}}(x)\mathrm{d}x=\begin{cases}
\frac{1}{\sqrt{2\ell\pi}}, & n_{2}=2n_{1}, \\
0, & n_{2}\neq 2n_{1},
\end{cases} \\
&\beta_{5}:=\int_{0}^{\ell\pi}\gamma_{n_{1}}(x)\gamma_{n_{2}}(x)\gamma_{n_{2}}(x)\mathrm{d}x=\begin{cases}
\frac{1}{\sqrt{2\ell\pi}}, & n_{1}=2n_{2}, \\
0, & n_{1}\neq 2n_{2},
\end{cases} \\
&\beta_{6}:=\int_{0}^{\ell\pi}\gamma_{n_{2}}^{2}(x)\gamma_{n_{2}}(x)\mathrm{d}x=0,~\beta_{7}:=\int_{0}^{\ell\pi}\gamma_{n_{1}}(x)\gamma_{n_{2}}(x)\mathrm{d}x=0, \\
&\beta_{8}:=\int_{0}^{\ell\pi}\gamma_{n_{2}}(x)\gamma_{n_{2}}(x)\mathrm{d}x=1,
\end{aligned}\end{equation*}
we have
\begin{equation}\begin{aligned}
&\frac{1}{2!}g_{2}^{1}(z,0,\mu)=\frac{1}{2!}\operatorname{Proj}_{\operatorname{Ker}(M_{2}^{1})}f_{2}^{1}(z,0,\mu) \\
&=\left(\begin{array}{c}
B_{2000}^{(1)}z_{1}^{2}+B_{1100}^{(1)}z_{1}z_{2}+(B_{1010}^{(1)}\mu_{1}+B_{1001}^{(1)}\mu_{2})z_{1}+B_{0200}^{(1)}z_{2}^{2}+(B_{0110}^{(1)}\mu_{1}+B_{0101}^{(1)}\mu_{2})z_{2} \\
B_{2000}^{(2)}z_{1}^{2}+B_{1100}^{(2)}z_{1}z_{2}+(B_{1010}^{(2)}\mu_{1}+B_{1001}^{(2)}\mu_{2})z_{1}+B_{0200}^{(2)}z_{2}^{2}+(B_{0110}^{(2)}\mu_{1}+B_{0101}^{(2)}\mu_{2})z_{2}
\end{array}\right),
\end{aligned}\end{equation}
where
\begin{equation*}\left\{\begin{aligned}
&B_{2000}^{(1)}=\frac{1}{2}\psi_{n_{1}}^{T}\left(\alpha^{*}\frac{n_{1}^{2}}{\ell^{2}}A_{20}^{(d,1)}\alpha_{1}-\alpha^{*}\frac{n_{1}^{2}}{\ell^{2}}A_{20}^{(d,1)}\alpha_{4}+A_{20}\alpha_{4}\right), \\
&B_{1100}^{(1)}=\frac{1}{2}\psi_{n_{1}}^{T}\left(\alpha^{*}\frac{n_{1}n_{2}}{\ell^{2}}A_{11}^{(d,1)}\alpha_{2}+\alpha^{*}\frac{n_{1}n_{2}}{\ell^{2}}A_{11}^{(d,2)}\alpha_{2}-\alpha^{*}\frac{n_{2}^{2}}{\ell^{2}}A_{11}^{(d,1)}\alpha_{5}-\alpha^{*}\frac{n_{1}^{2}}{\ell^{2}}A_{11}^{(d,2)}\alpha_{5}+A_{11}\alpha_{5}\right), \\
&B_{1010}^{(1)}=\frac{1}{2}\psi_{n_{1}}^{T}\left(-\frac{n_{1}^{2}}{\ell^{2}}A_{1010}^{(d,1)}\alpha_{7}\right),~B_{1001}^{(1)}=\frac{1}{2}\psi_{n_{1}}^{T}\left(-u_{*}\frac{n_{1}^{2}}{\ell^{2}}A_{1001}^{(d,1)}\alpha_{7}\right), \\
&B_{0200}^{(1)}=\frac{1}{2}\psi_{n_{1}}^{T}\left(\alpha^{*}\frac{n_{2}^{2}}{\ell^{2}}A_{02}^{(d,1)}\alpha_{3}-\alpha^{*}\frac{n_{2}^{2}}{\ell^{2}}A_{02}^{(d,1)}\alpha_{6}+A_{02}\alpha_{6}\right), \\
&B_{0110}^{(1)}=\frac{1}{2}\psi_{n_{1}}^{T}\left(-\frac{n_{2}^{2}}{\ell^{2}}A_{0110}^{(d,1)}\alpha_{8}\right),~B_{0101}^{(1)}=\frac{1}{2}\psi_{n_{1}}^{T}\left(-u_{*}\frac{n_{2}^{2}}{\ell^{2}}A_{0101}^{(d,1)}\alpha_{8}\right), \\
&B_{2000}^{(2)}=\frac{1}{2}\psi_{n_{2}}^{T}\left(\alpha^{*}\frac{n_{1}^{2}}{\ell^{2}}A_{20}^{(d,1)}\beta_{1}-\alpha^{*}\frac{n_{1}^{2}}{\ell^{2}}A_{20}^{(d,1)}\beta_{4}+A_{20}\beta_{4}\right), \\
&B_{1100}^{(2)}=\frac{1}{2}\psi_{n_{2}}^{T}\left(\alpha^{*}\frac{n_{1}n_{2}}{\ell^{2}}A_{11}^{(d,1)}\beta_{2}+\alpha^{*}\frac{n_{1}n_{2}}{\ell^{2}}A_{11}^{(d,2)}\beta_{2}-\alpha^{*}\frac{n_{2}^{2}}{\ell^{2}}A_{11}^{(d,1)}\beta_{5}-\alpha^{*}\frac{n_{1}^{2}}{\ell^{2}}A_{11}^{(d,2)}\beta_{5}+A_{11}\beta_{5}\right), \\
&B_{1010}^{(2)}=\frac{1}{2}\psi_{n_{2}}^{T}\left(-\frac{n_{1}^{2}}{\ell^{2}}A_{1010}^{(d,1)}\beta_{7}\right),~B_{1001}^{(2)}=\frac{1}{2}\psi_{n_{2}}^{T}\left(-u_{*}\frac{n_{1}^{2}}{\ell^{2}}A_{1001}^{(d,1)}\beta_{7}\right), \\
&B_{0200}^{(2)}=\frac{1}{2}\psi_{n_{2}}^{T}\left(\alpha^{*}\frac{n_{2}^{2}}{\ell^{2}}A_{02}^{(d,1)}\beta_{3}-\alpha^{*}\frac{n_{2}^{2}}{\ell^{2}}A_{02}^{(d,1)}\beta_{6}+A_{02}\beta_{6}\right), \\
&B_{0110}^{(2)}=\frac{1}{2}\psi_{n_{2}}^{T}\left(-\frac{n_{2}^{2}}{\ell^{2}}A_{0110}^{(d,1)}\beta_{8}\right),~B_{0101}^{(2)}=\frac{1}{2}\psi_{n_{2}}^{T}\left(-u_{*}\frac{n_{2}^{2}}{\ell^{2}}A_{0101}^{(d,1)}\beta_{8}\right).
\end{aligned}\right.\end{equation*}
\par~~~
\par\noindent {\bf{Calculation of $g_{3}^{1}(z,0,\mu)$}}
\par~~~

Denote
\begin{equation}\begin{aligned}
&f_{2}^{(1,1)}(z,w):=\Psi\left(\begin{aligned}
&\left[F_{2}(\Phi z_{x}+w),\beta_{\nu}^{(1)}\right] \\
&\left[F_{2}(\Phi z_{x}+w),\beta_{\nu}^{(2)}\right]
\end{aligned}\right)_{\nu=n_{1}}^{\nu=n_{2}}, \\
&f_{2}^{(1,2)}(z,w,\mu):=\Psi\left(\begin{aligned}
&\left[F_{2}^{d}(\Phi z_{x}+w,\mu),\beta_{\nu}^{(1)}\right] \\
&\left[F_{2}^{d}(\Phi z_{x}+w,\mu),\beta_{\nu}^{(2)}\right]
\end{aligned}\right)_{\nu=n_{1}}^{\nu=n_{2}}.
\end{aligned}\end{equation}
In this subsection, we calculate the third-order term $g_{3}^{1}(z,0,\mu)$ in terms of (4.29). Notice that $\widetilde{f}_{3}^{1}(z,0,\mu)$ is the term of order 3 obtained after the changes of variables in previous step, which is determined by
\begin{equation}\begin{aligned}
\widetilde{f}_{3}^{1}(z,0,\mu)&=f_{3}^{1}(z,0,\mu)+\frac{3}{2}\left((D_{z}f_{2}^{1}(z,0,\mu))U_{2}^{1}(z,\mu)+(D_{w}f_{2}^{(1,1)}(z,0))U_{2}^{2}(z,\mu)\right. \\
&\left.+(D_{w,w_{x},w_{xx}}f_{2}^{(1,2)}(z,0,\mu))U_{2}^{(2,d)}(z,\mu)-(D_{z}U_{2}^{1}(z,\mu))g_{2}^{1}(z,0,\mu)\right),
\end{aligned}\end{equation}
where
\begin{equation}\begin{aligned}
&f_{2}^{1}(z,0,\mu)=f_{2}^{(1,1)}(z,0)+f_{2}^{(1,2)}(z,0,\mu), \\
&D_{w,w_{x},w_{xx}}f_{2}^{(1,2)}(z,0,\mu)=\left(D_{w}f_{2}^{(1,2)}(z,0,\mu),D_{w_{x}}f_{2}^{(1,2)}(z,0,\mu),D_{w_{x x}}f_{2}^{(1,2)}(z,0,\mu)\right),
\end{aligned}\end{equation}
and
\begin{equation}\begin{aligned}
&U_{2}^{1}(z,\mu)=(M_{2}^{1})^{-1}\operatorname{Proj}_{\operatorname{Im}(M_{2}^{1})}f_{2}^{1}(z,0,\mu),~U_{2}^{2}(z,\mu)=(M_{2}^{2})^{-1}f_{2}^{2}(z,0,\mu), \\
&U_{2}^{(2,d)}(z,\mu)=\operatorname{col}\left(U_{2}^{2}(z,\mu),U_{2,x}^{2}(z,\mu),U_{2,xx}^{2}(z,\mu)\right).
\end{aligned}\end{equation}
Notice that in (A.14),
\begin{equation}\left\{\begin{aligned}
&U_{2}^{2}(z,\mu)\triangleq h(z,\mu)=\sum_{n \in \mathbb{N}_{0}}h_{n}(z,\mu)\gamma_{n}(x), \\
&U_{2,x}^{2}(z,\mu)=h_{x}(z,\mu)=-\sum_{n \in \mathbb{N}_{0}}(n/\ell)h_{n}(z,\mu)\widetilde{\gamma}_{n}(x), \\
&U_{2,xx}^{2}(z,\mu)=h_{xx}(z,\mu)=-\sum_{n \in \mathbb{N}_{0}}(n/\ell)^{2}h_{n}(z,\mu)\gamma_{n}(x),
\end{aligned}\right.\end{equation}
where $\widetilde{\gamma}_{n}(x)$ is given by (A.9). From (A.12), we have
\begin{equation}\begin{aligned}
\widetilde{f}_{3}^{1}(z,0,0)&=f_{3}^{1}(z,0,0)+\frac{3}{2}\left((D_{z}f_{2}^{1}(z,0,0))U_{2}^{1}(z,0)+(D_{w}f_{2}^{(1,1)}(z,0))U_{2}^{2}(z,0)\right. \\
&\left.+(D_{w,w_{x},w_{xx}}f_{2}^{(1,2)}(z,0,0))U_{2}^{(2,d)}(z,0)-(D_{z}U_{2}^{1}(z,0))g_{2}^{1}(z,0,0)\right),
\end{aligned}\end{equation}
which implies that it is sufficient for obtaining $g_{3}^{1}(z,0,0)$ to compute
\begin{equation}\begin{aligned}
&\operatorname{Proj}_{S_{1}}f_{3}^{1}(z,0,0),~\operatorname{Proj}_{S_{1}}((D_{z}f_{2}^{1}(z,0,0))U_{2}^{1}(z,0)),~\operatorname{Proj}_{S_{1}}((D_{w}f_{2}^{(1,1)}(z,0))U_{2}^{2}(z,0)), \\
&\operatorname{Proj}_{S_{1}}((D_{w,w_{x},w_{xx}}f_{2}^{(1,2)}(z,0,0))U_{2}^{(2,d)}(z,0)),~\operatorname{Proj}_{S_{1}}((D_{z}U_{2}^{1}(z,0))g_{2}^{1}(z,0,0)).
\end{aligned}\end{equation}

In what follows, by combining with (4.31), (A.16) and (A.17), we calculate $\operatorname{Proj}_{S_{1}}\widetilde{f}_{3}^{1}(z,0,0)$ step by step.

\noindent {\bf{Step 1}}\quad The calculation of $\operatorname{Proj}_{S_{1}}f_{3}^{1}(z,0,0)$

From (4.22), we have
\begin{equation}
f_{3}^{1}(z,0,\mu)=\Psi\left(\begin{aligned}
&\left[F_{3}^{d}(\Phi z_{x},\mu)+F_{3}(\Phi z_{x}),\beta_{\nu}^{(1)}\right] \\
&\left[F_{3}^{d}(\Phi z_{x},\mu)+F_{3}(\Phi z_{x}),\beta_{\nu}^{(2)}\right]
\end{aligned}\right)_{\nu=n_{1}}^{\nu=n_{2}},
\end{equation}
and similar to (A.5), we can write $F_{3}(\Phi z_{x})$ as
\begin{equation}
F_{3}(\Phi z_{x})=\sum_{q_{1}+q_{2}=3}A_{q_{1}q_{2}}z_{1}^{q_{1}}z_{2}^{q_{2}}\gamma_{n_{1}}^{q_{1}}(x)\gamma_{n_{2}}^{q_{2}}(x).
\end{equation}
Then by combining with (A.18) and (A.19), we have
\begin{equation}\begin{aligned}
f_{3}^{1}(z,0,0)&=\Psi\left(\begin{aligned}
&\left[F_{3}(\Phi z_{x}),\beta_{\nu}^{(1)}\right] \\
&\left[F_{3}(\Phi z_{x}),\beta_{\nu}^{(2)}\right]
\end{aligned}\right)_{\nu=n_{1}}^{\nu=n_{2}} \\
&=\Psi\left(\begin{aligned}
&\sum_{q_{1}+q_{2}=3}A_{q_{1}q_{2}}z_{1}^{q_{1}}z_{2}^{q_{2}}\int_{0}^{\ell\pi}\gamma_{n_{1}}^{q_{1}+1}(x)\gamma_{n_{2}}^{q_{2}}(x)\mathrm{d}x  \\
&\sum_{q_{1}+q_{2}=3}A_{q_{1}q_{2}}z_{1}^{q_{1}}z_{2}^{q_{2}}\int_{0}^{\ell\pi}\gamma_{n_{1}}^{q_{1}}(x)\gamma_{n_{2}}^{q_{2}+1}(x)\mathrm{d}x
\end{aligned}\right).
\end{aligned}\end{equation}
Thus by combining with (4.31), (A.20) and
\begin{equation*}\begin{aligned}
&\alpha_{9}:=\int_{0}^{\ell\pi}\gamma_{n_{1}}^{4}(x)\mathrm{d}x=\frac{3}{2\ell\pi},~\alpha_{10}:=\int_{0}^{\ell\pi}\gamma_{n_{1}}^{3}(x)\gamma_{n_{2}}(x)\mathrm{d}x=\begin{cases}
\frac{1}{2\ell\pi}, & \text { for } n_{2}=3n_{1}, \\
0, & \text { for } n_{2} \neq 3n_{1},
\end{cases} \\
&\alpha_{11}:=\int_{0}^{\ell\pi}\gamma_{n_{1}}^{2}(x)\gamma_{n_{2}}^{2}(x)\mathrm{d}x=\frac{1}{\ell\pi},~\alpha_{12}:=\int_{0}^{\ell\pi}\gamma_{n_{1}}(x)\gamma_{n_{2}}^{3}(x)\mathrm{d}x=\begin{cases}
\frac{1}{2\ell\pi}, & \text { for } n_{1}=3n_{2}, \\
0, & \text { for } n_{1} \neq 3n_{2},
\end{cases} \\
&\alpha_{13}:=\int_{0}^{\ell\pi}\gamma_{n_{2}}^{4}(x)\mathrm{d}x=\frac{3}{2\ell\pi},
\end{aligned}\end{equation*}
we have
\begin{equation}
\frac{1}{3!}\operatorname{Proj}_{S_{1}}f_{3}^{1}(z,0,0)=\left(\begin{aligned}
&C_{3000}^{(1)}z_{1}^{3}+C_{2100}^{(1)}z_{1}^{2}z_{2}+C_{1200}^{(1)}z_{1}z_{2}^{2}+C_{0300}^{(1)}z_{2}^{3} \\
&C_{3000}^{(2)}z_{1}^{3}+C_{2100}^{(2)}z_{1}^{2}z_{2}+C_{1200}^{(2)}z_{1}z_{2}^{2}+C_{0300}^{(2)}z_{2}^{3}
\end{aligned}\right),
\end{equation}
where
\begin{equation*}\left\{\begin{aligned}
&C_{3000}^{(1)}=\frac{1}{6}\psi_{n_{1}}^{T}A_{30}\alpha_{9},~C_{2100}^{(1)}=\frac{1}{6}\psi_{n_{1}}^{T}A_{21}\alpha_{10},~C_{1200}^{(1)}=\frac{1}{6}\psi_{n_{1}}^{T}A_{12}\alpha_{11}, \\
&C_{0300}^{(1)}=\frac{1}{6}\psi_{n_{1}}^{T}A_{03}\alpha_{12},~C_{3000}^{(2)}=\frac{1}{6}\psi_{n_{2}}^{T}A_{30}\alpha_{10},~C_{2100}^{(2)}=\frac{1}{6}\psi_{n_{2}}^{T}A_{21}\alpha_{11}, \\
&C_{1200}^{(2)}=\frac{1}{6}\psi_{n_{2}}^{T}A_{12}\alpha_{12},~C_{0300}^{(2)}=\frac{1}{6}\psi_{n_{2}}^{T}A_{03}\alpha_{13}.
\end{aligned}\right.\end{equation*}

\noindent {\bf{Step 2}}\quad The calculation of $\operatorname{Proj}_{S_{1}}\left((D_{z}f_{2}^{1}(z,0,0))U_{2}^{1}(z,0)\right)$

From (4.22), we have
\begin{equation*}\begin{aligned}
f_{2}^{1}(z,0,0)&=\Psi\left(\begin{aligned}
&\left[F_{2}^{d}(\Phi z_{x},0)+F_{2}(\Phi z_{x}),\beta_{\nu}^{(1)}\right] \\
&\left[F_{2}^{d}(\Phi z_{x},0)+F_{2}(\Phi z_{x}),\beta_{\nu}^{(2)}\right]
\end{aligned}\right)_{\nu=n_{1}}^{\nu=n_{2}},
\end{aligned}\end{equation*}
where
\begin{equation*}\begin{aligned}
F_{2}^{d}(\Phi z_{x},0)+F_{2}(\Phi z_{x})&=F_{20}^{d}(\Phi z_{x})+F_{2}(\Phi z_{x}) \\
&=\alpha^{*}\frac{n_{1}^{2}}{\ell^{2}}A_{20}^{(d,1)}z_{1}^{2}\widetilde{\gamma}_{n_{1}}^{2}(x)+\alpha^{*}\frac{n_{1}n_{2}}{\ell^{2}}A_{11}^{(d,1)}z_{1}z_{2}\widetilde{\gamma}_{n_{1}}(x)\widetilde{\gamma}_{n_{2}}(x) \\
&+\alpha^{*}\frac{n_{1}n_{2}}{\ell^{2}}A_{11}^{(d,2)}z_{1}z_{2}\widetilde{\gamma}_{n_{1}}(x)\widetilde{\gamma}_{n_{2}}(x)+\alpha^{*}\frac{n_{2}^{2}}{\ell^{2}}A_{02}^{(d,1)} z_{2}^{2}\widetilde{\gamma}_{n_{2}}^{2}(x) \\
&-\alpha^{*}\frac{n_{1}^{2}}{\ell^{2}}A_{20}^{(d,1)}z_{1}^{2}\gamma_{n_{1}}^{2}(x)-\alpha^{*}\frac{n_{2}^{2}}{\ell^{2}}A_{11}^{(d,1)}z_{1}z_{2}\gamma_{n_{1}}(x)\gamma_{n_{2}}(x) \\
&-\alpha^{*}\frac{n_{1}^{2}}{\ell^{2}}A_{11}^{(d,2)}z_{1}z_{2}\gamma_{n_{1}}(x)\gamma_{n_{2}}(x)-\alpha^{*}\frac{n_{2}^{2}}{\ell^{2}}A_{02}^{(d,1)}z_{2}^{2}\gamma_{n_{2}}^{2}(x) \\
&+A_{20}z_{1}^{2}\gamma_{n_{1}}^{2}(x)+A_{02}z_{2}^{2}\gamma_{n_{2}}^{2}(x)+A_{11}z_{1}z_{2}\gamma_{n_{1}}(x)\gamma_{n_{2}}(x).
\end{aligned}\end{equation*}
Furthermore, notice that
\begin{equation*}\begin{aligned}
&\Psi\left(\begin{aligned}
&\left[F_{2}(\Phi z_{x}),\beta_{\nu}^{(1)}\right] \\
&\left[F_{2}(\Phi z_{x}),\beta_{\nu}^{(2)}\right]
\end{aligned}\right)_{\nu=n_{1}}^{\nu=n_{2}}=\Psi\left(\begin{array}{c}
\sum_{q_{1}+q_{2}=2}A_{q_{1}q_{2}}z_{1}^{q_{1}}z_{2}^{q_{2}}\int_{0}^{\ell\pi}\gamma_{n_{1}}^{q_{1}+1}(x)\gamma_{n_{2}}^{q_{2}}(x)dx  \\
\sum_{q_{1}+q_{2}=2}A_{q_{1}q_{2}}z_{1}^{q_{1}}z_{2}^{q_{2}}\int_{0}^{\ell\pi}\gamma_{n_{1}}^{q_{1}}(x)\gamma_{n_{2}}^{q_{2}+1}(x)dx
\end{array}\right) \\
&=\Psi\left(\begin{array}{c}
A_{11}z_{1}z_{2}\alpha_{5}+A_{02}z_{2}^{2}\alpha_{6} \\
A_{20}z_{1}^{2}\alpha_{5}+A_{11}z_{1}z_{2}\alpha_{6}
\end{array}\right), \\
&\Psi\left(\begin{aligned}
&\left[F_{20}^{d}(\Phi z_{x},0),\beta_{\nu}^{(1)}\right] \\
&\left[F_{20}^{d}(\Phi z_{x},0),\beta_{\nu}^{(2)}\right]
\end{aligned}\right)_{\nu=n_{1}}^{\nu=n_{2}} \\
&=\Psi\left(\begin{array}{c}
\left(-\alpha^{*}\frac{n_{1}^{2}}{\ell^{2}}A_{20}^{(d,1)}\alpha_{4}+\alpha^{*}\frac{n_{1}^{2}}{\ell^{2}}A_{20}^{(d,1)}\alpha_{1}\right)z_{1}^{2}+\left(-\alpha^{*}\frac{n_{2}^{2}}{\ell^{2}}A_{11}^{(d,1)}\alpha_{5}+\alpha^{*}\frac{n_{1}n_{2}}{\ell^{2}}A_{11}^{(d,1)}\alpha_{2}\right)z_{1}z_{2} \\
+\left(-\alpha^{*}\frac{n_{1}^{2}}{\ell^{2}}A_{11}^{(d,2)}\alpha_{5}+\alpha^{*}\frac{n_{1}n_{2}}{\ell^{2}}A_{11}^{(d,2)}\alpha_{2}\right)z_{1}z_{2}+\left(-\alpha^{*}\frac{n_{2}^{2}}{\ell^{2}}A_{02}^{(d,1)}\alpha_{6}+\alpha^{*}\frac{n_{2}^{2}}{\ell^{2}}A_{02}^{(d,1)}\alpha_{3}\right)z_{2}^{2} \\
\left(-\alpha^{*}\frac{n_{1}^{2}}{\ell^{2}}A_{20}^{(d,1)}\beta_{4}+\alpha^{*}\frac{n_{1}^{2}}{\ell^{2}}A_{20}^{(d,1)}\beta_{1}\right)z_{1}^{2}+\left(-\alpha^{*}\frac{n_{2}^{2}}{\ell^{2}}A_{11}^{(d,1)}\beta_{5}+\alpha^{*}\frac{n_{1}n_{2}}{\ell^{2}}A_{11}^{(d,1)}\beta_{2}\right)z_{1}z_{2} \\
+\left(-\alpha^{*}\frac{n_{1}^{2}}{\ell^{2}}A_{11}^{(d,2)}\beta_{5}+\alpha^{*}\frac{n_{1}n_{2}}{\ell^{2}}A_{11}^{(d,2)}\beta_{2}\right)z_{1}z_{2}+\left(-\alpha^{*}\frac{n_{2}^{2}}{\ell^{2}}A_{02}^{(d,1)}\beta_{6}+\alpha^{*}\frac{n_{2}^{2}}{\ell^{2}}A_{02}^{(d,1)}\beta_{3}\right)z_{2}^{2}
\end{array}\right),
\end{aligned}\end{equation*}
thus we have
\begin{equation}\begin{aligned}
&f_{2}^{1}(z,0,0) \\
&=\Psi\left(\begin{array}{c}
\left(-\alpha^{*}\frac{n_{1}^{2}}{\ell^{2}}A_{20}^{(d,1)}\alpha_{4}+\alpha^{*}\frac{n_{1}^{2}}{\ell^{2}}A_{20}^{(d,1)}\alpha_{1}\right)z_{1}^{2}+\left(-\alpha^{*}\frac{n_{2}^{2}}{\ell^{2}}A_{11}^{(d,1)}\alpha_{5}+\alpha^{*}\frac{n_{1}n_{2}}{\ell^{2}}A_{11}^{(d,1)}\alpha_{2}\right)z_{1}z_{2} \\
+\left(-\alpha^{*}\frac{n_{1}^{2}}{\ell^{2}}A_{11}^{(d,2)}\alpha_{5}+\alpha^{*}\frac{n_{1}n_{2}}{\ell^{2}}A_{11}^{(d,2)}\alpha_{2}\right)z_{1}z_{2}+\left(-\alpha^{*}\frac{n_{2}^{2}}{\ell^{2}}A_{02}^{(d,1)}\alpha_{6}+\alpha^{*}\frac{n_{2}^{2}}{\ell^{2}}A_{02}^{(d,1)}\alpha_{3}\right)z_{2}^{2} \\
+A_{11}z_{1}z_{2}\alpha_{5}+A_{02}z_{2}^{2}\alpha_{6} \\
\left(-\alpha^{*}\frac{n_{1}^{2}}{\ell^{2}}A_{20}^{(d,1)}\beta_{4}+\alpha^{*}\frac{n_{1}^{2}}{\ell^{2}}A_{20}^{(d,1)}\beta_{1}\right)z_{1}^{2}+\left(-\alpha^{*}\frac{n_{2}^{2}}{\ell^{2}}A_{11}^{(d,1)}\beta_{5}+\alpha^{*}\frac{n_{1}n_{2}}{\ell^{2}}A_{11}^{(d,1)}\beta_{2}\right)z_{1}z_{2} \\
+\left(-\alpha^{*}\frac{n_{1}^{2}}{\ell^{2}}A_{11}^{(d,2)}\beta_{5}+\alpha^{*}\frac{n_{1}n_{2}}{\ell^{2}}A_{11}^{(d,2)}\beta_{2}\right)z_{1}z_{2}+\left(-\alpha^{*}\frac{n_{2}^{2}}{\ell^{2}}A_{02}^{(d,1)}\beta_{6}+\alpha^{*}\frac{n_{2}^{2}}{\ell^{2}}A_{02}^{(d,1)}\beta_{3}\right)z_{2}^{2} \\
+A_{20}z_{1}^{2}\alpha_{5}+A_{11}z_{1}z_{2}\alpha_{6}
\end{array}\right).
\end{aligned}\end{equation}
Moreover, from (4.25), (A.22) and $M_{j}^{1}(z^{q}\mu^{l}e_{\xi})=0$ with $\xi=1,2,~j=2,3,~q=q_{1}+q_{2}=2,3,~l=l_{1}+l_{2}=2$, and by a direct calculation, we have
\begin{equation}
U_{2}^{1}(z,0)=\left(M_{2}^{1}\right)^{-1}\operatorname{Proj}_{\operatorname{Im}\left(M_{2}^{1}\right)}f_{2}^{1}(z,0,0)=(0,0)^{T}.
\end{equation}
Therefore, by combining with (4.31), (A.22) and (A.23), we have
\begin{equation}\begin{aligned}
&\frac{1}{3!}\operatorname{Proj}_{S_{1}}\left((D_{z}f_{2}^{1}(z,0,0))U_{2}^{1}(z,0)\right) \\
&=\left(\begin{array}{c}
D_{3000}^{(1)}z_{1}^{3}+D_{2100}^{(1)}z_{1}^{2}z_{2}+D_{1200}^{(1)}z_{1}z_{2}^{2}+D_{0300}^{(1)}z_{2}^{3} \\
D_{3000}^{(2)}z_{1}^{3}+D_{2100}^{(2)}z_{1}^{2}z_{2}+D_{1200}^{(2)}z_{1}z_{2}^{2}+D_{0300}^{(2)}z_{2}^{3}
\end{array}\right),
\end{aligned}\end{equation}
where
\begin{equation*}\left\{\begin{aligned}
&D_{3000}^{(1)}=0,~D_{2100}^{(1)}=0,~D_{1200}^{(1)}=0,~D_{0300}^{(1)}=0, \\
&D_{3000}^{(2)}=0,~D_{2100}^{(2)}=0,~D_{1200}^{(2)}=0,~D_{0300}^{(2)}=0.
\end{aligned}\right.\end{equation*}

\noindent {\bf{Step 3}}\quad The calculation of $\operatorname{Proj}_{S_{1}}((D_{w}f_{2}^{(1,1)}(z,0))U_{2}^{2}(z,0))$

From (A.15), we can define
\begin{equation}\begin{aligned}
&U_{2}^{2}(z,0)\triangleq h(z)=\sum_{n \in \mathbb{N}_{0}}h_{n}(z)\gamma_{n}(x), \\
&U_{2,x}^{2}(z,0)\triangleq h_{x}(z)=-\sum_{n \in \mathbb{N}_{0}}(n/\ell)h_{n}(z)\widetilde{\gamma}_{n}(x), \\
&U_{2,xx}^{2}(z,0)\triangleq h_{xx}(z)=-\sum_{n \in \mathbb{N}_{0}}(n/\ell)^{2}h_{n}(z)\gamma_{n}(x),
\end{aligned}\end{equation}
where $h_{n}(z)=\sum_{q_{1}+q_{2}=2}h_{n,q_{1}q_{2}}z_{1}^{q_{1}}z_{2}^{q_{2}}$ with $h_{n,q_{1}q_{2}}=(h_{n,q_{1}q_{2}}^{(1)},h_{n,q_{1}q_{2}}^{(2)})^{T}$. Furthermore, from (A.11), we have
\begin{equation}
f_{2}^{(1,1)}(z,w)=\Psi\left(\begin{aligned}
&\left[F_{2}(\Phi z_{x}+w),\beta_{\nu}^{(1)}\right] \\
&\left[F_{2}(\Phi z_{x}+w),\beta_{\nu}^{(2)}\right]
\end{aligned}\right)_{\nu=n_{1}}^{\nu=n_{2}}.
\end{equation}
Then by combining with (A.25) and (A.26), we have
\begin{equation}\begin{aligned}
&(D_{w}f_{2}^{(1,1)}(z,0))U_{2}^{2}(z,0) \\
&=\Psi\left(\begin{aligned}
&\left[\left.D_{w}F_{2}(\Phi z_{x}+w)\right|_{w=0}\left(\sum_{n \in \mathbb{N}_{0}}h_{n}(z)\gamma_{n}(x)\right),\beta_{\nu}^{(1)}\right] \\
&\left[\left.D_{w}F_{2}(\Phi z_{x}+w)\right|_{w=0}\left(\sum_{n \in \mathbb{N}_{0}}h_{n}(z)\gamma_{n}(x)\right),\beta_{\nu}^{(2)}\right]
\end{aligned}\right)_{\nu=n_{1}}^{\nu=n_{2}}.
\end{aligned}\end{equation}
Furthermore, from (A.4), we have
\begin{equation}
\left.D_{w}F_{2}(\Phi z_{x}+w)\right|_{w=0}\left(\sum_{n \in \mathbb{N}_{0}}h_{n}(z)\gamma_{n}(x)\right)=S_{2}\left(\Phi z_{x},\sum_{n \in \mathbb{N}_{0}}h_{n}(z)\gamma_{n}(x)\right).
\end{equation}
Moreover, notice that $\Phi z_{x}=\phi_{n_{1}}z_{1}\gamma_{n_{1}}(x)+\phi_{n_{2}}z_{2}\gamma_{n_{2}}(x)$, then we have
\begin{equation}\begin{aligned}
&\left(\begin{aligned}
\left[S_{2}\left(\Phi z_{x},\sum_{n \in \mathbb{N}_{0}}h_{n}(z)\gamma_{n}(x)\right),\beta_{\nu}^{(1)}\right] \\
\left[S_{2}\left(\Phi z_{x},\sum_{n \in \mathbb{N}_{0}}h_{n}(z)\gamma_{n}(x)\right),\beta_{\nu}^{(2)}\right]
\end{aligned}\right)_{\nu=n_{1}}^{\nu=n_{2}} \\
&=\sum_{n \in \mathbb{N}_{0}}b_{n_{1},n,\nu}S_{2}(\phi_{n_{1}}z_{1},h_{n}(z))+\sum_{n \in \mathbb{N}_{0}}b_{n_{2},n,\nu}S_{2}(\phi_{n_{2}}z_{2},h_{n}(z)),
\end{aligned}\end{equation}
where $n=0,1,2, \cdots$, and $\nu=n_{1}, n_{2}$,
\begin{equation}\begin{aligned}
&b_{n_{1},n,\nu}=\int_{0}^{\ell\pi}\gamma_{n_{1}}(x)\gamma_{n}(x)\gamma_{\nu}(x)\mathrm{d}x=\begin{cases}
\frac{1}{\sqrt{\ell\pi}}, & \operatorname { for } n=0, \nu=n_{1}, \\
\frac{1}{\sqrt{2\ell\pi}}, & \operatorname { for }\begin{cases}
n=2n_{1}, & \nu=n_{1}, \\
n=n_{1}+n_{2}\left(\operatorname { or } n=|n_{2}-n_{1}|\right), & \nu=n_{2},
\end{cases} \\
0, & \operatorname{otherwise},
\end{cases} \\
&b_{n_{2},n,\nu}=\int_{0}^{\ell\pi}\gamma_{n_{2}}(x)\gamma_{n}(x)\gamma_{\nu}(x)\mathrm{d}x=\begin{cases}
\frac{1}{\sqrt{\ell\pi}}, & \operatorname { for } n=0, \nu=n_{2}, \\
\frac{1}{\sqrt{2\ell\pi}}, & \operatorname { for }\begin{cases}
n=n_{1}+n_{2}\left(\operatorname { or } n=|n_{2}-n_{1}|\right), & \nu=n_{1}, \\
n=2n_{2}, & \nu=n_{2},
\end{cases} \\
0, & \operatorname{otherwise.}\end{cases}
\end{aligned}\end{equation}
Then by combining with (A.27), (A.28), (A.29) and (A.30), we have
\begin{equation}\begin{aligned}
&(D_{w}f_{2}^{(1,1)}(z,0))U_{2}^{2}(z,0) \\
&=\Psi\left(\begin{array}{c}
b_{n_{1},0,n_{1}}S_{2}(\phi_{n_{1}}z_{1},h_{0}(z))+b_{n_{1},2n_{1},n_{1}}S_{2}(\phi_{n_{1}}z_{1},h_{2n_{1}}(z)) \\
+b_{n_{2},n_{1}+n_{2},n_{1}}S_{2}(\phi_{n_{2}}z_{2},h_{n_{1}+n_{2}}(z))+b_{n_{2},|n_{2}-n_{1}|,n_{1}}S_{2}(\phi_{n_{2}}z_{2},h_{|n_{2}-n_{1}|}(z)) \\
b_{n_{1},n_{1}+n_{2},n_{2}}S_{2}(\phi_{n_{1}}z_{1},h_{n_{1}+n_{2}}(z))+b_{n_{1},|n_{2}-n_{1}|,n_{2}}S_{2}(\phi_{n_{1}}z_{1},h_{|n_{2}-n_{1}|}(z)) \\
+b_{n_{2},0,n_{2}}S_{2}(\phi_{n_{2}}z_{2},h_{0}(z))+b_{n_{2},2n_{2},n_{2}}S_{2}(\phi_{n_{2}}z_{2},h_{2n_{2}}(z))
\end{array}\right).\end{aligned}\end{equation}
Therefore, by combining with (4.31) and (A.31), we have
\begin{equation}\begin{aligned}
&\frac{1}{3!}\operatorname{Proj}_{S_{1}}((D_{w}f_{2}^{(1,1)}(z,0))U_{2}^{2}(z,0)) \\
&=\left(\begin{array}{c}
E_{3000}^{(1)}z_{1}^{3}+E_{2100}^{(1)}z_{1}^{2}z_{2}+E_{1200}^{(1)}z_{1}z_{2}^{2}+E_{0300}^{(1)}z_{2}^{3} \\
E_{3000}^{(2)}z_{1}^{3}+E_{2100}^{(2)}z_{1}^{2}z_{2}+E_{1200}^{(2)}z_{1}z_{2}^{2}+E_{0300}^{(2)}z_{2}^{3}
\end{array}\right),
\end{aligned}\end{equation}
where
\begin{equation*}\left\{\begin{aligned}
E_{3000}^{(1)}&=\frac{1}{6\sqrt{\ell\pi}}\psi_{n_{1}}^{T}S_{2}(\phi_{n_{1}},h_{0,20})+\frac{1}{6\sqrt{2\ell\pi}}\psi_{n_{1}}^{T}S_{2}(\phi_{n_{1}},h_{2n_{1},20}), \\
E_{2100}^{(1)}&=\frac{1}{6\sqrt{\ell\pi}}\psi_{n_{1}}^{T}S_{2}(\phi_{n_{1}},h_{0,11})+\frac{1}{6\sqrt{2\ell\pi}}\psi_{n_{1}}^{T}S_{2}(\phi_{n_{1}},h_{2n_{1},11}) \\
&+\frac{1}{6\sqrt{2\ell\pi}}\psi_{n_{1}}^{T}S_{2}(\phi_{n_{2}},h_{n_{1}+n_{2},20})+\frac{1}{6\sqrt{2\ell\pi}}\psi_{n_{1}}^{T}S_{2}(\phi_{n_{2}},h_{|n_{2}-n_{1}|,20}), \\
E_{1200}^{(1)}&=\frac{1}{6\sqrt{\ell\pi}}\psi_{n_{1}}^{T}S_{2}(\phi_{n_{1}},h_{0,02})+\frac{1}{6\sqrt{2\ell\pi}}\psi_{n_{1}}^{T}S_{2}(\phi_{n_{1}},h_{2n_{1},02}) \\
&+\frac{1}{6\sqrt{2\ell\pi}}\psi_{n_{1}}^{T}S_{2}(\phi_{n_{2}},h_{n_{1}+n_{2},11})+\frac{1}{6\sqrt{2\ell\pi}}\psi_{n_{1}}^{T}S_{2}(\phi_{n_{2}},h_{|n_{2}-n_{1}|,11}), \\
E_{0300}^{(1)}&=\frac{1}{6\sqrt{2\ell\pi}}\psi_{n_{1}}^{T}S_{2}(\phi_{n_{2}},h_{n_{1}+n_{2},02})+\frac{1}{6\sqrt{2\ell\pi}}\psi_{n_{1}}^{T}S_{2}(\phi_{n_{2}},h_{|n_{2}-n_{1}|,02}), \\
E_{3000}^{(2)}&=\frac{1}{6\sqrt{2\ell\pi}}\psi_{n_{2}}^{T}S_{2}(\phi_{n_{1}},h_{n_{1}+n_{2},20})+\frac{1}{6\sqrt{2\ell\pi}}\psi_{n_{2}}^{T}S_{2}(\phi_{n_{1}},h_{|n_{2}-n_{1}|,20}), \\
E_{2100}^{(2)}&=\frac{1}{6\sqrt{2\ell\pi}}\psi_{n_{2}}^{T}S_{2}(\phi_{n_{1}},h_{n_{1}+n_{2},11})+\frac{1}{6\sqrt{2\ell\pi}}\psi_{n_{2}}^{T}S_{2}(\phi_{n_{1}},h_{|n_{2}-n_{1}|,11}) \\
&+\frac{1}{6\sqrt{\ell\pi}}\psi_{n_{2}}^{T}S_{2}(\phi_{n_{2}},h_{0,20})+\frac{1}{6\sqrt{2\ell\pi}}\psi_{n_{2}}^{T}S_{2}(\phi_{n_{2}},h_{2n_{2},20}), \\
E_{1200}^{(2)}&=\frac{1}{6\sqrt{2\ell\pi}}\psi_{n_{2}}^{T}S_{2}(\phi_{n_{1}},h_{n_{1}+n_{2},02})+\frac{1}{6\sqrt{2\ell\pi}}\psi_{n_{2}}^{T}S_{2}(\phi_{n_{1}},h_{|n_{2}-n_{1}|,02}) \\
&+\frac{1}{6\sqrt{\ell\pi}}\psi_{n_{2}}^{T}S_{2}(\phi_{n_{2}},h_{0,11})+\frac{1}{6\sqrt{2\ell\pi}}\psi_{n_{2}}^{T}S_{2}(\phi_{n_{2}},h_{2n_{2},11}), \\
E_{0300}^{(2)}&=\frac{1}{6\sqrt{\ell\pi}}\psi_{n_{2}}^{T}S_{2}(\phi_{n_{2}},h_{0,02})+\frac{1}{6\sqrt{2\ell\pi}}\psi_{n_{2}}^{T}S_{2}(\phi_{n_{2}},h_{2n_{2},02}).
\end{aligned}\right.\end{equation*}

\noindent {\bf{Step 4}}\quad The calculation of $\operatorname{Proj}_{S_{1}}((D_{w,w_{x},w_{xx}}f_{2}^{(1,2)}(z,0,0))U_{2}^{(2,d)}(z,0))$

From (A.11), we have
\begin{equation}\begin{aligned}
f_{2}^{(1,2)}(z,w,0)&=\Psi\left(\begin{aligned}
&\left[F_{2}^{d}(\Phi z_{x}+w,0),\beta_{\nu}^{(1)}\right] \\
&\left[F_{2}^{d}(\Phi z_{x}+w,0),\beta_{\nu}^{(2)}\right]
\end{aligned}\right)_{\nu=n_{1}}^{\nu=n_{2}} \\
&=\Psi\left(\begin{aligned}
&\left[F_{20}^{d}(\Phi z_{x}+w),\beta_{\nu}^{(1)}\right] \\
&\left[F_{20}^{d}(\Phi z_{x}+w),\beta_{\nu}^{(2)}\right]
\end{aligned}\right)_{\nu=n_{1}}^{\nu=n_{2}}.
\end{aligned}\end{equation}
Furthermore, from
\begin{equation*}\begin{aligned}
F_{20}^{d}(\varphi)&=F_{20}^{d}(\Phi z_{x}+w) \\
&=\alpha^{*}\left(\begin{array}{c}
(\varphi^{(1)}_{x}+w^{(1)}_{x})(\varphi^{(2)}_{x}+w^{(2)}_{x}) \\
0
\end{array}\right) \\
&+\alpha^{*}\left(\begin{array}{c}
(\varphi^{(1)}+w^{(1)})(\varphi^{(2)}_{xx}+w^{(2)}_{xx}) \\
0
\end{array}\right),
\end{aligned}\end{equation*}
we can let
\begin{equation*}\begin{aligned}
&\widetilde{S}_{2}^{(d,1)}(\varphi,w)=\alpha^{*}\left(\begin{array}{c}
\varphi_{xx}^{(2)}w^{(1)} \\
0
\end{array}\right), \\
&\widetilde{S}_{2}^{(d,2)}(\varphi,w_{x})=\alpha^{*}\left(\begin{array}{c}
\varphi_{x}^{(1)}w_{x}^{(2)} \\
0
\end{array}\right)+\alpha^{*}\left(\begin{array}{c}
w_{x}^{(1)}\varphi_{x}^{(2)} \\
0
\end{array}\right), \\
&\widetilde{S}_{2}^{(d,3)}(\varphi,w_{xx})=\alpha^{*}\left(\begin{array}{c}
\varphi^{(1)}w_{xx}^{(2)} \\
0
\end{array}\right),
\end{aligned}\end{equation*}
and thus we have
\begin{equation}\begin{aligned}
&\left(D_{w,w_{x},w_{xx}}F_{2}^{d}(\varphi,w,w_{x},w_{xx})\right)|_{w,w_{x},w_{xx}=0}U_{2}^{(2,d)}(z,0) \\
&=\widetilde{S}_{2}^{(d,1)}(\varphi,h(z))+\widetilde{S}_{2}^{(d,2)}(\varphi,h_{x}(z))+\widetilde{S}_{2}^{(d,3)}(\varphi,h_{xx}(z)).
\end{aligned}\end{equation}
Furthermore, by noticing that $\varphi=\Phi z_{x}+w=\phi_{n_{1}}z_{1}\gamma_{n_{1}}(x)+\phi_{n_{2}}z_{2}\gamma_{n_{2}}(x)+w$, then we have
\begin{equation}\begin{aligned}
&\left(\begin{aligned}
&\left[\widetilde{S}_{2}^{(d,1)}(\varphi,h(z)),\beta_{\nu}^{(1)}\right] \\
&\left[\widetilde{S}_{2}^{(d,1)}(\varphi,h(z)),\beta_{\nu}^{(2)}\right]
\end{aligned}\right)_{\nu=n_{1}}^{\nu=n_{2}} \\
&=(n_{1}/\ell)^{2}\sum_{n \in \mathbb{N}_{0}}b_{n_{1},n,\nu}S_{2}^{(d,1)}(\phi_{n_{1}}z_{1},h_{n}(z))+(n_{2}/\ell)^{2}\sum_{n \in \mathbb{N}_{0}}b_{n_{2},n,\nu}S_{2}^{(d,1)}(\phi_{n_{2}}z_{2},h_{n}(z)), \\
&\left(\begin{aligned}
&\left[\widetilde{S}_{2}^{(d,2)}(\varphi,h_{x}(z)),\beta_{\nu}^{(1)}\right] \\
&\left[\widetilde{S}_{2}^{(d,2)}(\varphi,h_{x}(z)),\beta_{\nu}^{(2)}\right]
\end{aligned}\right)_{\nu=n_{1}}^{\nu=n_{2}} \\
&=-(n_{1}/\ell)\sum_{n \in \mathbb{N}_{0}}(n/\ell)b_{n_{1},n,\nu}^{s}S_{2}^{(d,2)}(\phi_{n_{1}}z_{1},h_{n}(z))-(n_{2}/\ell)\sum_{n \in \mathbb{N}_{0}}(n/\ell)b_{n_{2},n,\nu}^{s}S_{2}^{(d,2)}(\phi_{n_{2}}z_{2},h_{n}(z)), \\
&\left(\begin{aligned}
&\left[\widetilde{S}_{2}^{(d,3)}(\varphi,h_{xx}(z)),\beta_{\nu}^{(1)}\right] \\
&\left[\widetilde{S}_{2}^{(d,3)}(\varphi,h_{xx}(z)),\beta_{\nu}^{(2)}\right]
\end{aligned}\right)_{\nu=n_{1}}^{\nu=n_{2}} \\
&=\sum_{n \in \mathbb{N}_{0}}(n/\ell)^{2}b_{n_{1},n,\nu}S_{2}^{(d,3)}(\phi_{n_{1}}z_{1},h_{n}(z))+\sum_{n \in \mathbb{N}_{0}}(n/\ell)^{2}b_{n_{2},n,\nu}S_{2}^{(d,3)}(\phi_{n_{2}}z_{2},h_{n}(z)),
\end{aligned}\end{equation}
where
\begin{equation}\begin{aligned}
&b_{n_{1},n,\nu}^{s}=\int_{0}^{\ell\pi}\widetilde{\gamma}_{n_{1}}(x)\widetilde{\gamma}_{n}(x)\gamma_{\nu}(x)\mathrm{d}x=\begin{cases}
\frac{1}{\sqrt{\ell\pi}}, & n=n_{1},~\nu=0, \\
-\frac{1}{\sqrt{2\ell\pi}}, & n=n_{1},~\nu=2n_{1}, \\
\frac{1}{\sqrt{2\ell\pi}}, & n=2n_{1},~\nu=n_{1}, \\
0, & \text{otherwise},
\end{cases} \\
&b_{n_{2},n,\nu}^{s}=\int_{0}^{\ell\pi}\widetilde{\gamma}_{n_{2}}(x)\widetilde{\gamma}_{n}(x)\gamma_{\nu}(x)\mathrm{d}x=\begin{cases}
\frac{1}{\sqrt{\ell\pi}}, & n=n_{2},~\nu=0, \\
-\frac{1}{\sqrt{2\ell\pi}}, & n=n_{2},~\nu=2n_{2}, \\
\frac{1}{\sqrt{2\ell\pi}}, & n=2n_{2},~\nu=n_{2}, \\
0, & \text{otherwise},
\end{cases}
\end{aligned}\end{equation}
and for $\phi=\left(\phi^{(1)},\phi^{(2)}\right)^{T}$, $y=\left(y^{(1)},y^{(2)}\right)^{T} \in C\left([-1,0],\mathbb{R}^{2}\right)$, we have
\begin{equation}\begin{aligned}
&S_{2}^{(d,1)}(\phi,y)=-\alpha^{*}\left(\begin{array}{c}
\phi^{(2)}y^{(1)} \\
0
\end{array}\right), \\
&S_{2}^{(d,2)}(\phi,y)=-\alpha^{*}\left(\begin{array}{c}
\phi^{(1)}y^{(2)} \\
0
\end{array}\right)-\alpha^{*}\left(\begin{array}{c}
\phi^{(2)}y^{(1)} \\
0
\end{array}\right), \\
&S_{2}^{(d,3)}(\phi,y)=-\alpha^{*}\left(\begin{array}{c}
\phi^{(1)}y^{(2)} \\
0
\end{array}\right).
\end{aligned}\end{equation}
Therefore, by combining with (4.31), (A.13), (A.33), (A.34), (A.35), (A.36) and (A.37), we have
\begin{equation}\begin{aligned}
&\frac{1}{3!}\operatorname{Proj}_{S_{1}}((D_{w,w_{x},w_{xx}}f_{2}^{(1,2)}(z,0,0))U_{2}^{(2,d)}(z,0)) \\
&=\left(\begin{array}{c}
F_{3000}^{(1)}z_{1}^{3}+F_{2100}^{(1)}z_{1}^{2}z_{2}+F_{1200}^{(1)}z_{1}z_{2}^{2}+F_{0300}^{(1)}z_{2}^{3} \\
F_{3000}^{(2)}z_{1}^{3}+F_{2100}^{(2)}z_{1}^{2}z_{2}+F_{1200}^{(2)}z_{1}z_{2}^{2}+F_{0300}^{(2)}z_{2}^{3}
\end{array}\right),
\end{aligned}\end{equation}
where
\begin{equation*}\left\{\begin{aligned}
F_{3000}^{(1)}&=\frac{n_{1}^{2}}{6\ell^{2}}\psi_{n_{1}}^{T}\left(\frac{1}{\sqrt{\ell\pi}}S_{2}^{(d,1)}(\phi_{n_{1}},h_{0,20})+\frac{1}{\sqrt{2\ell\pi}}S_{2}^{(d,1)}(\phi_{n_{1}},h_{2n_{1},20})\right) \\
&-\frac{n_{1}^{2}}{3\ell^{2}}\frac{1}{\sqrt{2\ell\pi}}\psi_{n_{1}}^{T}S_{2}^{(d,2)}(\phi_{n_{1}},h_{2n_{1},20})+\frac{(2n_{1})^{2}}{6\ell^{2}}\frac{1}{\sqrt{2\ell\pi}}\psi_{n_{1}}^{T}S_{2}^{(d,3)}(\phi_{n_{1}},h_{2n_{1},20}), \\
F_{2100}^{(1)}&=\frac{n_{1}^{2}}{6\ell^{2}}\psi_{n_{1}}^{T}\left(\frac{1}{\sqrt{\ell\pi}}S_{2}^{(d,1)}(\phi_{n_{1}},h_{0,11})+\frac{1}{\sqrt{2\ell\pi}}S_{2}^{(d,1)}(\phi_{n_{1}},h_{2n_{1},11})\right) \\
&+\frac{n_{2}^{2}}{6\ell^{2}}\frac{1}{\sqrt{2\ell\pi}}\psi_{n_{1}}^{T}\left(S_{2}^{(d,1)}(\phi_{n_{2}},h_{n_{1}+n_{2},20})+S_{2}^{(d,1)}(\phi_{n_{2}},h_{|n_{2}-n_{1}|,20})\right) \\
&-\frac{n_{1}^{2}}{3\ell^{2}}\frac{1}{\sqrt{2\ell\pi}}\psi_{n_{1}}^{T}S_{2}^{(d,2)}(\phi_{n_{1}},h_{2n_{1},11})+\frac{(2n_{1})^{2}}{6\ell^{2}}\frac{1}{\sqrt{2\ell\pi}}\psi_{n_{1}}^{T}S_{2}^{(d,3)}(\phi_{n_{1}},h_{2n_{1},11}) \\
&+\frac{(n_{1}+n_{2})^{2}}{6\ell^{2}}\frac{1}{\sqrt{2\ell\pi}}\psi_{n_{1}}^{T}S_{2}^{(d,3)}(\phi_{n_{2}},h_{n_{1}+n_{2},20}) \\
&+\frac{|n_{2}-n_{1}|^{2}}{6\ell^{2}}\frac{1}{\sqrt{2\ell\pi}}\psi_{n_{1}}^{T}S_{2}^{(d,3)}(\phi_{n_{2}},h_{|n_{2}-n_{1}|,20}), \\
F_{1200}^{(1)}&=\frac{n_{1}^{2}}{6\ell^{2}}\psi_{n_{1}}^{T}\left(\frac{1}{\sqrt{\ell\pi}}S_{2}^{(d,1)}(\phi_{n_{1}},h_{0,02})+\frac{1}{\sqrt{2\ell\pi}}S_{2}^{(d,1)}(\phi_{n_{1}},h_{2n_{1},02})\right) \\
&+\frac{n_{2}^{2}}{6\ell^{2}}\frac{1}{\sqrt{2\ell\pi}}\psi_{n_{1}}^{T}\left(S_{2}^{(d,1)}(\phi_{n_{2}},h_{n_{1}+n_{2},11})+S_{2}^{(d,1)}(\phi_{n_{2}},h_{|n_{2}-n_{1}|,11})\right) \\
&-\frac{n_{1}^{2}}{3\ell^{2}}\frac{1}{\sqrt{2\ell\pi}}\psi_{n_{1}}^{T}S_{2}^{(d,2)}(\phi_{n_{1}},h_{2n_{1},02})+\frac{(2n_{1})^{2}}{6\ell^{2}}\frac{1}{\sqrt{2\ell\pi}}\psi_{n_{1}}^{T}S_{2}^{(d,3)}(\phi_{n_{1}},h_{2n_{1},02}) \\
&+\frac{(n_{1}+n_{2})^{2}}{6\ell^{2}}\frac{1}{\sqrt{2\ell\pi}}\psi_{n_{1}}^{T}S_{2}^{(d,3)}(\phi_{n_{2}},h_{n_{1}+n_{2},11}) \\
&+\frac{|n_{2}-n_{1}|^{2}}{6\ell^{2}}\frac{1}{\sqrt{2\ell\pi}}\psi_{n_{1}}^{T}S_{2}^{(d,3)}(\phi_{n_{2}},h_{|n_{2}-n_{1}|,11}), \\
F_{0300}^{(1)}&=\frac{n_{2}^{2}}{6\ell^{2}}\frac{1}{\sqrt{2\ell\pi}}\psi_{n_{1}}^{T}\left(S_{2}^{(d,1)}(\phi_{n_{2}},h_{n_{1}+n_{2},02})+S_{2}^{(d,1)}(\phi_{n_{2}},h_{|n_{2}-n_{1}|,02})\right) \\
&+\frac{(n_{1}+n_{2})^{2}}{6\ell^{2}}\frac{1}{\sqrt{2\ell\pi}}\psi_{n_{1}}^{T}S_{2}^{(d,3)}(\phi_{n_{2}},h_{n_{1}+n_{2},02}) \\
&+\frac{|n_{2}-n_{1}|^{2}}{6\ell^{2}}\frac{1}{\sqrt{2\ell\pi}}\psi_{n_{1}}^{T}S_{2}^{(d,3)}(\phi_{n_{2}},h_{|n_{2}-n_{1}|,02}), \\
F_{3000}^{(2)}&=\frac{n_{1}^{2}}{6\ell^{2}}\frac{1}{\sqrt{2\ell\pi}}\psi_{n_{2}}^{T}\left(S_{2}^{(d,1)}(\phi_{n_{1}},h_{n_{1}+n_{2},20})+S_{2}^{(d,1)}(\phi_{n_{1}},h_{|n_{2}-n_{1}|,20})\right) \\
&+\frac{(n_{1}+n_{2})^{2}}{6\ell^{2}}\frac{1}{\sqrt{2\ell\pi}}\psi_{n_{2}}^{T}S_{2}^{(d,3)}(\phi_{n_{1}},h_{n_{1}+n_{2},20}) \\
&+\frac{|n_{2}-n_{1}|^{2}}{6\ell^{2}}\frac{1}{\sqrt{2\ell\pi}}\psi_{n_{2}}^{T}S_{2}^{(d,3)}(\phi_{n_{1}},h_{|n_{2}-n_{1}|,20}),
\end{aligned}\right.\end{equation*}

\begin{equation*}\left\{\begin{aligned}
F_{2100}^{(2)}&=\frac{n_{1}^{2}}{6\ell^{2}}\frac{1}{\sqrt{2\ell\pi}}\psi_{n_{2}}^{T}\left(S_{2}^{(d,1)}(\phi_{n_{1}},h_{n_{1}+n_{2},11})+S_{2}^{(d,1)}(\phi_{n_{1}},h_{|n_{2}-n_{1}|,11})\right) \\
&+\frac{n_{2}^{2}}{6\ell^{2}}\psi_{n_{2}}^{T}\left(\frac{1}{\sqrt{\ell\pi}}S_{2}^{(d,1)}(\phi_{n_{2}},h_{0,20})+\frac{1}{\sqrt{2\ell\pi}}S_{2}^{(d,1)}(\phi_{n_{2}},h_{2n_{2},20})\right) \\
&-\frac{n_{2}^{2}}{3\ell^{2}}\frac{1}{\sqrt{2\ell\pi}}\psi_{n_{2}}^{T}S_{2}^{(d,2)}(\phi_{n_{2}},h_{2n_{2},20})+\frac{(n_{1}+n_{2})^{2}}{6\ell^{2}}\frac{1}{\sqrt{2\ell\pi}}\psi_{n_{2}}^{T}S_{2}^{(d,3)}(\phi_{n_{1}},h_{n_{1}+n_{2},11}) \\
&+\frac{|n_{2}-n_{1}|^{2}}{6\ell^{2}}\frac{1}{\sqrt{2\ell\pi}}\psi_{n_{2}}^{T}S_{2}^{(d,3)}(\phi_{n_{1}},h_{|n_{2}-n_{1}|,11})+\frac{(2n_{2})^{2}}{6\ell^{2}}\frac{1}{\sqrt{2\ell\pi}}\psi_{n_{2}}^{T}S_{2}^{(d,3)}(\phi_{n_{2}},h_{2n_{2},20}), \\
F_{1200}^{(2)}&=\frac{n_{1}^{2}}{6\ell^{2}}\frac{1}{\sqrt{2\ell\pi}}\psi_{n_{2}}^{T}\left(S_{2}^{(d,1)}(\phi_{n_{1}},h_{n_{1}+n_{2},02})+S_{2}^{(d,1)}(\phi_{n_{1}},h_{|n_{2}-n_{1}|,02})\right) \\
&+\frac{n_{2}^{2}}{6\ell^{2}}\psi_{n_{2}}^{T}\left(\frac{1}{\sqrt{\ell\pi}}S_{2}^{(d,1)}(\phi_{n_{2}},h_{0,11})+\frac{1}{\sqrt{2\ell\pi}}S_{2}^{(d,1)}(\phi_{n_{2}},h_{2n_{2},11})\right) \\
&-\frac{n_{2}^{2}}{3\ell^{2}}\frac{1}{\sqrt{2\ell\pi}}\psi_{n_{2}}^{T}S_{2}^{(d,2)}(\phi_{n_{2}},h_{2n_{2},11})+\frac{(n_{1}+n_{2})^{2}}{6\ell^{2}}\frac{1}{\sqrt{2\ell\pi}}\psi_{n_{2}}^{T}S_{2}^{(d,3)}(\phi_{n_{1}},h_{n_{1}+n_{2},02}) \\
&+\frac{|n_{2}-n_{1}|^{2}}{6\ell^{2}}\frac{1}{\sqrt{2\ell\pi}}\psi_{n_{2}}^{T}S_{2}^{(d,3)}(\phi_{n_{1}},h_{|n_{2}-n_{1}|,02})+\frac{(2n_{2})^{2}}{6\ell^{2}}\frac{1}{\sqrt{2\ell\pi}}\psi_{n_{2}}^{T}S_{2}^{(d,3)}(\phi_{n_{2}},h_{2n_{2},11}), \\
F_{0300}^{(2)}&=\frac{n_{2}^{2}}{6\ell^{2}}\psi_{n_{2}}^{T}\left(\frac{1}{\sqrt{\ell\pi}}S_{2}^{(d,1)}(\phi_{n_{2}},h_{0,02})+\frac{1}{\sqrt{2\ell\pi}}S_{2}^{(d,1)}(\phi_{n_{2}},h_{2n_{2},02})\right) \\
&-\frac{n_{2}^{2}}{3\ell^{2}}\frac{1}{\sqrt{2\ell\pi}}\psi_{n_{2}}^{T}S_{2}^{(d,2)}(\phi_{n_{2}},h_{2n_{2},02})+\frac{(2n_{2})^{2}}{6\ell^{2}}\frac{1}{\sqrt{2\ell\pi}}\psi_{n_{2}}^{T}S_{2}^{(d,3)}(\phi_{n_{2}},h_{2n_{2},02}).
\end{aligned}\right.\end{equation*}

\noindent {\bf{Step 5}}\quad The calculation of $\operatorname{Proj}_{S_{1}}((D_{z}U_{2}^{1}(z,0))g_{2}^{1}(z,0,0))$

Since
\begin{equation*}
U_{2}^{1}(z,0)=\left(M_{2}^{1}\right)^{-1}\operatorname{Proj}_{\operatorname{Im}\left(M_{2}^{1}\right)}f_{2}^{1}(z,0,0)=(0,0)^{T},
\end{equation*}
then by combining with (4.31), we have
\begin{equation}\begin{aligned}
&\frac{1}{3!}\operatorname{Proj}_{S_{1}}\left((D_{z}U_{2}^{1}(z,0))g_{2}^{1}(z,0,0)\right) \\
&=\left(\begin{array}{c}
G_{3000}^{(1)}z_{1}^{3}+G_{2100}^{(1)}z_{1}^{2}z_{2}+G_{1200}^{(1)}z_{1}z_{2}^{2}+G_{0300}^{(1)}z_{2}^{3} \\
G_{3000}^{(2)}z_{1}^{3}+G_{2100}^{(2)}z_{1}^{2}z_{2}+G_{1200}^{(2)}z_{1}z_{2}^{2}+G_{0300}^{(2)}z_{2}^{3}
\end{array}\right),
\end{aligned}\end{equation}
where
\begin{equation*}\left\{\begin{aligned}
&G_{3000}^{(1)}=0,~G_{2100}^{(1)}=0,~G_{1200}^{(1)}=0,~G_{0300}^{(1)}=0, \\
&G_{3000}^{(2)}=0,~G_{2100}^{(2)}=0,~G_{1200}^{(2)}=0,~G_{0300}^{(2)}=0.
\end{aligned}\right.\end{equation*}

\section*{Appendix B: Calculations of $f_{\iota_{1}\iota_{2}}$, $A_{q_{1}q_{2}}$ and $S_{2}(\Phi z_{x},w)$}
\setcounter{equation}{0}
\renewcommand\theequation{B.\arabic{equation}}

For the model (1.1), let $F(\varphi)=\operatorname{col}(F^{(1)}(\varphi),F^{(2)}(\varphi))$ for any $\varphi=\operatorname{col}(\varphi_{1},\varphi_{2}) \in X$. Notice that here for the convenience of notation, we denote $\varphi=\operatorname{col}(\varphi_{1},\varphi_{2}) \in X$, instead of $\varphi=\operatorname{col}(\varphi^{(1)},\varphi^{(2)}) \in X$, and then we write the $m$-th Fr$\acute{e}$chet derivative $F_{m}(\varphi),~m \geq 2$ as
\begin{equation*}
\frac{1}{m!}F_{m}(\varphi)=\sum_{\iota_{1}+\iota_{2}=m}\frac{1}{\iota_{1}!\iota_{2}!}f_{\iota_{1}\iota_{2}}\varphi_{1}^{\iota_{1}}\varphi_{2}^{\iota_{2}},
\end{equation*}
where $f_{\iota_{1}\iota_{2}}=\operatorname{col}(f_{\iota_{1}\iota_{2}}^{(1)},f_{\iota_{1}\iota_{2}}^{(2)})$ with
\begin{equation}
f_{\iota_{1}\iota_{2}}^{(1)}=\frac{\partial^{\iota_{1}+\iota_{2}}F^{(1)}(0,0)}{\partial\varphi_{1}^{\iota_{1}}\partial\varphi_{2}^{\iota_{2}}},~f_{\iota_{1}\iota_{2}}^{(2)}=\frac{\partial^{\iota_{1}+\iota_{2}}F^{(2)}(0,0)}{\partial\varphi_{1}^{\iota_{1}}\partial\varphi_{2}^{\iota_{2}}},
\end{equation}
and
\begin{equation}
F(\varphi)=\left(\begin{aligned}
&(\varphi^{(1)}+u_{*})\left(r_{0}-a(\varphi^{(1)}+u_{*})\right)-\frac{b_{1}(1-\beta)(\varphi^{(1)}+u_{*})(\varphi^{(2)}+v_{*})}{b_{2}(\varphi^{(2)}+v_{*})+(1-\beta)(\varphi^{(1)}+u_{*})} \\
&-m_{1}(\varphi^{(2)}+v_{*})+\frac{cb_{1}(1-\beta)(\varphi^{(1)}+u_{*})(\varphi^{(2)}+v_{*})}{b_{2}(\varphi^{(2)}+v_{*})+(1-\beta)(\varphi^{(1)}+u_{*})}
\end{aligned}\right)-L\varphi.
\end{equation}
It follows from (B.1) and (B.2), we can see that $f_{\iota_{1}\iota_{2}}^{(1)}=0$ for $\iota_{2}\geq 2$, and $f_{\iota_{1}\iota_{2}}^{(2)}=0$ for $\iota_{2}\geq 3$. Moreover, we can easily verify that $f_{12}^{(1)}=0$ and $f_{12}^{(2)}=0$, then we have
\begin{equation}\begin{aligned}
F_{2}^{(1)}(\varphi)&=f_{20}^{(1)}\varphi_{1}^{2}+2f_{11}^{(1)}\varphi_{1}\varphi_{2}+f_{02}^{(1)}\varphi_{2}^{2}, \\
F_{2}^{(2)}(\varphi)&=f_{20}^{(2)}\varphi_{1}^{2}+2f_{11}^{(2)}\varphi_{1}\varphi_{2}+f_{02}^{(2)}\varphi_{2}^{2}
\end{aligned}\end{equation}
and
\begin{equation}\begin{aligned}
F_{3}^{(1)}(\varphi)&=f_{30}^{(1)}\varphi_{1}^{3}+3f_{21}^{(1)}\varphi_{1}^{2}\varphi_{2}+3f_{12}^{(1)}\varphi_{1}\varphi_{2}^{2}+f_{03}^{(1)}\varphi_{2}^{3}, \\
F_{3}^{(2)}(\varphi)&=f_{30}^{(2)}\varphi_{1}^{3}+3f_{21}^{(2)}\varphi_{1}^{2}\varphi_{2}+3f_{12}^{(2)}\varphi_{1}\varphi_{2}^{2}+f_{03}^{(2)}\varphi_{2}^{3}.
\end{aligned}\end{equation}
Moreover, by combining with (B.1), (B.2), (B.3) and (B.4), we have
\begin{equation*}\begin{aligned}
f_{20}^{(1)}&=-2a+2b_{1}(1-\beta)^{2}v_{*}(b_{2}v_{*}+(1-\beta)u_{*})^{-2}-2b_{1}(1-\beta)^{3}u_{*}v_{*}(b_{2}v_{*}+(1-\beta)u_{*})^{-3}, \\
f_{11}^{(1)}&=-b_{1}(1-\beta)(b_{2}v_{*}+(1-\beta)u_{*})^{-1}+b_{1}b_{2}(1-\beta)v_{*}(b_{2}v_{*}+(1-\beta)u_{*})^{-2} \\
&+b_{1}(1-\beta)^{2}u_{*}(b_{2}v_{*}+(1-\beta)u_{*})^{-2}-2b_{1}b_{2}(1-\beta)^{2}u_{*}v_{*}(b_{2}v_{*}+(1-\beta)u_{*})^{-3}, \\
f_{02}^{(1)}&=2b_{1}b_{2}(1-\beta)u_{*}(b_{2}v_{*}+(1-\beta)u_{*})^{-2}-2b_{1}b_{2}^{2}(1-\beta)u_{*}v_{*}(b_{2}v_{*}+(1-\beta)u_{*})^{-3}, \\
f_{20}^{(2)}&=-2cb_{1}(1-\beta)^{2}v_{*}(b_{2}v_{*}+(1-\beta)u_{*})^{-2}+2cb_{1}(1-\beta)^{3}u_{*}v_{*}(b_{2}v_{*}+(1-\beta)u_{*})^{-3}, \\
f_{11}^{(2)}&=cb_{1}(1-\beta)(b_{2}v_{*}+(1-\beta)u_{*})^{-1}-cb_{1}b_{2}(1-\beta)v_{*}(b_{2}v_{*}+(1-\beta)u_{*})^{-2} \\
&-cb_{1}(1-\beta)^{2}u_{*}(b_{2}v_{*}+(1-\beta)u_{*})^{-2}+2cb_{1}b_{2}(1-\beta)^{2}u_{*}v_{*}(b_{2}v_{*}+(1-\beta)u_{*})^{-3}, \\
f_{02}^{(2)}&=-2cb_{1}b_{2}(1-\beta)u_{*}(b_{2}v_{*}+(1-\beta)u_{*})^{-2}+2cb_{1}b_{2}^{2}(1-\beta)u_{*}v_{*}(b_{2}v_{*}+(1-\beta)u_{*})^{-3}, \\
f_{30}^{(1)}&=-6b_{1}(1-\beta)^{3}v_{*}(b_{2}v_{*}+(1-\beta)u_{*})^{-3}+6b_{1}(1-\beta)^{4}u_{*}v_{*}(b_{2}v_{*}+(1-\beta)u_{*})^{-4}, \\
f_{21}^{(1)}&=2b_{1}(1-\beta)^{2}(b_{2}v_{*}+(1-\beta)u_{*})^{-2}-4b_{1}b_{2}(1-\beta)^{2}v_{*}(b_{2}v_{*}+(1-\beta)u_{*})^{-3} \\
&-2b_{1}(1-\beta)^{3}u_{*}(b_{2}v_{*}+(1-\beta)u_{*})^{-3}+6b_{1}b_{2}(1-\beta)^{3}u_{*}v_{*}(b_{2}v_{*}+(1-\beta)u_{*})^{-4}, \\
f_{12}^{(1)}&=2b_{1}b_{2}(1-\beta)(b_{2}v_{*}+(1-\beta)u_{*})^{-2}-2b_{1}b_{2}^{2}(1-\beta)v_{*}(b_{2}v_{*}+(1-\beta)u_{*})^{-3} \\
&-4b_{1}b_{2}(1-\beta)^{2}u_{*}(b_{2}v_{*}+(1-\beta)u_{*})^{-3}+6b_{1}b_{2}^{2}(1-\beta)^{2}u_{*}v_{*}(b_{2}v_{*}+(1-\beta)u_{*})^{-4}, \\
f_{03}^{(1)}&=-6b_{1}b_{2}^{2}(1-\beta)u_{*}(b_{2}v_{*}+(1-\beta)u_{*})^{-3}+6b_{1}b_{2}^{3}(1-\beta)u_{*}v_{*}(b_{2}v_{*}+(1-\beta)u_{*})^{-4}, \\
f_{30}^{(2)}&=6cb_{1}(1-\beta)^{3}v_{*}(b_{2}v_{*}+(1-\beta)u_{*})^{-3}-6cb_{1}(1-\beta)^{4}u_{*}v_{*}(b_{2}v_{*}+(1-\beta)u_{*})^{-4}, \\
f_{21}^{(2)}&=-2cb_{1}(1-\beta)^{2}(b_{2}v_{*}+(1-\beta)u_{*})^{-2}+4cb_{1}b_{2}(1-\beta)^{2}v_{*}(b_{2}v_{*}+(1-\beta)u_{*})^{-3} \\
&+2cb_{1}(1-\beta)^{3}u_{*}(b_{2}v_{*}+(1-\beta)u_{*})^{-3}-6cb_{1}b_{2}(1-\beta)^{3}u_{*}v_{*}(b_{2}v_{*}+(1-\beta)u_{*})^{-4}, \\
f_{12}^{(2)}&=-2cb_{1}b_{2}(1-\beta)(b_{2}v_{*}+(1-\beta)u_{*})^{-2}+2cb_{1}b_{2}^{2}(1-\beta)v_{*}(b_{2}v_{*}+(1-\beta)u_{*})^{-3} \\
&+4cb_{1}b_{2}(1-\beta)^{2}u_{*}(b_{2}v_{*}+(1-\beta)u_{*})^{-3}-6cb_{1}b_{2}^{2}(1-\beta)^{2}u_{*}v_{*}(b_{2}v_{*}+(1-\beta)u_{*})^{-4}, \\
f_{03}^{(2)}&=6cb_{1}b_{2}^{2}(1-\beta)u_{*}(b_{2}v_{*}+(1-\beta)u_{*})^{-3}-6cb_{1}b_{2}^{3}(1-\beta)u_{*}v_{*}(b_{2}v_{*}+(1-\beta)u_{*})^{-4}.
\end{aligned}\end{equation*}
Notice that $\varphi=\Phi z_{x}+w$, and more precisely, we have
\begin{equation}
\varphi=\Phi z_{x}+w=\left(\begin{aligned}
&\phi_{n_{1}}^{(1)}z_{1}\gamma_{n_{1}}(x)+\phi_{n_{2}}^{(1)}z_{2}\gamma_{n_{2}}(x)+w_{1} \\
&\phi_{n_{1}}^{(2)}z_{1}\gamma_{n_{1}}(x)+\phi_{n_{2}}^{(2)}z_{2}\gamma_{n_{2}}(x)+w_{2}
\end{aligned}\right).
\end{equation}
From (B.5), and by noticing that
\begin{equation}\begin{aligned}
F_{2}(\Phi z_{x})&=\sum_{q_{1}+q_{2}=2}A_{q_{1}q_{2}}z_{1}^{q_{1}}z_{2}^{q_{2}}\gamma_{n_{1}}^{q_{1}}(x)\gamma_{n_{2}}^{q_{2}}(x) \\
&=A_{20}z_{1}^{2}\gamma_{n_{1}}^{2}(x)+A_{02}z_{2}^{2}\gamma_{n_{2}}^{2}(x)+A_{11}z_{1}z_{2}\gamma_{n_{1}}(x)\gamma_{n_{2}}(x)
\end{aligned}\end{equation}
and
\begin{equation}\begin{aligned}
F_{3}(\Phi z_{x})&=\sum_{q_{1}+q_{2}=3}A_{q_{1}q_{2}}z_{1}^{q_{1}}z_{2}^{q_{2}}\gamma_{n_{1}}^{q_{1}}(x)\gamma_{n_{2}}^{q_{2}}(x) \\
&=A_{30}z_{1}^{3}\gamma_{n_{1}}^{3}(x)+A_{21}z_{1}^{2}z_{2}\gamma_{n_{1}}^{2}(x)\gamma_{n_{2}}(x)+A_{12}z_{1}z_{2}^{2}\gamma_{n_{1}}(x)\gamma_{n_{2}}^{2}(x) \\
&+A_{03}z_{2}^{3}\gamma_{n_{2}}^{3}(x),
\end{aligned}\end{equation}
then by comparing the corresponding coefficients of (B.3) and (B.6) as well as (B.4) and (B.7), respectively, we have
\begin{equation*}\begin{aligned}
A_{20}&=\left(\begin{array}{c}
f_{20}^{(1)}(\phi_{n_{1}}^{(1)})^{2}+2f_{11}^{(1)}\phi_{n_{1}}^{(1)}\phi_{n_{1}}^{(2)}+f_{02}^{(1)}(\phi_{n_{1}}^{(2)})^{2} \\
f_{20}^{(2)}(\phi_{n_{1}}^{(1)})^{2}+2f_{11}^{(2)}\phi_{n_{1}}^{(1)}\phi_{n_{1}}^{(2)}+f_{02}^{(2)}(\phi_{n_{1}}^{(2)})^{2}
\end{array}\right), \\
A_{02}&=\left(\begin{array}{c}
f_{20}^{(1)}(\phi_{n_{2}}^{(1)})^{2}+2f_{11}^{(1)}\phi_{n_{2}}^{(1)}\phi_{n_{2}}^{(2)}+f_{02}^{(1)}(\phi_{n_{2}}^{(2)})^{2} \\
f_{20}^{(2)}(\phi_{n_{2}}^{(1)})^{2}+2f_{11}^{(2)}\phi_{n_{2}}^{(1)}\phi_{n_{2}}^{(2)}+f_{02}^{(2)}(\phi_{n_{2}}^{(2)})^{2}
\end{array}\right), \\
A_{11}&=\left(\begin{array}{c}
2f_{20}^{(1)}\phi_{n_{1}}^{(1)}\phi_{n_{2}}^{(1)}+2f_{11}^{(1)}\left(\phi_{n_{1}}^{(1)}\phi_{n_{2}}^{(2)}+\phi_{n_{2}}^{(1)}\phi_{n_{1}}^{(2)}\right)+2f_{02}^{(1)}\phi_{n_{1}}^{(2)}\phi_{n_{2}}^{(2)} \\
2f_{20}^{(2)}\phi_{n_{1}}^{(1)}\phi_{n_{2}}^{(1)}+2f_{11}^{(2)}\left(\phi_{n_{1}}^{(1)}\phi_{n_{2}}^{(2)}+\phi_{n_{2}}^{(1)}\phi_{n_{1}}^{(2)}\right)+2f_{02}^{(2)}\phi_{n_{1}}^{(2)}\phi_{n_{2}}^{(2)}
\end{array}\right)
\end{aligned}\end{equation*}
and
\begin{equation*}\begin{aligned}
A_{30}&=\left(\begin{array}{c}
f_{30}^{(1)}(\phi_{n_{1}}^{(1)})^{3}+3f_{21}^{(1)}(\phi_{n_{1}}^{(1)})^{2}\phi_{n_{1}}^{(2)}+3f_{12}^{(1)}\phi_{n_{1}}^{(1)}(\phi_{n_{1}}^{(2)})^{2}+f_{03}^{(1)}(\phi_{n_{1}}^{(2)})^{3} \\
f_{30}^{(2)}(\phi_{n_{1}}^{(1)})^{3}+3f_{21}^{(2)}(\phi_{n_{1}}^{(1)})^{2}\phi_{n_{1}}^{(2)}+3f_{12}^{(2)}\phi_{n_{1}}^{(1)}(\phi_{n_{1}}^{(2)})^{2}+f_{03}^{(2)}(\phi_{n_{1}}^{(2)})^{3}
\end{array}\right), \\
A_{21}&=\left(\begin{array}{c}
3f_{30}^{(1)}(\phi_{n_{1}}^{(1)})^{2}\phi_{n_{2}}^{(1)}+3f_{21}^{(1)}\left((\phi_{n_{1}}^{(1)})^{2}\phi_{n_{2}}^{(2)}+2\phi_{n_{1}}^{(1)}\phi_{n_{2}}^{(1)}\phi_{n_{1}}^{(2)}\right)+3f_{12}^{(1)}\left(\phi_{n_{2}}^{(1)}(\phi_{n_{1}}^{(2)})^{2}+2\phi_{n_{1}}^{(2)}\phi_{n_{2}}^{(2)}\phi_{n_{1}}^{(1)}\right)  \\
+3f_{03}^{(1)}(\phi_{n_{1}}^{(2)})^{2}\phi_{n_{2}}^{(2)} \\
3f_{30}^{(2)}(\phi_{n_{1}}^{(1)})^{2}\phi_{n_{2}}^{(1)}+3f_{21}^{(2)}\left((\phi_{n_{1}}^{(1)})^{2}\phi_{n_{2}}^{(2)}+2\phi_{n_{1}}^{(1)}\phi_{n_{2}}^{(1)}\phi_{n_{1}}^{(2)}\right)+3f_{12}^{(2)}\left(\phi_{n_{2}}^{(1)}(\phi_{n_{1}}^{(2)})^{2}+2\phi_{n_{1}}^{(2)}\phi_{n_{2}}^{(2)}\phi_{n_{1}}^{(1)}\right)  \\
+3f_{03}^{(2)}(\phi_{n_{1}}^{(2)})^{2}\phi_{n_{2}}^{(2)}
\end{array}\right), \\
A_{12}&=\left(\begin{array}{c}
3f_{30}^{(1)}\phi_{n_{1}}^{(1)}(\phi_{n_{2}}^{(1)})^{2}+3f_{21}^{(1)}\left((\phi_{n_{2}}^{(1)})^{2}\phi_{n_{1}}^{(2)}+2\phi_{n_{1}}^{(1)}\phi_{n_{2}}^{(1)}\phi_{n_{2}}^{(2)}\right)+3f_{12}^{(1)}\left((\phi_{n_{2}}^{(2)})^{2}\phi_{n_{1}}^{(1)}+2\phi_{n_{1}}^{(2)}\phi_{n_{2}}^{(2)}\phi_{n_{2}}^{(1)}\right) \\
+3f_{03}^{(1)}\phi_{n_{1}}^{(2)}(\phi_{n_{2}}^{(2)})^{2} \\
3f_{30}^{(2)}\phi_{n_{1}}^{(1)}(\phi_{n_{2}}^{(1)})^{2}+3f_{21}^{(2)}\left((\phi_{n_{2}}^{(1)})^{2}\phi_{n_{1}}^{(2)}+2\phi_{n_{1}}^{(1)}\phi_{n_{2}}^{(1)}\phi_{n_{2}}^{(2)}\right)+3f_{12}^{(2)}\left((\phi_{n_{2}}^{(2)})^{2}\phi_{n_{1}}^{(1)}+2\phi_{n_{1}}^{(2)}\phi_{n_{2}}^{(2)}\phi_{n_{2}}^{(1)}\right) \\
+3f_{03}^{(2)}\phi_{n_{1}}^{(2)}(\phi_{n_{2}}^{(2)})^{2}
\end{array}\right), \\
A_{03}&=\left(\begin{array}{c}
f_{30}^{(1)}(\phi_{n_{2}}^{(1)})^{3}+3f_{21}^{(1)}(\phi_{n_{2}}^{(1)})^{2}\phi_{n_{2}}^{(2)}+3f_{12}^{(1)}\phi_{n_{2}}^{(1)}(\phi_{n_{2}}^{(2)})^{2}+f_{03}^{(1)}(\phi_{n_{2}}^{(2)})^{3} \\
f_{30}^{(2)}(\phi_{n_{2}}^{(1)})^{3}+3f_{21}^{(2)}(\phi_{n_{2}}^{(1)})^{2}\phi_{n_{2}}^{(2)}+3f_{12}^{(2)}\phi_{n_{2}}^{(1)}(\phi_{n_{2}}^{(2)})^{2}+f_{03}^{(2)}(\phi_{n_{2}}^{(2)})^{3}
\end{array}\right).
\end{aligned}\end{equation*}
Moreover, if denote
\begin{equation*}\begin{aligned}
&\breve{\varphi}_{1}:=\phi_{n_{1}}^{(1)}z_{1}\gamma_{n_{1}}(x)+\phi_{n_{2}}^{(1)}z_{2}\gamma_{n_{2}}(x), \\
&\breve{\varphi}_{2}:=\phi_{n_{1}}^{(2)}z_{1}\gamma_{n_{1}}(x)+\phi_{n_{2}}^{(2)}z_{2}\gamma_{n_{2}}(x),
\end{aligned}\end{equation*}
then by combining with (A.4), (B.3) and (B.5), we have
\begin{equation*}
S_{2}(\Phi z_{x},w)=\left(\begin{array}{c}
2f_{20}^{(1)}\breve{\varphi}_{1}w_{1}+2f_{11}^{(1)}\left(\breve{\varphi}_{1}w_{2}+\breve{\varphi}_{2}w_{1}\right)+2f_{02}^{(1)}\breve{\varphi}_{2}w_{2} \\
2f_{20}^{(2)}\breve{\varphi}_{1}w_{1}+2f_{11}^{(2)}\left(\breve{\varphi}_{1}w_{2}+\breve{\varphi}_{2}w_{1}\right)+2f_{02}^{(2)}\breve{\varphi}_{2}w_{2}
\end{array}\right).
\end{equation*}

\section*{Appendix C: Calculations of $h_{n,q_{1}q_{2}}$}
\setcounter{equation}{0}
\renewcommand\theequation{C.\arabic{equation}}

By combining with (4.16), (4.26) and (A.25), we have
\begin{equation*}
M_{2}^{2}\left(\sum_{n \in \mathbb{N}_{0}}h_{n}(z)\gamma_{n}(x)\right)=D_{z}\left(\sum_{n \in \mathbb{N}_{0}}h_{n}(z)\gamma_{n}(x)\right)Bz-\mathcal{L}_{n}\left(\sum_{n \in \mathbb{N}_{0}}h_{n}(z)\gamma_{n}(x)\right)
\end{equation*}
and
\begin{equation}
\left(\begin{aligned}
&\left[M_{2}^{2}\left(\sum_{n \in \mathbb{N}_{0}}h_{n}(z)\gamma_{n}(x)\right),\beta_{n}^{(1)}\right] \\
&\left[M_{2}^{2}\left(\sum_{n \in \mathbb{N}_{0}}h_{n}(z)\gamma_{n}(x)\right),\beta_{n}^{(2)}\right]
\end{aligned}\right)=-\sum_{n \in \mathbb{N}_{0}}\mathcal{L}_{n}(h_{n}(z)),
\end{equation}
where $\mathcal{L}_{n}(.)$ is given by
\begin{equation*}
\mathcal{L}_{n}(h_{n}(z))=-\frac{n^{2}}{\ell^{2}}D_{0}h_{n}(z)+A_{1}h_{n}(z).
\end{equation*}
It follows from (4.12) and (4.22), we have
\begin{equation}\begin{aligned}
f_{2}^{2}(z,0,0)&=\widetilde{F}_{2}(\Phi z_{x},0)-\pi\left(\widetilde{F}_{2}(\Phi z_{x},0)\right) \\
&=\widetilde{F}_{2}(\Phi z_{x},0) \\
&-\Phi_{n_{1}}\Psi_{n_{1}}\left(\begin{aligned}
&\left[\widetilde{F}_{2}(\Phi z_{x},0),\beta_{n_{1}}^{(1)}\right] \\
&\left[\widetilde{F}_{2}(\Phi z_{x},0),\beta_{n_{1}}^{(2)}\right]
\end{aligned}\right)\gamma_{n_{1}}(x) \\
&-\Phi_{n_{2}}\Psi_{n_{2}}\left(\begin{aligned}
&\left[\widetilde{F}_{2}(\Phi z_{x},0),\beta_{n_{2}}^{(1)}\right] \\
&\left[\widetilde{F}_{2}(\Phi z_{x},0),\beta_{n_{2}}^{(2)}\right]
\end{aligned}\right)\gamma_{n_{2}}(x).
\end{aligned}\end{equation}
Furthermore, by noticing that
\begin{equation*}\begin{aligned}
\widetilde{F}_{2}(\Phi z_{x},0)&=F_{2}^{d}(\Phi z_{x},0)+F_{2}(\Phi z_{x}) \\
&=F_{20}^{d}(\Phi z_{x})+F_{21}^{d}(\Phi z_{x},0)+F_{2}(\Phi z_{x}) \\
&=\alpha^{*}\frac{n_{1}^{2}}{\ell^{2}}A_{20}^{(d,1)}z_{1}^{2}\widetilde{\gamma}_{n_{1}}^{2}(x) \\
&+\alpha^{*}\frac{n_{1}n_{2}}{\ell^{2}}\left(A_{11}^{(d,1)}+A_{11}^{(d,2)}\right)z_{1}z_{2}\widetilde{\gamma}_{n_{1}}(x)\widetilde{\gamma}_{n_{2}}(x) \\
&+\alpha^{*}\frac{n_{2}^{2}}{\ell^{2}}A_{02}^{(d,1)}z_{2}^{2}\widetilde{\gamma}_{n_{2}}^{2}(x) \\
&-\alpha^{*}\frac{n_{1}^{2}}{\ell^{2}}A_{20}^{(d,1)}z_{1}^{2}\gamma_{n_{1}}^{2}(x) \\
&-\alpha^{*}\frac{n_{2}^{2}}{\ell^{2}}A_{11}^{(d,1)}z_{1}z_{2}\gamma_{n_{1}}(x)\gamma_{n_{2}}(x) \\
&-\alpha^{*}\frac{n_{1}^{2}}{\ell^{2}}A_{11}^{(d,2)}z_{1}z_{2}\gamma_{n_{1}}(x)\gamma_{n_{2}}(x) \\
&-\alpha^{*}\frac{n_{2}^{2}}{\ell^{2}}A_{02}^{(d,1)}z_{2}^{2}\gamma_{n_{2}}^{2}(x) \\
&+A_{20}z_{1}^{2}\gamma_{n_{1}}^{2}(x)+A_{02}z_{2}^{2}\gamma_{n_{2}}^{2}(x)+A_{11}z_{1}z_{2}\gamma_{n_{1}}(x)\gamma_{n_{2}}(x),
\end{aligned}\end{equation*}
then we have
\begin{equation}\begin{aligned}
&\left(\begin{aligned}
&\left[\widetilde{F}_{2}(\Phi z_{x},0),\beta_{n_{1}}^{(1)}\right] \\
&\left[\widetilde{F}_{2}(\Phi z_{x},0),\beta_{n_{1}}^{(2)}\right]
\end{aligned}\right) \\
&=\left(-\alpha^{*}\frac{n_{1}^{2}}{\ell^{2}}A_{20}^{(d,1)}\alpha_{4}+\alpha^{*}\frac{n_{1}^{2}}{\ell^{2}}A_{20}^{(d,1)}\alpha_{1}\right)z_{1}^{2}+\left(-\alpha^{*}\frac{n_{2}^{2}}{\ell^{2}}A_{11}^{(d,1)}\alpha_{5}+\alpha^{*}\frac{n_{1}n_{2}}{\ell^{2}}A_{11}^{(d,1)}\alpha_{2}\right)z_{1}z_{2} \\
&+\left(-\alpha^{*}\frac{n_{1}^{2}}{\ell^{2}}A_{11}^{(d,2)}\alpha_{5}+\alpha^{*}\frac{n_{1}n_{2}}{\ell^{2}}A_{11}^{(d,2)}\alpha_{2}\right)z_{1}z_{2}+\left(-\alpha^{*}\frac{n_{2}^{2}}{\ell^{2}}A_{02}^{(d,1)}\alpha_{6}+\alpha^{*}\frac{n_{2}^{2}}{\ell^{2}}A_{02}^{(d,1)}\alpha_{3}\right)z_{2}^{2} \\
&+A_{11}z_{1}z_{2}\alpha_{5}+A_{02}z_{2}^{2}\alpha_{6}, \\
&\left(\begin{aligned}
&\left[\widetilde{F}_{2}(\Phi z_{x},0),\beta_{n_{2}}^{(1)}\right] \\
&\left[\widetilde{F}_{2}(\Phi z_{x},0),\beta_{n_{2}}^{(2)}\right]
\end{aligned}\right) \\
&=\left(-\alpha^{*}\frac{n_{1}^{2}}{\ell^{2}}A_{20}^{(d,1)}\beta_{4}+\alpha^{*}\frac{n_{1}^{2}}{\ell^{2}}A_{20}^{(d,1)}\beta_{1}\right)z_{1}^{2}+\left(-\alpha^{*}\frac{n_{2}^{2}}{\ell^{2}}A_{11}^{(d,1)}\beta_{5}+\alpha^{*}\frac{n_{1}n_{2}}{\ell^{2}}A_{11}^{(d,1)}\beta_{2}\right)z_{1}z_{2} \\
&+\left(-\alpha^{*}\frac{n_{1}^{2}}{\ell^{2}}A_{11}^{(d,2)}\beta_{5}+\alpha^{*}\frac{n_{1}n_{2}}{\ell^{2}}A_{11}^{(d,2)}\beta_{2}\right)z_{1}z_{2}+\left(-\alpha^{*}\frac{n_{2}^{2}}{\ell^{2}}A_{02}^{(d,1)}\beta_{6}+\alpha^{*}\frac{n_{2}^{2}}{\ell^{2}}A_{02}^{(d,1)}\beta_{3}\right)z_{2}^{2} \\
&+A_{20}z_{1}^{2}\alpha_{5}+A_{11}z_{1}z_{2}\alpha_{6}.
\end{aligned}\end{equation}
Furthermore, by combining with (C.2) and (C.3), when $n_{1}\neq 2n_{2}$, $n_{2}\neq 2n_{1}$ and $n_{2}\neq 3n_{1}$, we have
\begin{equation}\begin{aligned}
&\left(\begin{aligned}
&\left[f_{2}^{2}(z,0,0),\beta_{n}^{(1)}\right] \\
&\left[f_{2}^{2}(z,0,0),\beta_{n}^{(2)}\right]
\end{aligned}\right) \\
&=\begin{cases}
\frac{1}{\sqrt{\ell\pi}}\left(A_{20}z_{1}^{2}+A_{02}z_{2}^{2}\right),~n=0, \\
-\Phi_{n_{1}}\Psi_{n_{1}}\left(-\alpha^{*}\frac{n_{1}^{2}}{\ell^{2}}A_{20}^{(d,1)}(\alpha_{4}-\alpha_{1})z_{1}^{2}-\alpha^{*}\frac{n_{2}^{2}}{\ell^{2}}A_{02}^{(d,1)}(\alpha_{6}-\alpha_{3})z_{2}^{2}\right. \\
\left.+\left(-\alpha^{*}\left(\frac{n_{2}^{2}}{\ell^{2}}A_{11}^{(d,1)}+\frac{n_{1}^{2}}{\ell^{2}}A_{11}^{(d,2)}\right)\alpha_{5}+\alpha^{*}\frac{n_{1}n_{2}}{\ell^{2}}\left(A_{11}^{(d,1)}+A_{11}^{(d,2)}\right)\alpha_{2}\right)z_{1}z_{2}+A_{11}z_{1}z_{2}\alpha_{5}+A_{02}z_{2}^{2}\alpha_{6}\right), \\
n=n_{1}, \\
-\Phi_{n_{2}}\Psi_{n_{2}}\left(-\alpha^{*}\frac{n_{1}^{2}}{\ell^{2}}A_{20}^{(d,1)}(\beta_{4}-\beta_{1})z_{1}^{2}-\alpha^{*}\frac{n_{2}^{2}}{\ell^{2}}A_{02}^{(d,1)}(\beta_{6}-\beta_{3})z_{2}^{2}\right. \\
\left.+\left(-\alpha^{*}\left(\frac{n_{2}^{2}}{\ell^{2}}A_{11}^{(d,1)}-\alpha^{*}\frac{n_{1}^{2}}{\ell^{2}}A_{11}^{(d,2)}\right)\beta_{5}+\alpha^{*}\frac{n_{1}n_{2}}{\ell^{2}}\left(A_{11}^{(d,1)}+A_{11}^{(d,2)}\right)\beta_{2}\right)z_{1}z_{2}+A_{20}z_{1}^{2}\alpha_{5}+A_{11}z_{1}z_{2}\alpha_{6}\right), \\
n=n_{2}, \\
\frac{1}{\sqrt{2\ell\pi}}\left(-2\alpha^{*}\frac{n_{1}^{2}}{\ell^{2}}A_{20}^{(d,1)}+A_{20}\right)z_{1}^{2},~n=2n_{1}, \\
\frac{1}{\sqrt{2\ell\pi}}\left(-2\alpha^{*}\frac{n_{2}^{2}}{\ell^{2}}A_{02}^{(d,1)}+A_{02}\right)z_{2}^{2},~n=2n_{2}, \\
\frac{1}{\sqrt{2\ell\pi}}\left(-\alpha^{*}\frac{n_{2}^{2}}{\ell^{2}}A_{11}^{(d,1)}-\alpha^{*}\frac{n_{1}^{2}}{\ell^{2}}A_{11}^{(d,2)}+\frac{1}{\sqrt{2\ell\pi}}A_{11}\right)z_{1}z_{2},~n=n_{1}+n_{2}, \\
\frac{1}{\sqrt{2\ell\pi}}\left(-\alpha^{*}\frac{n_{2}^{2}}{\ell^{2}}A_{11}^{(d,1)}-\alpha^{*}\frac{n_{1}^{2}}{\ell^{2}}A_{11}^{(d,2)}+A_{11}\right)z_{1}z_{2},~n=|n_{2}-n_{1}|.
\end{cases}\end{aligned}\end{equation}
Notice that $n_{1}=3$ and $n_{2}=4$ are chosen for the numerical simulations, thus we have $n_{1}\neq 2n_{2}$, $n_{2}\neq 2n_{1}$ and $n_{2}\neq 3n_{1}$. Therefore, by combining with (C.1), (C.4) and
\begin{equation*}
\left(\begin{aligned}
&\left[M_{2}^{2}\left(h_{n}(z)\gamma_{n}(x)\right),\beta_{n}^{(1)}\right] \\
&\left[M_{2}^{2}\left(h_{n}(z)\gamma_{n}(x)\right),\beta_{n}^{(2)}\right]
\end{aligned}\right)=\left(\begin{aligned}
&\left[f_{2}^{2}(z,0,0),\beta_{n}^{(1)}\right] \\
&\left[f_{2}^{2}(z,0,0),\beta_{n}^{(2)}\right]
\end{aligned}\right),
\end{equation*}
we have
\begin{equation}\begin{aligned}
&n=0~\left\{\begin{aligned}
&z_{1}^{2}: -\mathcal{L}_{0}(h_{0,20})=\frac{1}{\sqrt{\ell\pi}}A_{20}, \\
&z_{1}z_{2}: -\mathcal{L}_{0}(h_{0,11})=(0,0)^{T}, \\
&z_{2}^{2}: -\mathcal{L}_{0}(h_{0,02})=\frac{1}{\sqrt{\ell\pi}}A_{02},
\end{aligned}\right. \\
&n=2n_{1}~\left\{\begin{aligned}
&z_{1}^{2}: -\mathcal{L}_{2n_{1}}(h_{2n_{1},20})=-\frac{2}{\sqrt{2\ell\pi}}\alpha^{*}\frac{n_{1}^{2}}{\ell^{2}}A_{20}^{(d,1)}+\frac{1}{\sqrt{2\ell\pi}}A_{20}, \\
&z_{1}z_{2}: -\mathcal{L}_{2n_{1}}(h_{2n_{1},11})=(0,0)^{T}, \\
&z_{2}^{2}: -\mathcal{L}_{2n_{1}}(h_{2n_{1},02})=(0,0)^{T},
\end{aligned}\right. \\
&n=2n_{2}~\left\{\begin{aligned}
&z_{1}^{2}: -\mathcal{L}_{2n_{2}}(h_{2n_{2},20})=(0,0)^{T}, \\
&z_{1}z_{2}: -\mathcal{L}_{2n_{2}}(h_{2n_{2},11})=(0,0)^{T}, \\
&z_{2}^{2}: -\mathcal{L}_{2n_{2}}(h_{2n_{2},02})=\frac{1}{\sqrt{2\ell\pi}}\left(-2\alpha^{*}\frac{n_{2}^{2}}{\ell^{2}}A_{02}^{(d,1)}+A_{02}\right),
\end{aligned}\right. \\
&n=n_{1}+n_{2}~\left\{\begin{aligned}
&z_{1}^{2}: -\mathcal{L}_{n_{1}+n_{2}}(h_{n_{1}+n_{2},20})=(0,0)^{T}, \\
&z_{1}z_{2}: -\mathcal{L}_{n_{1}+n_{2}}(h_{n_{1}+n_{2},11})=\frac{1}{\sqrt{2\ell\pi}}\left(-\alpha^{*}\frac{n_{2}^{2}}{\ell^{2}}A_{11}^{(d,1)}-\alpha^{*}\frac{n_{1}^{2}}{\ell^{2}}A_{11}^{(d,2)}+A_{11}\right), \\
&z_{2}^{2}: -\mathcal{L}_{n_{1}+n_{2}}(h_{n_{1}+n_{2},02})=(0,0)^{T},
\end{aligned}\right. \\
&n=|n_{2}-n_{1}|~\left\{\begin{aligned}
&z_{1}^{2}: -\mathcal{L}_{|n_{2}-n_{1}|}(h_{|n_{2}-n_{1}|,20})=(0,0)^{T}, \\
&z_{1}z_{2}: -\mathcal{L}_{|n_{2}-n_{1}|}(h_{|n_{2}-n_{1}|,11})=-\frac{1}{\sqrt{2\ell\pi}}\alpha^{*}\frac{n_{2}^{2}}{\ell^{2}}A_{11}^{(d,1)}+\frac{1}{\sqrt{2\ell\pi}}\left(-\alpha^{*}\frac{n_{1}^{2}}{\ell^{2}}A_{11}^{(d,2)}+A_{11}\right), \\
&z_{2}^{2}: -\mathcal{L}_{|n_{2}-n_{1}|}(h_{|n_{2}-n_{1}|,02})=(0,0)^{T}.
\end{aligned}\right.
\end{aligned}\end{equation}
Then from (C.5), and by a direct calculation, we have
\begin{equation*}\left\{\begin{aligned}
&h_{0,20}=-\frac{1}{\sqrt{\ell\pi}}A_{1}^{-1}A_{20},~h_{0,11}=(0,0)^{T},~h_{0,02}=-\frac{1}{\sqrt{\ell\pi}}A_{1}^{-1}A_{02}, \\
&h_{2n_{1},20}=\frac{1}{\sqrt{2\ell\pi}}\left(\frac{(2n_{1})^{2}}{\ell^{2}}D_{0}-A_{1}\right)^{-1}\left(-2\alpha^{*}\frac{n_{1}^{2}}{\ell^{2}}A_{20}^{(d,1)}+A_{20}\right), \\
&h_{2n_{1},11}=(0,0)^{T},~h_{2n_{1},02}=(0,0)^{T},~h_{2n_{2},20}=(0,0)^{T},~h_{2n_{2},11}=(0,0)^{T}, \\
&h_{2n_{2},02}=\frac{1}{\sqrt{2\ell\pi}}\left(\frac{(2n_{2})^{2}}{\ell^{2}}D_{0}-A_{1}\right)^{-1}\left(-2\alpha^{*}\frac{n_{2}^{2}}{\ell^{2}}A_{02}^{(d,1)}+A_{02}\right), \\
&h_{n_{1}+n_{2},20}=(0,0)^{T}, \\
&h_{n_{1}+n_{2},11}=\frac{1}{\sqrt{2\ell\pi}}\left(\frac{(n_{1}+n_{2})^{2}}{\ell^{2}}D_{0}-A_{1}\right)^{-1}\left(-\alpha^{*}\frac{n_{2}^{2}}{\ell^{2}}A_{11}^{(d,1)}-\alpha^{*}\frac{n_{1}^{2}}{\ell^{2}}A_{11}^{(d,2)}+A_{11}\right), \\
&h_{n_{1}+n_{2},02}=(0,0)^{T},~h_{|n_{2}-n_{1}|,20}=(0,0)^{T}, \\
&h_{|n_{2}-n_{1}|,11}=\frac{1}{\sqrt{2\ell\pi}}\left(\frac{|n_{2}-n_{1}|^{2}}{\ell^{2}}D_{0}-A_{1}\right)^{-1}\left(-\alpha^{*}\frac{n_{2}^{2}}{\ell^{2}}A_{11}^{(d,1)}-\alpha^{*}\frac{n_{1}^{2}}{\ell^{2}}A_{11}^{(d,2)}+A_{11}\right), \\
&h_{|n_{2}-n_{1}|,02}=(0,0)^{T}.
\end{aligned}\right.\end{equation*}


\end{document}